\documentclass[12pt]{amsart}

\usepackage{amsfonts,amsmath,amsthm,amssymb}
\usepackage[mathscr]{eucal}

\usepackage{tikz}

\usepackage{enumitem}  
\setlist[enumerate]{font=\normalfont}

\usepackage{hyperref}

\newtheorem{appendixoneprp}{Proposition}

\newtheorem{appendixtwothm}{Theorem}

\newtheorem{appendixtwolm}{Lemma}

\theoremstyle{definition}
\newtheorem{appendixtwodf}{Definition}
\newtheorem{appendixtwoexmp}{Example}
\newtheorem{appendixtwoclm}{Claim}

\theoremstyle{remark}

\theoremstyle{plain}

\newtheorem{thm}{Theorem}[section]
\newtheorem{cor}[thm]{Corollary}
\newtheorem{lm}[thm]{Lemma}
\newtheorem{prp}[thm]{Proposition}

\theoremstyle{definition}
\newtheorem{df}[thm]{Definition}

\newtheorem{prb}{Problem}

\newtheorem{assum}[thm]{Assumption}

\theoremstyle{remark}
\newtheorem{rem}[thm]{Remark}

\theoremstyle{plain}

\theoremstyle{definition}

\numberwithin{equation}{section}

\newcommand{\set}[2]{\{#1\,:\,\text{#2}\}} 
\newcommand{\m}[1]{{\mathbf{\uppercase{#1}}}}
\DeclareMathOperator{\cg}{Cg}
\DeclareMathOperator{\sg}{Sg}
\DeclareMathOperator{\Con}{Con}

\DeclareMathOperator{\Pol}{Pol}

\newcommand{\ra}{\rightarrow}
\newcommand{\tup}[1]{\mathbf{#1}}             
\newcommand{\calV}{\mathcal V}                
\newcommand{\calR}{\mathcal R}

\newcommand{\proj}{\operatorname{pr}}

\newcommand{\Do}{D^o}
\newcommand{\Aut}{\operatorname{Aut}}
\newcommand{\stackequiv}[1]{\stackrel{#1}{\equiv}}
\newcommand{\tr}{\operatorname{tr}}

\newcommand{\HSP}{\mathsf{HSP}}
\newcommand{\HS}{\mathsf{HS}}
\newcommand{\Deltah}{\Delta_{\mathsf{h}}}

\newcommand{\Grp}[1]{\operatorname{Grp}_{\m #1}}

\newcommand{\Grpd}[1]{\operatorname{Grp}_{#1}}
\newcommand{\FGrp}[1]{{}_\bbF\hspace{-2pt}\Grp{#1}}
\newcommand{\FGrpd}[1]{{}_\bbF\hspace{-2pt}\Grpd{#1}}

\newcommand{\RGrp}{{}_{\m r}\hspace{-1pt}}

\newcommand{\embed}[1]{\chi_{#1}}
\newcommand{\ran}{\operatorname{ran}}

\newcommand{\typ}{\operatorname{typ}}

\newcommand{\baralpha}{\overline{\alpha}}
\newcommand{\barmu}{\overline{\mu}}

\newcommand{\bartheta}{\overline{\theta}}

\newcommand{\barvarphi}{\overline{\varphi}}
\newcommand{\barvarepsilon}{\overline{\varepsilon}}

\newcommand{\barC}{\overline{C}}

\newcommand{\coverdelta}{\delta^+}

\newcommand{\coveralpha}{\alpha^+}

\newcommand{\Dmon}{\varphi}     

\newcommand{\montwo}{\kappa} 

\newcommand{\TD}[1]{T_{\m #1}^{D}}

\newcommand{\bbF}{\mathbb F}
\newcommand{\bbD}{\mathbb D}

\newcommand{\Vn}[2]{V_{(#2,#1)}}

\newcommand{\Wn}[2]{W_{(#2,#1)}}

\newcommand{\sigman}[2]{\sigma_{(#2,#1)}}

\newcommand{\Fn}[3]{F_{(#2,#1)}^{#3}}

\newcommand{\zeron}[1]{0_{#1}}
\newcommand{\zeroiell}[2]{0_{(#2,#1)}}
\newcommand{\zerobar}{\overline{0}}
\newcommand{\scrF}{\mathscr{F}}
\newcommand{\scrG}{\mathscr{G}}
\newcommand{\scrH}{\mathscr{H}}
\newcommand{\scrK}{\mathscr{K}}

\newcommand{\sfd}{\mathsf{d}}
\newcommand{\sff}{\mathsf{f}}

\newcommand{\nonempty}{non-empty}
\newcommand{\Implies}{\,$\Rightarrow$\,}
\newcommand{\parenImplies}{($\Rightarrow$)}
\newcommand{\Iff}{\,$\Leftrightarrow$\,}
\newcommand{\parenImpliedby}{($\Leftarrow$)}

\newcommand{\barma}{\overline{\m a}}

\newcommand{\barf}{\overline{f}}

\newcommand{\sfL}{\mathsf{L}}
\newcommand{\End}{\operatorname{End}}
\newcommand{\Hom}{\operatorname{Hom}}

\newcommand{\restrict}[1]{{\upharpoonright}_{#1}}

\newcommand{\thcl}{C}  
\newcommand{\alcl}{E}  

\newcommand{\Bmax}{\thcl_{\mathsf{max}}}
\newcommand{\barthcl}{\overline{\thcl}}

\newcommand{\Cn}[1]{\alcl_{#1}}
\newcommand{\bEnd}{\mathbf{End}}

\newcommand{\proper}{similarity}

\begin{document}
\bibliographystyle{siam}

\title[Abelian congruences and similarity]{Abelian congruences and similarity in
varieties with a weak difference term}

\author{Ross Willard}
\address{Pure Mathematics Department, University of Waterloo, Waterloo,
ON N2L 3G1 Canada}
\email{ross.willard@uwaterloo.ca}

\urladdr{www.math.uwaterloo.ca/$\sim$rdwillar}
\thanks{The support of the Natural Sciences and Engineering Research Council (NSERC) of Canada is gratefully acknowledged.}


\date{July 24, 2026}

\begin{abstract}
This is the first of three papers motivated by the author's desire to 
understand and explain ``algebraically" 
one aspect of Dmitriy Zhuk's proof of the CSP Dichotomy Theorem. 
In this paper we study abelian congruences in varieties having a weak
difference term.  Each class of the congruence supports an abelian group structure; if the congruence is minimal, each class supports
the structure of a vector space over a division ring determined by the
congruence.  A construction due to 
J. Hagemann, C. Herrmann and R. Freese in the congruence modular setting
extends to varieties with a weak difference term, and provides a 
``universal domain" for the abelian groups or vector spaces that 
arise from the classes of the congruence within a single class of the 
centralizer of the congruence.  
The construction also supports an extension of Freese's
similarity relation (between subdirectly irreducible algebras)
from the congruence modular setting to varieties with
a weak difference term.  
\end{abstract}

\maketitle

\section{Introduction} \label{sec:intro}

Arguably the most important result in universal algebra in the last ten years is the
positive resolution of the Constraint Satisfaction Problem (CSP) Dichotomy Conjecture,
announced independently in 2017 by Andrei Bulatov \cite{bulatov2017} and
Dmitriy Zhuk \cite{zhuk2017,zhuk2020}.
One particular feature of Zhuk's proof is his analysis of 
multi-sorted
``rectangular critical" relations $R\leq_{sd} \m a_1\times \cdots \times \m a_n$ of
finite algebras $\m a_1,\ldots,\m a_n$ in certain locally finite idempotent Taylor varieties.  
Zhuk showed that, in the varieties he considers, 
such relations produce derived relations, 
which he called ``bridges," 
between certain meet-irreducible congruences $\rho_i \in \Con(\m a_i)$ 
determined by $R$.
Zhuk established a number of useful properties of his bridges and used them to 
tease 
out implicit linear equations in the CSP instances he considered.

Strikingly, Zhuk's definition of ``bridge" formally resembles some aspects of 
Hagemann and Herrmann's formulation \cite{hagemann-herrmann} from 1979 of the centrality relation
on congruences of a single algebra.
Moreover, the ``there-exists-a-bridge" relation, under suitable restrictions,
has some resemblance to the
``similarity" relation of Freese \cite{freese1983} from 1983; this latter relation is between
subdirectly irreducible algebras in a congruence modular variety, and can be viewed as an 
amplification of a certain correspondence involving the centrality relation.
As well, Zhuk's bridges arising in rectangular critical subdirect products in 
locally finite idempotent Taylor varieties
resemble the ``graphs of similarity" which Kearnes and Szendrei 
proved arise in 
rectangular critical relations of the form $R\leq \m a^n$ 
in congruence modular varieties
\cite[Theorem 2.5(6)]{kearnes-szendrei-parallel}. 

Our main goal in this paper and two companion papers \cite{bridges,critical}
is to establish precise connections between Zhuk's bridges,
centrality, and similarity, and then to apply these connections to the
study of rectangular critical relations among finite algebras in arbitrary
locally finite Taylor varieties.
In order to connect Zhuk's bridges to similarity, we need to
first generalize the  similarity relation from its original setting (congruence modular varieties)
to a broader class also containing all locally finite idempotent Taylor varieties.  
For this purpose, we focus on the class of varieties 
which have a \emph{weak difference term}.
Roughly speaking,
a ``weak difference term" is an idempotent ternary term that specializes to a Maltsev
operation whenever it is restricted to a block of an abelian congruence.
Such terms are weak versions of terms, known as ``difference terms," that exist
and play an important role in congruence modular varieties \cite{gumm-easyway}.
Weak difference terms arose organically in tame congruence theory in the 1980s, when
Hobby and McKenzie \cite{tct} showed that every locally finite Taylor has 
such a term.  
The property of a (general) variety having a weak difference term
was first explicitly identified in the 1990s in work of Lipparini \cite{lipparini1996},
Kearnes \cite{kearnes1995},
and Kearnes and Szendrei \cite{kearnes-szendrei},
who explored the extent to which centralizer properties in congruence modular varieties extend to
varieties with a weak difference term.
Recent work concerning varieties with a weak difference term is represented by
\cite{kearnes-kiss,wires,simpler,kearnes2023}.

This paper is organized around a construction, due to Freese \cite{freese1983}
(building on Hagemann and Herrmann \cite{hagemann-herrmann}) in the congruence
modular setting, which we call the \emph{difference algebra} of an abelian
congruence.   
Whenever $\theta$ is an abelian congruence of an
algebra in a variety with a weak difference term, 
every $\theta$-class naturally supports the structure of an abelian group.
Letting $(0:\theta)$ denote the largest congruence centralizing $\theta$, 
it turns out that the family $\mathscr{F}$ of
abelian groups supported on the $\theta$-classes within a fixed
$(0:\theta)$-class $E$ can be coherently
superimposed on each other, forming a single abelian group $G_E$ for which 
$\mathscr{F}$ is
a directed family of subgroups whose union is $G_E$.
Roughly speaking, the difference algebra of $\theta$ replaces
all the $\theta$-classes within each $(0:\theta)$-class $E$ with 
a single class (of a new abelian congruence $\varphi$) 
supporting the group $G_E$.
We define the difference algebra and its congruence $\varphi$,
and prove their basic properties,
in Sections~\ref{sec:update}.  We make precise the connection between
the abelian groups on $\theta$-classes and $\varphi$-classes 
in Section~\ref{sec:structure}.

In Section~\ref{sec:sim}, we show that Freese's similarity relation
generalizes nicely to varieties with a weak difference term.
In the abelian monolith case, 
Freese's original definition of similarity, as well as the reformulation
in \cite{freese-mckenzie}, does not generalize in a straightforward way.
Instead, we take as our starting point two equivalent
(in the congruence modular setting) characterizations 
of similarity due to Freese -- 
one based on the difference algebra construction,
the other a congruence lattice-theoretic formulation --
and show that both characterizations generalize to, 
and are provably equivalent in, varieties with a 
weak difference term (Theorems~\ref{thm:D(A)} and~\ref{thm:persp}).
We also give a third equivalent characterization using a variant
of Zhuk's bridges, which we call ``\proper\ bridges" (Definition~\ref{df:bridge}).

Along the way,
we expend some energy explaining a folklore result: that 
each class of an abelian \emph{minimal} congruence $\theta$
(in an algebra $\m a$ with a weak difference
term) naturally supports the structure of a left vector space over a
division ring $\bbF_\theta$.  Here $\bbF_\theta$ depends on
$\m a$ and $\theta$ but not on the particular $\theta$-class.
We provide a new, simpler construction of $\bbF_\theta$ 
based on the difference algebra construction in Section~\ref{sec:divring},
and show in
Appendix~\hyperlink{A1}{1} that it agrees with the folklore construction.
We also show that if two subdirectly irreducible algebras in a
variety with a weak difference term have abelian monoliths and are similar,
then the division rings determined by their monoliths are isomorphic.

In the companion paper \cite{bridges},
we will develop Zhuk's theory of bridges 
and describe the precise relationships between
the objects in Zhuk's theory, the centrality relation, similarity, and
\proper\ bridges as defined here.  In a third paper \cite{critical}, we will 
study rectangular critical relations in locally finite Taylor varieties, 
extending the work of Zhuk~\cite{zhuk2020} and
connecting these relations to similarity (as was done by
Kearnes and Szendrei \cite{kearnes-szendrei-parallel} in the congruence
modular setting).  In particular, we will show explicitly how each such 
relations can be understood as encoding
a family of ``parallel" linear equations over a finite field.

\medskip
\noindent\textsc{Acknowledgement.}  I thank Keith Kearnes for several helpful conversations,
and Dmitriy Zhuk for providing the inspiration.


\section{Definitions} \label{sec:defs}

We assume that the reader is familiar with the fundamentals of universal algebra as given
in \cite{burris-sanka}, \cite{alvin1} or \cite{bergman}.
Our notation generally follows that in \cite{alvin1}, \cite{bergman} and \cite{kearnes-kiss}.
We follow \cite{kearnes-kiss} and refer to a set of operation symbols with assigned
arities as a \emph{signature}.

If $\m a$ is an algebra and $\alpha,\beta \in \Con(\m a)$ with $\alpha\leq\beta$, then $\beta/\alpha$
denotes the congruence of $\m a/\alpha$ corresponding to $\beta$ via the
Correspondence Theorem (\cite[Theorem 4.12]{alvin1} or \cite[Theorem 3.6]{bergman}).
If in addition $\gamma,\delta \in \Con(\m a)$ with $\gamma\leq\delta$, then we write
$(\alpha,\beta)\nearrow (\gamma,\delta)$ and say that $(\alpha,\beta)$ is
\emph{perspective up to} $(\gamma,\delta)$, or \emph{transposes up to} $(\gamma,\delta)$,
if $\beta \wedge \gamma = \alpha$ and $\beta \vee \gamma = \delta$.  The notation
$(\gamma,\delta)\searrow (\alpha,\beta)$ means the same thing.
We write $\alpha\prec \beta$ and say that $\beta$ \emph{covers} $\alpha$, 
or is an \emph{upper cover of} $\alpha$, if $\alpha<\beta$ and there does not exist
a congruence $\gamma$ satisfying $\alpha<\gamma<\beta$.
A congruence $\alpha$ is \emph{minimal} if it covers $0$, and is
\emph{completely meet-irreducible} 
if $\alpha\ne 1$ and there exists $\coveralpha$
with $\alpha\prec\coveralpha$ and such that $\alpha<\beta \implies 
\coveralpha \leq \beta$ for all $\beta \in \Con(\m a)$.
A subset $T\subseteq A$ is a \emph{transversal} for a congruence $\alpha$
if it contains exactly one element from each $\alpha$-class. 

$[n]$ denotes $\{1,2,\ldots,n\}$.
If $\m a$ is an algebra, an $n$-ary term $t(x_1,\ldots,x_n)$ of $\m a$ is
\emph{idempotent} if $\m a$ satisfies the identity $t(x,\ldots,x)\approx x$, 
and is a \emph{Taylor term}
if it is idempotent, $n>1$, and for each $i \in [n]$, $\m a$ satisfies an identity
of the form 
\[
t(u_1,\ldots,u_n)\approx t(v_1,\ldots,v_n)
\]
where each $u_j$ and $v_k$ is the variable $x$ or $y$, and
$\{u_i,v_i\}=\{x,y\}$.
A particularly important example of a Taylor term is a \emph{weak near-unanimity term}
(WNU),
which is an $n$-ary idempotent term $w(x_1,\ldots,x_n)$ with $n>1$ which satisfies 
\[
w(y,x,x,\ldots,x) \approx w(x,y,x,\ldots,x) \approx w(x,x,y,\ldots,x) \approx \cdots \approx
w(x,\ldots,x,y).
\]
Another important example of a Taylor term is a \emph{Maltsev term}; this is a ternary
term $p(x,y,z)$ satisfying the identities
\begin{equation} \label{eq:maltsev}
p(x,x,y) \approx y \approx p(y,x,x).
\end{equation}
The identities \eqref{eq:maltsev} are called the \emph{Maltsev identities}.  Any 
ternary operation
(whether a term or not) satisfying them is called a \emph{Maltsev operation}.

Before defining ``weak difference term," we recall the ternary centralizer relation on
congruences and the notion of
abelian congruences.
Given a \nonempty\ set $A$, let $A^{2\times 2}$ denote the set of all
$2\times 2$ matrics over $A$.  If $\m a$ is an algebra, let 
$\m a^{2\times 2}$ denote the algebra with universe $A^{2\times 2}$ which is 
isomorphic to $\m a^4$ via the bijection
\[
\begin{pmatrix} a_1&a_3\\a_2&a_4\end{pmatrix} \mapsto (a_1,a_2,a_3,a_4).
\]

\begin{df} \label{df:M}
Suppose $\m a$ is an algebra and $\theta,\varphi \in \Con(\m a)$.
\begin{enumerate}
\item
$\m a(\theta,\varphi)$ is the subalgebra of $\m a^{2\times 2}$ consisting of 
the matrices whose columns belong to $\theta$ and rows belong to $\varphi$.
\item
$M(\theta,\varphi)$ is the subuniverse of $\m a(\theta,\varphi)$ generated by the
set
\[
X(\theta,\varphi):=\left\{ \begin{pmatrix} c&c\\d&d\end{pmatrix} : (c,d) \in \theta\right\}
\cup
\left\{ \begin{pmatrix} a&b\\a&b\end{pmatrix} : (a,b) \in \varphi\right\}.
\]
The matrices in $M(\theta,\varphi)$ are called $(\theta,\varphi)$-\emph{matrices}.
\end{enumerate}
\end{df}

\begin{df} \label{df:cent}
Suppose $\theta,\varphi,\delta \in \Con(\m a)$.  We say that 
$\varphi$ \emph{centralizes $\theta$ modulo $\delta$}, and write
$C(\varphi,\theta;\delta)$,
if any of the following equivalent conditions holds:
\begin{enumerate}
\item\label{dfcent:it1}
For every matrix in $M(\varphi,\theta)$, if one row is in $\delta$, then so is the other row.
\item \label{dfcent:it2}
For every matrix in $M(\theta,\varphi)$, if one column is in $\delta$, then so is the 
other column.
\item \label{dfcent:it3}
For every term $t(x_1,\ldots,x_m,y_1,\ldots,y_n)$ and all pairs $(a_i,b_i) \in \varphi$
($i\in [m]$) and $(c_j,d_j) \in \theta$ ($j \in [n]$),
\[
\mbox{if 
$t(\tup{a},\tup{c})\stackrel{\delta}{\equiv} t(\tup{a},\tup{d})$, then
$t(\tup{b},\tup{c})\stackrel{\delta}{\equiv} t(\tup{b},\tup{d})$.}
\]
\end{enumerate}
More generally, if $S,T$ are \emph{tolerances} of $\m a$ (i.e., reflexive symmetric
subuniverses of $\m a^2$),
then we can define $\m a(S,T)$, $X(S,T)$, $M(S,T)$, and the relation
$C(S,T;\delta)$ in the same way; see \cite[\S2.5]{kearnes-kiss}.
\end{df}

Let $\theta,\delta$ be congruences of an algebra $\m a$.  
The \emph{centralizer} (or \emph{annihilator}) 
\emph{of $\theta$ modulo $\delta$}, denoted
$(\delta:\theta)$, is the unique largest congruence
$\varphi$ for which $C(\varphi,\theta;\delta)$ holds.
We say that $\theta$ is \emph{abelian} if $C(\theta,\theta;0)$ holds; equivalently,
if $\theta \leq (0:\theta)$.  More generally, if $\theta,\delta \in 
\Con(\m a)$ with
$\delta\leq\theta$, then we say that $\theta$ is \emph{abelian modulo} $\delta$ if
any of the following equivalent conditions hold: (i) $\theta/\delta$ is an abelian
congruence of $\m a/\delta$; (ii) $C(\theta,\theta;\delta)$ holds; (iii) 
$\theta \leq
(\delta:\theta)$.  
The following lemma expresses that centralizers behave nicely under 
certain quotients and surjective
homomorphisms.

\begin{lm} \label{lm:quotient}
Suppose $\m a,\m b$ are algebras in a common signature.
\begin{enumerate}
\item \label{quotient:it1}
If $\delta',\delta,\theta \in \Con(\m a)$ with $\delta' \leq\delta\leq \theta$, then
$(\delta/\delta':\theta/\delta') = (\delta:\theta)/\delta'$.
\item \label{quotient:it2}
Suppose $f:\m a\ra\m b$ is a surjective homomorphism and 
$\lambda,\mu \in \Con(\m b)$ with $\lambda\leq\mu$.
Then $f^{-1}((\lambda:\mu)) = (f^{-1}(\lambda):f^{-1}(\mu))$.
\end{enumerate}
\end{lm}

\begin{proof}
\eqref{quotient:it1} can be deduced from a corresponding fact about
the relation $C(\rule{3pt}{0pt},\rule{3pt}{0pt};\rule{3pt}{0pt})$ such
as that given in \cite[Theorem 2.19(10)]{kearnes-kiss}.  To prove
\eqref{quotient:it2}, let $\delta'=\ker(f) = f^{-1}(0)$, $\delta = f^{-1}(\lambda)$,
and $\theta = f^{-1}(\mu)$.  Also let $\barf$ denote the isomorphism 
$\barf:\m a/\delta'\cong\m b$ naturally induced by $f$.
The corresponding isomorphism $\barf:\Con(\m a/\delta')\cong\Con(\m b)$
is then given explicitly by $\barf(\psi/\delta')=\varphi\iff\psi=f^{-1}(\varphi)$.
Thus 
\begin{align*}
\barf((\delta:\theta)/\delta') &= \barf((\delta/\delta':\theta/\delta')) &\mbox{by \eqref{quotient:it1}}\\
&= (\barf(\delta/\delta'):\barf(\theta/\delta')) &\mbox{as $\barf$ is an isomorphism}\\
&= (\lambda:\mu),
\end{align*}
which proves $(\delta:\theta) = f^{-1}((\lambda:\mu))$.
\end{proof}

\begin{df}
Let $\calV$ be a variety, $\m a\in \calV$, and $d(x,y,z)$ a ternary term in the
signature of $\calV$.
\begin{enumerate}
\item
$d$ is a \emph{weak difference term} for $\m a$ if $d$ is idempotent and for
every pair $\delta,\theta$ of congruences with $\delta\leq \theta$ and $\theta/\delta$
abelian, we have
\begin{equation} \label{eq:wdt}
\mbox{$d(a,a,b) \stackrel{\delta}{\equiv} b \stackrel{\delta}{\equiv} d(b,a,a)$
for all $(a,b) \in \theta$.}
\end{equation}
\item
$d$ is a \emph{weak difference term} for $\calV$ if it is a weak difference term for
every algebra in $\calV$.
\end{enumerate}
\end{df}

Note in particular that if $d$ is a weak difference term for $\m a$ and $\theta$ is
an abelian congruence,
then setting $\delta=0$ in \eqref{eq:wdt} gives that
the restriction of $d$ to any $\theta$-class is a Maltsev operation on that class.

Because the definition of ``weak difference term" is not given in terms of identities, 
it can happen that an
algebra has a weak difference term, but does not belong to any
variety having a weak difference term.  
Therefore, 
we will take care to distinguish ``having a weak difference term" from
the stronger ``belonging to a variety with a weak difference term."

At the level of varieties, the property of 
having a weak difference term can be 
characterized by the existence of idempotent terms satisfying a certain
pattern of identities, i.e., by an \emph{idempotent Maltsev condition},
as shown by Kearnes and Szendrei \cite[Theorem 4.8]{kearnes-szendrei}.
(See Theorem~\ref{thm:maltsev-cond} 
for a slight improvement.)
In the context of locally finite varieties, the situation is simpler.

\begin{thm} \label{thm:locfin}
For a locally finite variety $\calV$, the following are equivalent:
\begin{enumerate}
\item \label{locfin:it1}
$\calV$ has a Taylor term.
\item \label{locfin:it2}
$\calV$ has a WNU term.
\item \label{locfin:it3}
$\calV$ has a weak difference term.
\end{enumerate}
\end{thm}

\begin{proof}
This follows by combining \cite[Theorem 9.6]{tct},
\cite[Corollary 5.3]{taylor}, \cite[Theorem 2.2]{maroti-mckenzie},
and \cite[Theorem 4.8]{kearnes-szendrei}.
\end{proof}

Lemmas~\ref{lm:inject} and~\ref{lm:arrow} and Corollary~\ref{cor:directed}
mention \emph{difference terms}; these
are weak difference terms which moreover satisfy the identity
$d(x,x,y)\approx y$.  By an old theorem of Gumm \cite{gumm-easyway},
every congruence modular variety has a difference term.

The proof of Proposition~\ref{prp:TC2} will make use
of the \emph{two-term condition}.  
A congruence $\theta \in \Con(\m a)$ is said to satisfy 
the \emph{two-term condition} if
for all
$(\theta,\theta)$-matrices
\[
\begin{pmatrix} a_1&a_3\\a_2&a_4\end{pmatrix},
\begin{pmatrix} b_1&b_3\\b_2&b_4\end{pmatrix} \in M(\theta,\theta),
\]
if $a_i=b_i$ for $i=1,2,3$, then $a_4=b_4$.  It is easy an easy exercise to show that
if $\theta$ satisfies the two-term condition, then $\theta$ is abelian.
The converse implication is also true if $\m a$ belongs to a variety having
a weak difference term, by \cite[Theorem 3.1]{lipparini1996}
or \cite[Corollary 4.5]{kearnes-szendrei}.  (See Lemma~\ref{lm:2TC} for 
a slight improvement.)

\section{Tools} \label{sec:tools}

In this section we gather some technical results 
which will be used in later sections.
We start with results about algebras having a weak difference term.

The first result is an easily proved generalization of the
well-known fact (see e.g.\ \cite[Theorem 4.65(1)]{bergman})
that if an algebra $\m a$ has a Maltsev term, then
every reflexive subuniverse of $\m a^2$ is a congruence of $\m a$.
(A proof can be easily extracted from 
the proof of \cite[Corollary 6.22]{kearnes-kiss}.)

\begin{lm} \label{lm:maltsev}
Suppose $\m a$ is an algebra and $\rho$ is a reflexive subuniverse of $\m a^2$.
Suppose $\m a$ has a ternary 
term $d(x,y,z)$ such that for all $(a,b) \in \rho$ we have
$d(a,a,b)=b$ and $d(a,b,b)=a$.
Then $\rho\in\Con(\m a)$.
\end{lm}

\begin{cor} \label{cor:maltsev}
Suppose $\m a$ is an algebra with a weak difference term, $\theta\in \Con(\m a)$
is abelian, and $\theta$ is minimal.  Then for all $(a,b),(a',b') \in \theta$,
if $a\ne b$ then there exists a unary polynomial $f \in \Pol_1(\m a)$ 
satisfying $f(a)=a'$ and $f(b)=b'$.
\end{cor}

\begin{proof}
Let $\rho = \set{(f(a),f(b))}{$f \in \Pol_1(\m a)$}$.  Then
$\rho$ is a reflexive
subuniverse of $\m a^2$ 
satisfying $0_A\ne  \rho \subseteq \theta$.
Hence $\rho \in \Con(\m a)$ by Lemma~\ref{lm:maltsev}, and thus $\rho=\theta$
as $\theta$ is minimal.
\end{proof}

\begin{lm} \label{lm:2TC}
Suppose the algebra $\m a$ has a weak difference term and $\theta \in
\Con(\m a)$.  If $\theta$ is abelian, then $\theta$ satisfies the 
two-term condition.
\end{lm}

\begin{proof}
See Lipparini's proof of \cite[Theorem 3.1(a)]{lipparini1996}.  
His Theorem 3.1(a) assumes
that $\m a$ belongs to a variety with a weak
difference term, but in fact all
that the argument requires is that $\m a$ itself have a weak difference term.
\end{proof}

\medskip
As a consequence, a well-known property of 
difference terms in congruence modular varieties due to Gumm
\cite{gumm-easyway} partially extends to algebras with weak difference terms.

\begin{prp} \label{prp:TC2}
Suppose $\m a$ is an algebra with a weak difference term $d(x,y,z)$, and
$\theta \in \Con(\m a)$ is an abelian congruence of $\m a$.
Then for every polynomial operation $t(x_1,\ldots,x_n)$ and all
$a_i,b_i,c_i \in A$  $(1 \leq i \leq n)$ with $a_i \stackrel{\theta}{\equiv} b_i \stackrel{\theta}{\equiv} c_i$,
\[
d(t(\tup{a}),t(\tup{b}),t(\tup{c})) = 
t(d(a_1,b_1,c_1),\ldots,d(a_n,b_n,c_n)).
\]
\end{prp}

\begin{proof}
Obviously $t(\tup{a}),t(\tup{b}),t(\tup{c})$ belong to a common $\theta$-class.
Define polynomial operations $r(\tup{x},\tup{y}) =d(t(\tup{a}),t(\tup{x}),t(\tup{y}))$ and
$s(\tup{x},\tup{y}) = t(d(a_1,x_1,y_1),\ldots,d(a_n,x_n,y_n))$.
Then
\[
\begin{pmatrix}
r(\tup{a},\tup{b}) & r(\tup{a},\tup{c})\\
r(\tup{b},\tup{b}) & r(\tup{b},\tup{c})
\end{pmatrix} = \begin{pmatrix}
t(\tup{b}) & t(\tup{c})\\
t(\tup{a}) & d(t(\tup{a}),t(\tup{b}),t(\tup{c}))\end{pmatrix}  =:M_1
\]
while
\[
\begin{pmatrix}
s(\tup{a},\tup{b}) & s(\tup{a},\tup{c})\\
s(\tup{b},\tup{b}) & s(\tup{b},\tup{c})
\end{pmatrix} = \begin{pmatrix}
t(\tup{b}) & t(\tup{c})\\
t(\tup{a}) & t(d(a_1,b_1,c_1),\ldots,d(a_n,b_n,c_n))\end{pmatrix} =: M_2.
\]
Obviously $M_1,M_2$ are $(\theta,\theta)$-matrices, and since they agree
at all positions except one, they form an instance of the two-term condition
for $\theta$.  As $\theta$ is abelian and $\m a$ has a weak difference term,
$\theta$ satisfies the two-term condition by
Lemma~\ref{lm:2TC}.
Thus the lower-right entries of $M_1$ and $M_2$ are equal, as desired.
\end{proof}

The next result is an immediate consequence of the
definition of weak difference term and a classical 
theorem of Gumm \cite{gumm-maltsev} and Herrmann \cite{herrmann1979}. 

\begin{lm} \label{lm:gumm}
Suppose $\m a$ is an algebra having a weak difference term $d(x,y,z)$,
and $\theta$ is an abelian congruence of $\m a$.  Then for all $e \in A$,
the structure $(e/\theta,+,e)$ is an abelian group where $x+y
:= d(x,e,y)$ for $x,y \in e/\theta$.  Moreover, in this group we have
$-x=d(e,x,e)$ and $d(x,y,z)=x-y+z$ for all $x,y,z \in e/\theta$.
\end{lm}

\begin{proof}
Let $\thcl=e/\theta$.  Then $(\thcl,d\restrict{\thcl^3})$ is an abelian algebra with 
Maltsev operation $d\restrict{\thcl^3}$.  Now apply Gumm 
\cite[Theorem 4.7]{gumm-maltsev} or Herrmann's main theorem from
\cite{herrmann1979} (or 
see \cite[Corollary~2]{taylor1982} for a short proof).
\end{proof}

While an algebra may have many different weak difference terms, they all
agree whenever restricted to a class of an abelian congruence.

\begin{lm} \label{lm:samewdt}
Suppose $\m a$ is an algebra and $\theta$ is
an abelian congruence.  If $d(x,y,z)$ and $t(x,y,z)$ are two weak difference
terms for $\m a$, then 
for every $\theta$-class $\thcl$ we have that 
$d\restrict{\thcl^3}=t\restrict{\thcl^3}$.
\end{lm}

\begin{proof}
Again this is a well-known consequence of $(\thcl,d\restrict{\thcl^3})$ being
an abelian Maltsev operation.  For $a,b,c \in \thcl$, Proposition~\ref{prp:TC2}
yields
\[
d(t(a,b,b),t(b,b,b),t(b,b,c)) = t(d(a,b,b),d(b,b,b),d(b,b,c))
\]
which simplifies to $d(a,b,c)=t(a,b,c)$.
\end{proof}

\begin{df}
Suppose $\m a$ has a weak difference term, $\theta$ is an abelian
congruence, and $e \in A$.
We let $\Grp a(\theta,e)$ denote the abelian group with universe $e/\theta$,
zero element $e$, and group operation $x+y:=d(x,e,y)$, 
where $d$ is a weak difference term.
\end{df}

The operation $+$ in the previous definition does not depend
on the choice of the weak difference term, by Lemma~\ref{lm:samewdt}.
The notation $M(\theta,e)$ has been used in \cite{freese-mckenzie} and
elsewhere for $\Grp a(\theta,e)$.

The next result is known in the congruence modular setting
(e.g., \cite[Chapter 5]{freese-mckenzie}), and is folklore for varieties
with a weak difference term.  We sketch a proof showing that only the algebra
$\m a$ itself needs to have a weak difference term.

\begin{prp} \label{prp:affine}
Suppose $\m a$ is an algebra having a weak difference term, and
$\theta$ is an abelian congruence.  
\begin{enumerate}
\item \label{affine:it1}
Suppose $\thcl_1,\thcl_2$ are $\theta$-classes and $e_i \in \thcl_i$
for $i=1,2$.  If $r \in \Pol_1(\m a)$ satisfies $r(e_1)=e_2$, then
$r\restrict{\thcl_1}$ is a group homomorphism from $\Grp a(\theta_1,e_1)$
to $\Grp a(\theta_2,e_2)$.
\item \label{affine:it2}
Suppose $f \in \Pol_n(\m a)$ and $\thcl_1,\ldots,\thcl_n,\thcl$ are 
$\theta$-classes with $f(\thcl_1,\ldots,\thcl_n)\subseteq \thcl$.
Fix $\tup e = (e_1,\ldots,e_n) \in \thcl_1\times \cdots \times \thcl_n$ and 
$e \in \thcl$.  There exist
$r_1,\ldots,r_n \in \Pol_1(\m a)$ with $r_i(e_i)=e$ for $i\in [n]$, such that
for all $\tup a=(a_1,\ldots,a_n) \in \thcl_1\times \cdots \times \thcl_n$,
\[
f(\tup a) = \sum_{i=1}^n r_i(a_i) + f(\tup e) \quad
\mbox{computed in $\Grp a(\theta,e)$}.
\]
\end{enumerate}
\end{prp}

\begin{proof}[Proof sketch]
\eqref{affine:it1} follows from Proposition~\ref{prp:TC2}.  \eqref{affine:it2}
is proved by induction on $n$.
When $n=1$, define $r_1(x) = d(f(x),f(e_{1}),e)$, 
whose restriction to $\thcl_1$ is just $f\restrict{\thcl_{1}}(x)-f(e_{1})$ computed in $\Grp a(\theta,e)$.
When $n>1$, define $g \in \Pol_{n-1}(\m a)$ by 
\[
g(x_1,\ldots,x_{n-1}) =
f(x_1,\ldots,x_{n-1},e_{n})
\]
and inductively get $r_1,\ldots,r_{n-1} \in \Pol_1(\m a)$ with $r_t(e_t)=e$ 
for $t\in [n-1]$, such that 
$g(a_1,\ldots,a_{n-1}) = \sum_{t=1}^{n-1} r_t(a_t) + f(\tup e)$ for all $\tup a
\in \thcl_{1}\times \cdots \times \thcl_{n-1}$.
Define $r_n(x) = d(f(e_{1},\ldots,e_{{n-1}},x),f(\tup e),e)$
and note that $r_n(e_n)=e$.
Then for all $\tup a \in \thcl_{1}\times \cdots \times \thcl_{n}$,
computing in $\Grp a(\theta,e)$,
\begin{align*}
\sum_{t=1}^n r_t(a_t) + f(\tup e) &= \left(
\sum_{t=1}^{n-1} r_t(a_t) + f(\tup e)\right) + r_n(a_{n})\\
&= f(a_1,\ldots,a_{n-1},e_{n}) + (f(e_{1},\ldots,e_{n},a_n) - f(\tup e))\\
&= d(f(a_1,\ldots,a_{n-1},e_{n}),f(e_{1},\ldots,e_{n}),f(e_{1},\ldots,
e_{{n-1}},a_n)).  
\end{align*}
Using Proposition~\ref{prp:TC2}, the last expression
displayed above is equal to 
\[
f(d(a_1,e_{1},e_{1}),\ldots,d(a_{n-1},e_{{n-1}},e_{{n-1}}),d(e_{n},
e_{n},a_n))
\]
which simplifies to $f(\tup a)$.
\end{proof}

The following result 
is folklore; we include it here for completeness.

\begin{thm} \label{thm:finite-vecspace}
Suppose $\m a$ is a finite algebra with a weak difference term,
$\theta\in \Con(\m a)$ is abelian, and $0\prec \theta$.
There exists a prime $p$ such that for all $e \in A$,
$\Grp a(\theta,e)$ is an elementary abelian $p$-group.
Hence every $\theta$-class has cardinality $p^k$ for some $k\geq 0$
(where $k$ can depend on the $\theta$-class).
\end{thm}

\begin{proof}[Proof sketch]
This was proved by Freese \cite[\S4]{freese1983} in the congruence
modular case, and essentially the same proof works here.  We recast
his argument using tools from
a recent paper of Mayr and Szendrei \cite{mayr-szendrei}.  (See also
the arguments in Section~\ref{sec:divring},
especially Corollary~\ref{cor:ran-sub}, where a generalization is proved.)

First note that we can
assume that $\m a$ has no nullary operations (by replacing each nullary operation with
a constant unary operation).  
Let $\m i = \m a/\theta$, and let $\nu_\theta:\m a\ra \m i$ be the natural map.
The pair $(\m a,\nu_\theta)$ is called an \emph{algebra over $\m i$} in 
\cite[Definition 2.1]{mayr-szendrei}.  By an old construction of Novotn\'{y}
 \cite{novotny}, 
one can obtain a multisorted algebra $\m m=\mathfrak{M}(\m a,\nu_\theta)$ 
whose sort domains are the $\theta$-classes of $\m a$; then applying another old
construction of Gardner \cite{gardner}, one can convert $\m m$ to a one-sorted
algebra whose universe is the Cartesian product of the
$\theta$-classes of $\m a$.  Call this final algebra $\mathfrak{C}(\m a,\nu_\theta)$.
For precise definitions, see \cite[Propositions 2.2 and 2.5]{mayr-szendrei}.
Because $\theta$ is minimal and abelian, it follows from
\cite[Corollary 3.5 and Theorem 3.6]{mayr-szendrei} that $\mathfrak{C}(\m a,\nu_\theta)$
is simple and abelian.  And because $\m a$ has a term $d(x,y,z)$ whose restriction
to each $\theta$-class is a Maltsev operation, it follows from the definitions 
and \cite[Remark 2.8]{mayr-szendrei} that $\mathfrak{C}(\m a,\nu_\theta)$ has a Maltsev
term.  In fact, writing the universe of $\mathfrak{C}(\m a,\nu_\theta)$ as
$\thcl_1\times \cdots \times \thcl_k$ where
$\thcl_1,\ldots,\thcl_k$ is an enumeration of the $\theta$-classes,
then the Maltsev operation $m(x,y,z)$ of $\mathfrak{C}(\m a,\nu_\theta)$
is given coordinatewise by $d$.

Since $\mathfrak{C}(\m a,\nu_\theta)$ is simple, abelian, and has a Maltsev
term, it follows 
that $\mathfrak{C}(\m a,\nu_\theta)$ is polynomially equivalent to
a (finite) simple module by a theorem of Gumm \cite{gumm-maltsev}
and Herrmann \cite{herrmann1979}.
If we fix a $k$-tuple $\tup e=(e_1,\ldots,e_k) \in \thcl_1\times \cdots \times \thcl_k$,
then we can assume that $\tup e$ is the zero element of this module and
addition is given by $x+y = m(x,\tup e,y)$.  By Schur's Lemma and finiteness,
the endomorphism ring of this module is a finite field $\bbF$, 
say of cardinality $p^m$.  It follows that the abelian group 
$(\thcl_1\times \cdots \times \thcl_k,+,\tup e)$ inherits
the structure of a vector space over $\bbF$, which forces
$px=\tup e$ for all $x \in \thcl_1\times \cdots \times \thcl_k$.
Interpreting this equation coordinatewise gives that each group
$\Grp a(\theta,e_i)$ satisfies
$px=e_i$ as required.  
\end{proof}

In the context of the proof of 
Theorem~\ref{thm:finite-vecspace}, it can be further shown that the
finite field $\bbF$ acts on $C_1\times \cdots \times C_k$ in a coordinate-wise
fashion, meaning that each $\theta$-class $C_i$ can be viewed as a vector
space over $\bbF$.
%
In congruence modular varieties, Freese proved an analogous result when 
$\m a$ is 
infinite: there is a division
ring $\bbF$ determined by $\m a$ and the abelian minimal congruence $\theta$
such that every $\theta$-class inherits the structure of a left
$\bbF$-vector space.  
In section~\ref{sec:divring} we will extend Freese's result to arbitrary
algebras in varieties with a weak difference term, via a new method
for constructing the division ring and its action.

The next two lemmas will be used in Sections~\ref{sec:divring}.
The proof of the second lemma, in the congruence modular case, 
can be extracted from the proof of
\cite[Theorem 10.11]{freese-mckenzie}.

\begin{lm} \label{lm:poly}
Suppose $\m a$ is an algebra having a weak difference term,
$\mu \in \Con(\m a)$ is an abelian minimal congruence, and $S\leq \m a$ is a
subuniverse of $\m a$ which is a transversal for $\mu$.  
Then $S$ is a maximal proper subuniverse of $\m a$.
\end{lm}

\begin{proof}
Clearly $S\ne A$ as $\mu\ne 0_A$.
Let $\pi:\m a\ra \m s$ be the homomorphism given by letting $\pi(x)$ be the 
unique element $y \in S$ satisfying $(x,y) \in \mu$.
Fix $a \in A\setminus S$.
Let $\rho = \sg^{\m a^2}(\{(a,\pi(a))\} \cup 0_A)$.  $\rho$ is a reflexive
subuniverse of $\m a^2$ satisfying $0_A\ne \rho\subseteq \mu$.  Thus
by Lemma~\ref{lm:maltsev} and minimality of $\mu$, we get $\rho=\mu$.
Now let $b \in A$ be arbitrary.
Then
$(b,\pi(b))\in \mu=\rho$, so we can pick a term $t(x,\tup y)$
and a tuple $\tup c=(c_1,\ldots,c_n)$ of elements of $A$ such that
\[
(b,\pi(b)) = t^{\m a^2}((a,\pi(a)),(c_1,c_1),\ldots,(c_n,c_n))=(t(a,\tup c),t(\pi(a),\tup c)).
\]
Let $\tup u$ be the tuple in $S$ obtained by applying $\pi$ coordinatewise
to $\tup c$.  Observe that
$\pi(b) = \pi(t(a,\tup c)) = t(\pi(a),\tup u)$, so
\[
t(\underline{\pi(a)},\tup c) = t(\underline{\pi(a)},\tup u).
\]
As $\tup c,\tup u$ are coordinatewise in $\mu$, $(a,\pi(a))\in \mu$, and 
$\mu$ is abelian, we can use Definition~\ref{df:cent}\eqref{dfcent:it3} and
replace the underlined instances of $\pi(a)$ with $a$ to get
$t(a,\tup c) = t(a,\tup u)$, proving 
$b =t(a,\tup u) \in \sg(S\cup \{a\})$.  As $b$ was arbitrary, we get $\sg(S\cup \{a\})=A$.
\end{proof}

\begin{lm} \label{lm:aut}
Suppose $\m a$ is an algebra with a weak difference term and
$\theta \in \Con(\m a)$ is abelian.
Let $D_1,D_2$ be subuniverses of $\m a$ which are
transversals for $\theta$.
Then there exists an automorphism $\sigma$ of $\m a$ satisfying
$\sigma(D_1)=D_2$ and $\sigma(x) \stackrel\theta\equiv x$ for all $x \in A$.
\end{lm}

\begin{proof}
For $i=1,2$ let $\pi_i:A\ra D_i$ be the unique map given by
$\pi_i(x) \stackequiv{\theta} x$ for all $x \in A$, and note
that $\pi_1$ and $\pi_2$ are retractions of $\m a$ with kernel $\theta$.
Let $d(x,y,z)$ be a weak difference term for $\calV$ and
define $\sigma:A\ra A$ and $\tau:A\ra A$ by
\begin{align*}
\sigma(x) &= d(x,\pi_1(x),\pi_2(x))\\
\tau(x) &= d(x,\pi_2(x),\pi_1(x)).
\end{align*}
Observe first that $\sigma(x) \stackequiv{\theta} \tau(x) 
\stackequiv{\theta} d(x,x,x)=x$ for all $x \in A$.
Fix a $\theta$-class $\thcl$; by construction, there exist $b,c \in \thcl$
such that $\pi_1(x)=b$ and $\pi_2(x)=c$ for all $x \in \thcl$.  
Pick an abelian group operation $+$ on $\thcl$ such that $d(x,y,z)=x-y+z$
for all $x,y,z\in \thcl$.  Then 
$\sigma(x) = x-b+c$ and $\tau(x)=x-c+b$ for all $x \in \thcl$, and since this
is true for every $\theta$-class, we get $\sigma\circ\tau=\tau\circ\sigma
= \mathrm{id}_A$, proving $\sigma$ is a bijection.  Our aim is to prove
that $\sigma$ is the desired automorphism.

Let $f$ be an $n$-ary basic operation.  Let $a_1,\ldots,a_n \in A$ and
put $a=f(a_1,\ldots,a_n)$; we need to prove
$\sigma(a) = f(\sigma(a_1),\ldots,\sigma(a_n))$.  
Let $b_i=\pi_1(a_i)$ and $c_i=\pi_2(a_i)$, and note that $a_i,b_i,c_i$ belong
to a common $\theta$-class $\thcl_i$ for each $i=1,\ldots,n$.  Let 
$b=f(b_1,\ldots,b_n)$ and $c=f(c_1,\ldots,c_n)$, and note that $b \in D_1$ and
$c \in D_2$ (since $D_1,D_2\leq \m a$), so $\pi_1(a)=b$ and $\pi_2(a)=c$.
Thus what must be shown is
\[
d(f(a_1,\ldots,a_n),f(b_1,\ldots,b_n),f(c_1\ldots,c_n)) =
f(d(a_1,b_1,c_1),\ldots,d(a_n,b_n,c_n))
\]
given that $a_i\stackequiv{\theta} b_i \stackequiv{\theta}c_i$ for all $i=1,\ldots,n$.
This follows from Proposition~\ref{prp:TC2}. 
Hence $\sigma$ is an automorphism of $\m a$.

Finally, suppose $a \in D_1$.  Then $\pi_1(a)=a$, so
$\sigma(a) = d(a,a,\pi_2(a)) = \pi_2(a) \in D_2$.  This proves $\sigma(D_1)
\subseteq D_2$.  A similar argument shows $\tau(D_2)\subseteq D_1$,
so $\sigma(D_1)=D_2$.
\end{proof}

Next we give some results concerning varieties with a weak difference term.
The first three Propositions are extracted from \cite{kearnes-kiss,corrig}.

\begin{prp} \label{prp:memoir1}
Suppose $\calV$ is a variety with a weak difference term, $\m a \in \calV$, and
$\theta,\delta \in \Con(\m a)$.
\begin{enumerate}
\item \label{mem:6.7}
If $C(\theta,\theta;\delta)$, then 
$\theta \vee \delta = \delta\circ\theta\circ\delta$ and
$(\theta\vee\delta)/\delta$ is abelian.
\item \label{mem:6.8}
If $\theta$ is abelian, then $C(\theta,\theta;\delta)$.
\end{enumerate}
\end{prp}

\begin{proof}
These are special cases of Lemmas 6.7 and 6.8 respectively from \cite{kearnes-kiss}.  
\end{proof}

\begin{prp} \label{prp:memoir2}
Suppose $\calV$ is a variety with a weak difference term, $\m a \in \calV$, and
$\alpha,\beta \in \Con(\m a)$ with $\alpha\leq \beta$.  If $\beta/\alpha$ is abelian,
then the interval $I[\alpha,\beta]$ is a modular lattice of permuting equivalence
relations.
\end{prp}

\begin{proof}
This is a special case of \cite[Theorem 6.16]{kearnes-kiss}.
\end{proof}

\begin{prp} \label{prp:memoir3}
Suppose $\calV$ is a variety with a weak difference term, $\m a \in \calV$, and
$\sigma,\tau,\delta,\varepsilon \in \Con(\m a)$ with $\sigma\leq \tau$ and $\delta\leq
\varepsilon$ and $(\sigma,\tau)\nearrow (\delta,\varepsilon)$.
\begin{enumerate}
\item \label{mem:3.27}
$\tau/\sigma$ is abelian if and only if $\varepsilon/\delta$ is abelian.
\item \label{mem:6.24}
If $\tau/\sigma$ is abelian, then the map $J:I[\sigma,\tau]\ra I[\delta,\varepsilon]$
given by $J(x)=x\vee \delta$ is surjective.
\end{enumerate}
\end{prp}

\begin{proof}
\eqref{mem:3.27} follows from \cite[Theorem 3.27]{kearnes-kiss} as corrected in
\cite{corrig}.
\eqref{mem:6.24} is a special case of \cite[Theorem 6.24]{kearnes-kiss}.
\end{proof}

Our last result in this section generalizes \cite[Lemma 2.6]{diffterm}.

\begin{prp} \label{prp:ann}
Suppose $\calV$ is a variety with a weak difference term,
$\m a \in \calV$, and $\delta,\varepsilon,\sigma,\tau \in \Con(\m a)$
with $\delta\leq\varepsilon$ and $\sigma\leq\tau$ where
$(\sigma,\tau)\nearrow (\delta,\varepsilon)$.
\begin{enumerate}
\item \label{ann:it1}
If $\varepsilon/\delta$ is abelian, then
$\varepsilon = \delta\circ\tau\circ\delta$ and
$(\delta:\varepsilon) = (\sigma:\tau)$.
\item \label{ann:it2}
If $\delta\prec \varepsilon$,
then $(\delta:\varepsilon) = (\sigma:\tau)$.
\end{enumerate}
\end{prp}

\begin{proof}
\eqref{ann:it1}
We have $C(\varepsilon,\varepsilon;\delta)$ by hypothesis, and hence
$C(\tau,\tau;\delta)$ by monotonicity.  
Then $\varepsilon=\tau\vee\delta = \delta\circ\tau\circ\delta$ by
Proposition~\ref{prp:memoir1}\eqref{mem:6.7}.

$\delta\wedge \tau=\sigma$ implies $C(\delta,\tau;\sigma)$ by e.g.\
\cite[Theorem 2.19(8)]{kearnes-kiss}, proving
$(\sigma:\tau)\geq \delta$.  Thus to prove 
$(\delta:\varepsilon)=(\sigma:\tau)$,
it is enough to show $C(\alpha,\varepsilon;\delta)\iff C(\alpha,\tau;\sigma)$ for all $\alpha \in \Con(\m a)$ with $\alpha \geq \delta$.  
We have $C(\alpha,\tau;\sigma) \iff C(\alpha,\tau;\delta)$ by
\cite[Theorem 2.19(4)]{kearnes-kiss}, since $\tau\wedge \sigma = \tau\wedge \delta$.
Thus we only need to show $C(\alpha,\varepsilon;\delta) \iff
C(\alpha,\tau;\delta)$.
The forward implication follows immediately by monotonicity.
Conversely, suppose $C(\alpha,\tau;\delta)$.  
Then $C(\alpha,\delta\circ\tau\circ
\delta; \delta)$ by \cite[Theorem 2.19(3)]{kearnes-kiss}, i.e.,
$C(\alpha,\varepsilon;\delta)$.

\eqref{ann:it2} 
If $\varepsilon/\delta$ is abelian, then we can apply item~\eqref{ann:it1}.
Assume $\varepsilon/\delta$ is nonabelian.
Then $\tau/\sigma$ is also nonabelian by Proposition~\ref{prp:memoir3}\eqref{mem:3.27}.
As in the proof of \eqref{ann:it1}, we only need to prove
$C(\alpha,\tau;\delta) \implies C(\alpha,\varepsilon;\delta)$ for all $\alpha \in \Con(\m a)$
with $\alpha\geq \delta$.
Assume $C(\alpha,\tau;\delta)$.  If $\alpha\geq \varepsilon$, then $C(\tau,\tau;\delta)$
by monotonicity, which contradicts nonabelianness of $\tau/\sigma$.  Thus
$\alpha\ngeq\varepsilon$.  Since $\alpha\geq \delta$ and $\delta\prec\varepsilon$, we
get $\alpha\wedge\varepsilon=\delta$ and hence $C(\alpha,\varepsilon;\delta)$.
\end{proof}

\section{The difference algebra of an abelian congruence} \label{sec:update}

In this section we revisit an old construction of 
Hagemann and Herrmann \cite{hagemann-herrmann} and
Freese~\cite{freese1983} 
in congruence modular varieties.
We will show that the construction
adapts nicely to varieties with a weak difference term.
In the following definition, we follow the notation of
\cite{freese-mckenzie}.

\begin{df} \label{df:A(mu)}
Suppose $\m a$ is an algebra and $\theta,\varphi \in \Con(\m a)$
with $\theta\leq\varphi$.
\begin{enumerate}
\item \label{A(mu):it1}
$\m a(\theta)$ denotes $\theta$ viewed as a subalgebra of $\m a^2$.
\item \label{A(mu):it3}
$\Delta_{\theta,\varphi}$ denotes
the congruence of $\m a(\theta)$ generated by $\set{((a,a),(b,b))}{$(a,b) \in 
\varphi$}$.
\end{enumerate}
\end{df}

The next notations depart from those in \cite{freese-mckenzie}.

\begin{df} \label{df:A(mu)2}
Let $\m a,\theta,\varphi$ be as in Definition~\ref{df:A(mu)}.
\begin{enumerate}
\item
$\proj_1,\proj_2:\m a(\theta)\ra \m a$ denote the two projection homomorphisms.
\item
$\eta_i = \ker(\proj_i) \in \Con(\m a(\theta))$ for $i=1,2$.
\item
If $\alpha \in \Con(\m a)$ with $\alpha\geq \theta$, then $\baralpha$ denotes
\[
\baralpha=\set{((a_1,a_2),(b_1,b_2)) \in A(\theta)\times A(\theta)}{$(a_1,b_1)
\in \alpha$}.
\]
\end{enumerate}
\end{df}

Note that in the context of Definition~\ref{df:A(mu)2} we have
$\baralpha \in \Con(\m a(\theta))$ and
$\baralpha$ is also given by
$\baralpha=\set{((a_1,a_2),(b_1,b_2)) \in A(\theta)\times A(\theta)}{$(a_2,b_2) 
\in \alpha$}$.  
Also observe that for each $i=1,2$, $\Con(\m a)$ is naturally isomorphic to 
the interval $I[\eta_i,1]$ in $\Con(\m a(\theta))$ 
via $\proj_i^{-1}$, 
and $\alpha$ maps to $\baralpha$ under either isomorphism; see Figure~\ref{fig:ConA(mu)}. 

It will be useful to view $\Delta_{\theta,\varphi}$ as a set of
$(\theta,\varphi)$-matrices.  Conversely, it will be useful to view
$M(\theta,\varphi)$ as a tolerance of $\m a(\theta)$.

\begin{df} \label{df:MD}
Suppose $\m a$ is an algebra and $\theta,\varphi \in \Con(\m a)$.
\begin{enumerate}
\item
Let $\Deltah(\theta,\varphi)$ denote the subuniverse of $\m a(\theta,\varphi)$
``corresponding to" $\Delta_{\theta,\varphi}$:
\[
\Deltah(\theta,\varphi) = \left\{
\begin{pmatrix} a&c\\b&d\end{pmatrix} : ((a,b),(c,d)) \in \Delta_{\theta,\varphi}
\right\}.
\]
\item
Let $M_{\theta,\varphi}$ denote the subuniverse of $\m a(\theta)\times 
\m a(\theta)$ ``corresponding to" $M(\theta,\varphi)$:
\[
M_{\theta,\varphi} = \left\{((a,b),(c,d)):\begin{pmatrix}
a&c\\b&d\end{pmatrix} \in M(\theta,\varphi)\right\}.
\]
\end{enumerate}
\end{df}

\begin{lm} \label{lm:faces}
Suppose $\m a$ is an algebra and $\theta,\varphi \in \Con(\m a)$.
\begin{enumerate}
\item \label{faces:it1}
$\Delta_{\theta,\varphi}$ is the transitive closure of of $M_{\theta,\varphi}$.
Thus
$\Deltah(\theta,\varphi)$ is the ``horizontal transitive closure of $M(\theta,\varphi)$."
\item \label{faces:it1.25}
$\Deltah(\theta,\varphi)$ is ``vertically symmetric."  I.e.,
\[
\mbox{If~} \begin{pmatrix} a&c\\b&d\end{pmatrix} \in \Deltah(\theta,\varphi),~
\mbox{then}~ \begin{pmatrix} b&d\\a&c\end{pmatrix} \in \Deltah(\theta,\varphi).\\
\]

\item \label{faces:it2}
If $\m a$ has a weak difference term and $\theta$ is abelian, then
$\Deltah(\theta,\varphi)$ is ``vertically transitive." I.e.,
\begin{align*}
\mbox{If~} &\begin{pmatrix} a&c\\b&d\end{pmatrix},
\begin{pmatrix} b&d\\r&s\end{pmatrix} \in \Deltah(\theta,\varphi),~
\mbox{then}~\begin{pmatrix}a&c\\r&s\end{pmatrix} \in \Deltah(\theta,\varphi).
\end{align*}

\item \label{faces:it1.5}
If $\m a$ has a weak difference term and $\theta$ is abelian,
then $\Deltah(\theta,\theta)=M(\theta,\theta)$.
\end{enumerate}
\end{lm}

\begin{proof}
\eqref{faces:it1}
$M_{\theta,\varphi}$ is the reflexive, symmetric subuniverse of
$\m a(\theta)\times \m a(\theta)$ generated by $\set{((a,a),(b,b))}{$(a,b)\in
\varphi$}$.

\eqref{faces:it1.25}
This follows from \eqref{faces:it1} and
the fact that 
$M(\theta,\varphi)$ is vertically symmetric.

\eqref{faces:it2}
We can consider $\Deltah(\theta,\varphi)$ as a binary relation on
the set of its rows.  As such, it is a reflexive subuniverse
of $\m a(\varphi)\times \m a(\varphi)$, and if
$\theta$ is abelian, a weak difference term for $\m a$, interpreted
coordinatewise in $\m a(\varphi)$, will
 satisfy the hypothesis of
Lemma~\ref{lm:maltsev} for this relation.

\eqref{faces:it1.5}
Let $\rho = M_{\theta,\theta}$; then $\rho$ is a reflexive 
subuniverse of $\m a(\theta)\times \m a(\theta)$.
If $\theta$ is abelian and $\m a$ has a weak difference term
$d$, then $d$ 
interpreted coordinatewise
in $\m a(\theta)$ satisfies the
hypothesis of Lemma~\ref{lm:maltsev} for $\rho$.
Thus $\rho$ is transitive by Lemma~\ref{lm:maltsev},
implying $\Delta_{\theta,\theta}=M_{\theta,\theta}$ by \eqref{faces:it1}
and hence $\Deltah(\theta,\theta)=M(\theta,\theta)$.
\end{proof}

\begin{rem} \label{rem:2-con}
Assuming the hypotheses of Lemma~\ref{lm:faces}\eqref{faces:it2},
$\Deltah(\theta,\varphi)$ is an example of what A.~Moorhead calls a
``2-dimensional congruence" \cite{moorhead-supernilp}.  In fact, 
$\Deltah(\theta,\varphi)$ is the smallest 2-dimensional congruence
generated by $M(\theta,\varphi)$, which Moorhead denotes by
$\Delta(\theta,\varphi)$.
\end{rem}

The next theorem has its roots in \cite[Proposition 1.1]{hagemann-herrmann}.

\begin{thm} \label{thm:D(A)}
Suppose $\calV$ is a variety with a weak difference term, 
$\m a \in \calV$, 
$\theta$ is a nonzero abelian congruence of $\m a$,
and $\alpha=(0:\theta)$. 
Let $\Delta=\Delta_{\theta,\alpha}$,
let $\eta_1,\eta_2,\bartheta$ 
be as in Definition~\ref{df:A(mu)2},
and let $\varepsilon=\bartheta\wedge \Delta$.
Then:
\begin{enumerate}
\item \label{D(A):it1}
For any $a \in A$, $(a,a)/\Delta = \set{(b,b)}{$(a,b) \in\alpha$}$.
\item \label{D(A):it3}
$\Delta<\baralpha$ and $(\Delta:\baralpha)=\baralpha$.
\item \label{D(A):it4a}
$\bartheta$ is abelian, and
the interval $I[0,\bartheta]$ in $\Con(\m a(\theta))$ is a modular lattice 
consisting of permuting congruences.
\item \label{D(A):it4b}
$\{0,\eta_1,\eta_2,\varepsilon,\bartheta\}$ is a sublattice 
of $I[0,\bartheta]$ isomorphic to $\m m_3$.
\item \label{D(A):it4c}
Each of $(\varepsilon,\bartheta), (0,\eta_1), (0,\eta_2)$ transposes up
to $(\Delta,\baralpha)$.
\item \label{D(A):new}
If $\theta$ is minimal, i.e., $0\prec \theta$, then 
the interval $I[0,\bartheta]$ has height $2$, $\Delta \prec\baralpha$,
and $\Delta$ is completely meet-irreducible.
\end{enumerate}
\end{thm}

\begin{proof}
\eqref{D(A):it1} follows from the fact that $\alpha$ centralizes $\theta$.
In $\Con(\m a(\theta))$ we have $\eta_1 \wedge \eta_2=0$
and $\eta_i< \bartheta$ for $i=1,2$.  
Clearly $\Delta \leq \baralpha$.  
Item~\eqref{D(A):it1} and $\theta\ne 0$ imply $\bartheta \nleq \Delta$ and hence
$\Delta<\baralpha$ and $\varepsilon < \bartheta$.  
If $\varphi$ is any congruence of $\m a$ satisfying
$\theta\leq \varphi\leq \alpha$, then we can show that
$(\Delta\wedge \barvarphi) \vee \eta_1 =(\Delta\wedge \barvarphi)\vee\eta_2= 
\barvarphi$.  Indeed, if $((a_1,a_2),(b_1,b_2)) \in
\barvarphi$ then 
\[
(a_1,a_2) \stackrel{\eta_i}{\equiv} (a_i,a_i) \stackrel{\Delta\wedge\barvarphi}{\equiv}
(b_i,b_i) \stackrel{\eta_i}{\equiv} (b_1,b_2).
\]
In particular, this proves $\Delta\vee\eta_i=\baralpha$ and
$\varepsilon\vee\eta_i=\bartheta$ for $i=1,2$; hence $\varepsilon\ne 0$.
The first equality also proves $\Delta \vee \bartheta = \baralpha$.
We also get $\Delta\wedge \eta_1=\Delta\wedge\eta_2=0$ 
by \eqref{D(A):it1}; hence $\varepsilon\wedge \eta_i=0$ for $i=1,2$.
And we can easily verify $\eta_1\vee \eta_2=\bartheta$.
This proves \eqref{D(A):it4b} and \eqref{D(A):it4c}.
See Figure~\ref{fig:ConA(mu)} showing the placement of these congruences in
$\Con(\m a(\theta))$.

\begin{figure}
\begin{tikzpicture}[scale=1.3]

\draw [fill] (0,0) circle (1.0pt);
\draw [fill] (-.4,.6) circle (1.0pt);
\draw [fill] (.4,.6) circle (1.0pt);
\draw [fill] (0,1) circle (1.0pt);

\draw [fill] (-5,.5) circle (1.0pt);
\draw [fill] (-5,1) circle (1.0pt);
\draw [fill] (-5,2.732) circle (1.0pt);
\draw (-5,1) arc(-60:60:1);
\draw (-5,1) arc(240:120:1);
\draw [-] (-5,.5) -- (-5,1);

\draw (0,1) arc(-60:60:1);
\draw (0,1) arc(240:120:1);
\draw [fill] (0,2.732) circle (1.0pt);
\draw (0,0) arc(260:100:1.40);
\draw (0,0) arc(-80:80:1.40);

\draw (-5,.5) arc(-60:60:1.29);
\draw (-5,.5) arc(240:120:1.29);

\draw (0,2.732) arc(120:220:1.42);
\draw (0,2.732) arc(60:-40:1.42);

\node [anchor=east] at (-.35,.6) {$\scriptstyle{\eta_1}$};
\node [anchor=west] at (.35,.6) {$\scriptstyle{\eta_2}$};

\node [anchor=north] at (0,0) {$\scriptstyle{0}$};
\node [anchor=south] at (.2,2.732) {$\scriptstyle{1 = \overline{1}}$};
\node [anchor=east] at (0,1) {$\scriptstyle{\bartheta}$};

\draw [-] (0,0) -- (.4,.6) -- (0,1) -- (-.4,.6) -- (0,0);

\node at (-2.4,1.36) {$\Con(\m a(\theta))=$};
\node at (-6.5,1.36) {$\Con(\m a)=$};
\node [anchor=north] at (-5,.5) {$\scriptstyle{0}$};
\node [anchor=east] at (-5,1) {$\scriptstyle{\theta}$};
\node [anchor=south] at (-5,2.732) {$\scriptstyle{1}$};

\draw [fill] (-4.75,1.8) circle (1.0pt);    
\node [anchor=west] at (-4.80,1.9) {$\scriptstyle{\alpha}$};

\draw [-] (-5,1) -- (-4.75,1.8) -- (-5,2.732);

\draw [fill] (.25,1.8) circle (1.0pt);    
\node [anchor=west] at (.20,1.9) {$\scriptstyle{\baralpha}$};

\draw [-] (0,1) -- (.25,1.8) -- (0,2.732);

\draw [fill] (.85,1.3) circle (1.0pt);    
\node [anchor=west] at (.80,1.3) {$\scriptstyle{\Delta}$};
\draw [-] (.85,1.3) -- (.25,1.8);

\draw [fill] (0,.5) circle (1.0pt);    
\node [anchor=east] at (0.05,.5) {$\scriptstyle{\varepsilon}$};
\draw [-] (.85,1.3) -- (0,.5);
\draw [-] (0,0) -- (0,1);

\end{tikzpicture}

\caption{$\Con(\m a)$ and $\Con(\m a(\theta))$} \label{fig:ConA(mu)}
\end{figure}

Since $\theta$ is abelian, we have $C(\bartheta,\bartheta;\eta_i)$ for $i=1,2$.
As $\eta_1\wedge \eta_2=0$, we get $C(\bartheta,\bartheta;0)$, i.e.,
$\bartheta$ is abelian.  Hence by 
Proposition~\ref{prp:memoir2},
the interval $I[0,\bartheta]$ is a modular lattice of permuting equivalence
relations.
This proves \eqref{D(A):it4a}.

As $\bartheta$ is abelian, we use Proposition~\ref{prp:memoir1}\eqref{mem:6.8} to get
$C(\bartheta,\bartheta;\varepsilon)$, i.e.,
$\bartheta/\varepsilon$ is abelian.  
As 
$(\varepsilon,\bartheta) \nearrow (\Delta,\baralpha)$,
we get that
$\baralpha/\Delta$ is abelian by
Proposition~\ref{prp:memoir3}\eqref{mem:3.27}.
Since also $\bartheta/\eta_1$ is abelian and $(\eta_1,\bartheta)
\searrow (0,\eta_2) \nearrow (\varepsilon,\bartheta) \nearrow (\Delta,\baralpha)$,
we get $(\eta_1:\bartheta) = (0:\eta_2) = (\varepsilon:\bartheta) =
(\Delta:\baralpha)$
by Proposition~\ref{prp:ann}\eqref{ann:it1}.
Finally, $(\eta_1:\bartheta)=\overline{(0:\theta)}=\baralpha$ by 
Lemma~\ref{lm:quotient}\eqref{quotient:it2} and the fact that
$(0:\theta)=\alpha$, so
$(\Delta:\baralpha)=\baralpha$.
This completes the proof of \eqref{D(A):it3}.

Now assume $0\prec \theta$.  Then $\eta_i\prec \bartheta$ for $i=1,2$.
Since $\eta_1\wedge
\eta_2=0$, we get $0\prec\eta_1,\eta_2$ (see e.g.\ \cite[Corollary 2.29(ii)]{alvin1}),
and every maximal chain in the interval $I[0,\bartheta]$ has length 2
\cite[Theorem 2.37]{alvin1}.  Thus $I[0,\bartheta]$ has height 2.
Note that this also implies $0\prec\varepsilon \prec\bartheta$.
Because $\bartheta/\varepsilon$ is abelian and $(\varepsilon,\bartheta)\nearrow (\Delta,
\baralpha)$, the map $J:I[\varepsilon,\bartheta]\ra I[\Delta,\baralpha]$ given
by $J(\gamma)=\gamma\vee\varepsilon$ is surjective by
Proposition~\ref{prp:memoir3}\eqref{mem:6.24}.
Thus $\Delta\prec\baralpha$.
Now let $\varphi \in \Con(\m a(\mu))$ with 
$\Delta\leq \varphi$ 
but $\baralpha\nleq \varphi$.
Then $\baralpha\wedge \varphi=\Delta$, so $C(\varphi,\baralpha;\Delta)$,
so $\varphi \leq (\Delta:\baralpha) = \baralpha$,
so $\varphi = \baralpha \wedge \varphi = \Delta$.
This proves~\eqref{D(A):new}.
\end{proof}

The next definition is essentially from 
\cite{freese1983,freese-mckenzie} in the 
congruence modular setting; we extend it to varieties with a weak
difference term.

\begin{df} \label{df:D(A,theta)}
Suppose $\calV$ is a variety with a weak difference term, 
$\m a \in \calV$, and $\theta \in \Con(\m a)$ where $\theta$ is abelian. 
Let $\alpha=(0:\theta)$.
\begin{enumerate}
\item
The \emph{difference algebra for $\theta$}, denoted
$D(\m a,\theta)$, is the algebra
$\m a(\theta)/\Delta_{\theta,\alpha}$.
\item
The \emph{derived congruence}, which we denote by $\varphi$,
is the congruence $\baralpha/\Delta_{\theta,\alpha}$.
\item
The \emph{canonical transversal} is the set $\Do := \set{0_\alcl}{$\alcl\in A/\alpha$}$.
\end{enumerate}
\end{df}

\begin{cor} \label{cor:D(A)}
Suppose $\calV$ is a variety with a weak difference term, $\m a\in
\calV$, $\theta$ is an abelian congruence of $\m a$,
and $\alpha=(0:\theta)$.  Let $D(\m a,\theta)$
be the difference algebra for $\theta$, let $\Dmon$ be the derived congruence,
let $\Do$ be the canonical transversal, and let
$\nu:\m a(\theta)\ra D(\m a,\theta)$ be the natural projection map.
Then:
\begin{enumerate}
\item \label{corD(A):itab}
 $\Dmon$ is an abelian congruence of $D(\m a,\theta)$.
\item \label{corD(A):itself}
$(0:\Dmon)=\Dmon$.
\item \label{corD(A):it1}
$\Do$ is a transversal for $\Dmon$ and 
$\Do\leq D(\m a,\theta)$. 
\item \label{corD(A):it4}
$\nu^{-1}(\Do)=0_A$.
\item \label{corD(A):it2}
There exists an isomorphism $h:\m a/\alpha \cong
D(\m a,\theta)/\Dmon$ such that
for all $(a,b) \in \theta$,
$\nu(a,b)/\Dmon = h(a/\alpha)$.
\end{enumerate}
If in addition 
$\theta$ is a minimal congruence of $\m a$, then $D(\m a,\theta)$ is subdirectly
irreducible with monolith $\Dmon$.
\end{cor}

\begin{proof}
Items~\eqref{corD(A):itab} and~\eqref{corD(A):itself}
follow from Theorem~\ref{thm:D(A)}\eqref{D(A):it3}.  
Let $\Delta=\Delta_{\theta,\alpha}$.
For each $\alpha$-class
$\alcl$ observe that $0_\alcl = \set{(c,c)}{$c \in \alcl$}$ is a $\Delta$-class 
by Theorem~\ref{thm:D(A)}\eqref{D(A):it1}.
Hence $\Do \subseteq D(\m a,\theta)$.
It is easy to check that $\Do\leq D(\m a,\theta)$.  Given $(a,b)\in \theta$ we 
have $(a,b) \stackrel{\baralpha}{\equiv}(a,a)$ and so
$(a,b)/\Delta \stackrel{\Dmon}{\equiv} (a,a)/\Delta \in \Do$.
Thus $\Do$ meets every $\Dmon$-class.  
Suppose $0_\alcl,0_{\alcl'} \in \Do$ 
and $0_\alcl\stackrel{\Dmon}{\equiv}0_{\alcl'}$.  
Pick $c \in \alcl$ and $c' \in \alcl'$.
Then $(c,c) \stackrel{\baralpha}{\equiv} (c',c')$, so $(c,c') \in \alpha$, so $\alcl=\alcl'$.
Thus $\Do$ is a transversal for $\Dmon$.

To prove item~\eqref{corD(A):it4}, we must show that for all $(a,b) \in \theta$
we have $(a,b)\stackrel{\Delta}{\equiv}(c,c)$ for some $c \in A$ if and only if
$a=b$; and this follows from Theorem~\ref{thm:D(A)}\eqref{D(A):it1}.  
To prove item~\eqref{corD(A):it2}, define
$h$ by $h(a/\alpha) = ((a,a)/\Delta)/\varphi$.

Finally, if $\theta$ is a minimal congruence, then
Theorem~\ref{thm:D(A)}\eqref{D(A):new} implies
that $D(\m a,\theta)$ is subdirectly irreducible with monolith $\Dmon$.
\end{proof}

\section{The classes of an abelian congruence} 
\label{sec:structure}

Suppose $\m a$ is an algebra in a congruence modular variety,
$\theta$ is an abelian congruence of $\m a$, and $\alpha = (0:\theta)$.
A classical result of Gumm \cite[see 3.3]{gumm-easyway} states that
if $\thcl,\thcl'$ are two $\theta$-classes contained in a common 
$\alpha$-class $\alcl$, 
then $|\thcl|=|\thcl'|$ and in fact
$\Grp a(\theta,e)$ and $\Grp a(\theta,e')$ are isomorphic for any
$e \in \thcl$ and $e' \in \thcl'$.
Moreover, it can be shown that
these groups are all isomorphic to $\Grpd{D(\m a,\theta)}(\Dmon,0_\alcl)$
where $\Dmon$ is the derived congruence of the difference algebra
$D(\m a,\theta)$.
These results extend to varieties with a difference term, as we shall see
in Lemma~\ref{lm:inject} below.
However they fail in general for varieties with a weak difference term, and in
Appendix~\hyperlink{A2}{2} we will show just how badly they can fail there.

Nonetheless, we shall show that in varieties with a weak difference term,
$D(\m a,\theta)$ 
and 
$\Dmon$
have the following property:
for every $\alpha$-class $\alcl$, the group $\Grpd{D(\m a,\theta)}(\Dmon,0_\alcl)$ forms a sort
of ``universal domain" for all the groups $\Grp a(\theta,e)$ with $e \in \alcl$.

\begin{df} \label{df:lambda-e}
Suppose $\m a$ is an algebra in a variety $\calV$ with a weak difference term,
$\theta$ is an abelian congruence, $\alpha=(0:\theta)$, $\Delta =
\Delta_{\theta,\alpha}$, and $\Dmon = \baralpha/\Delta$.
Given $e \in A$, the map $\embed e:e/\theta \ra D(\m a,\theta)$ is defined by
$\embed e(x)=(x,e)/\Delta$.
\end{df}

\begin{lm} \label{lm:inject}
Suppose $\m a,\calV,\theta,\alpha,\Delta,\Dmon$ are as in Definition~\ref{df:lambda-e}.
\begin{enumerate}
\item \label{inject:it1}
For all $e \in A$, the map $\embed e$
is an embedding of
$\Grp a(\theta,e)$ into $\Grpd{D(\m a,\theta)}(\Dmon,0_{e/\alpha})$.  
\item\label{inject:it2}
If $\calV$ has a difference term, then $\embed e$ is an
isomorphism.
\end{enumerate}
\end{lm}

\begin{proof}
\eqref{inject:it1}
Let $\alcl=e/\alpha$.
Clearly $\embed e(e) = 0_\alcl$.  If $(x,e)/\Delta = (y,e)/\Delta$, then 
$x,y,e$ all belong to a common $\theta$-class.  
So
\begin{align*}
(x,e) &\stackequiv{\Delta} (y,e), \quad
(y,e) \stackequiv{\Delta} (y,e), \quad
(y,y) \stackequiv{\Delta} (y,y).
\end{align*}
Applying the weak difference term $d$ to these pairs and using $(x,y),(y,e)\in
\theta$ gives
\[
(x,y)=(d(x,y,y),d(e,e,y)) \stackequiv{\Delta} (d(y,y,y),d(e,e,y)) = (y,y).
\]
By Theorem~\ref{thm:D(A)}\eqref{D(A):it1}, we get
$x=y$.  Thus $\embed e$
is injective.  It is easily checked that $\embed e$ preserves $+$.

\eqref{inject:it2}
Suppose that $d(x,y,z)$ is a difference term, i.e., $\calV \models d(x,x,y)\approx y$.
Given $(r,s)/\Delta \in 0_\alcl/\Dmon$, so $(r,s) \in \theta \cap \alcl^2$,
define $a:=d(r,s,e)$.  Then
$a \stackequiv{\theta} d(r,r,e)=e$, so $a \in e/\theta$.  Observe that
\[
\begin{pmatrix} r&r\\s&s\end{pmatrix},\
\begin{pmatrix} s&s\\s&s\end{pmatrix}, 
\begin{pmatrix} s&e\\s&e\end{pmatrix} \in M(\theta,\alpha).
\]
Also observe that $d(r,s,s)=r$ as $(r,s)\in \theta$.
Thus applying $d(x,y,z)$ to the above matrices gives
\[
\begin{pmatrix} r & a\\s & e\end{pmatrix} =
\begin{pmatrix} d(r,s,s) & d(r,s,e)\\d(s,s,s) & d(s,s,e)\end{pmatrix} \in M(\theta,\alpha)\subseteq \Deltah(\theta,\alpha).
\]
Hence 
$(r,s)\stackrel\Delta\equiv (a,e)$,
so $\embed e(a) = (r,s)/\Delta$, proving $\embed e$ is surjective.
\end{proof}

%

\begin{lm}\label{lm:samerange}
Suppose $\m a,\calV,\theta,\alpha,\Delta$ are as in Definition~\ref{df:lambda-e}.
\begin{enumerate}
\item \label{samerange:it1}
If $a\stackrel\theta\equiv b \stackrel\theta\equiv e$, then
$(a,b) \stackrel\Delta\equiv (d(a,b,e),e)$.
\item \label{samerange:new}
If $a\stackrel\theta\equiv b$ and
$a \stackrel\alpha\equiv c\stackrel\alpha\equiv e$,
then $(d(a,e,c),d(b,e,c))\stackrel\Delta\equiv (a,b)$.
\item
\label{samerange:it1.5}
For all $e \in A$, $\ran(\embed e) = \thcl^2/\Delta$ where $\thcl=e/\theta$.
\end{enumerate}
\end{lm}

\begin{proof}
\eqref{samerange:it1}
Let $x = d(a,b,e)$.  
Observe that
\[
\begin{pmatrix} a&a\\b&b\end{pmatrix},\
\begin{pmatrix} b&b\\b&b\end{pmatrix}, 
\begin{pmatrix} e&b\\e&b\end{pmatrix} \in M(\theta,\theta). 
\]
Applying $d$ to these matrices gives
\[
\begin{pmatrix} x & a\\e & b\end{pmatrix} =
\begin{pmatrix} d(a,b,e) & d(a,b,b)\\d(b,b,e) & d(b,b,b)\end{pmatrix} \in M(\theta,\theta)
\subseteq M(\theta,\alpha)\subseteq \Deltah(\theta,\alpha).
\]
Hence $(x,e) \stackequiv{\Delta} (a,b)$ as required.

\eqref{samerange:new} This is proved similarly to \eqref{samerange:it1},
by applying $d$ to the matrices
\[
\begin{pmatrix} a&a\\b&b\end{pmatrix},\quad 
\begin{pmatrix} e&a\\e&a\end{pmatrix},\quad 
\begin{pmatrix} c&a\\c&a\end{pmatrix}.
\]

\eqref{samerange:it1.5}
Clearly $\ran(\embed e)\subseteq \thcl^2/\Delta$.  Let $a,b \in \thcl$.
Then $d(a,b,e)\in \thcl$ and $\embed e(d(a,b,e)) = 
(d(a,b,e),e)/\Delta = (a,b)/\Delta$ by \eqref{samerange:it1},
proving $\thcl^2/\Delta\subseteq \ran(\embed e)$.
%
\end{proof}

The previous lemma justifies the next definition.

\begin{df} \label{df:Brange}
Suppose $\m a,\calV,\theta,\alpha,\Delta$ are as in Definition~\ref{df:lambda-e}.
If $\thcl$ is a $\theta$-class and $\alcl$ is the $\alpha$-class containing $\thcl$,
then let $\ran(\thcl)$ denote 
the subgroup of $\Grpd{D(\m a,\theta)}(\Dmon,0_\alcl)$ with universe
$\thcl^2/\Delta$ (equivalently, with universe 
$\ran(\embed e)$ for any $e \in \thcl$).
We call $\ran(\thcl)$ the \emph{range} of $\thcl$.
\end{df}

The next definitions and lemma contain an important tool that will help us
understand the arrangement of ranges of $\theta$-classes in their
corresponding $\Grpd {D(\m a,\theta)}(\Dmon,0_\alcl)$.

\begin{df} \label{df:dE-trans}
Suppose $\m a$ is an algebra with a weak difference term $d$, and 
$\alcl\subseteq A$.
\begin{enumerate}
\item
A \emph{basic positive $(d,\alcl)$-translation} is a unary polymial of $\m a$ of the
form $d(x,e,e')$ or $d(e,e',x)$ with $e,e' \in \alcl$.
\item
A \emph{positive $(d,\alcl)$-translation} is any composition of zero or more
basic positive $(d,\alcl)$-translations.  (The composition of zero 
unary polynomials is $\mathrm{id}_A$.)
\end{enumerate}
\end{df}

\begin{df} \label{df:arrow}
Suppose $\m a$ is an algebra with a weak difference term $d$,
$\theta \in \Con(\m a)$ is abelian, and $\alpha=(0:\theta)$.
Let $\thcl,\thcl'$ be $\theta$-classes contained in a common $\alpha$-class
$\alcl$.  
We write $\thcl\ra\thcl'$ to mean there exists a positive 
$(d,\alcl)$-translation
$f$ such that  $f(\thcl)\subseteq \thcl'$.  We call such $f$ a
\emph{witness} to $\thcl\ra\thcl'$.
\end{df}

Thus $\ra$ is a reflexive transitive directed edge relation on $A/\theta$
which does not connect $\theta$-classes from different $\alpha$-classes.
Note that $\ra$ depends on the choice of $d$.

\begin{lm} \label{lm:arrow}
Suppose $\m a,\calV,\theta,\alpha,\Delta$ are as in Definition~\ref{df:lambda-e}.
Fix a weak difference term $d$ for $\calV$.
Let $\alcl$ be an $\alpha$-class and let $\thcl,\thcl'$ be $\theta$-classes
contained in $\alcl$.
\begin{enumerate}
\item \label{arrow:it1.5}
If $\thcl\ra\thcl'$ witnessed by $f$, then 
$\ran(\thcl)\subseteq \ran(\thcl')$, and
$(f(a),f(b)) \stackrel\Delta\equiv (a,b)$ for all $a,b \in \thcl$.
\item \label{arrow:it3}
There exists a $\theta$-class $\thcl''\subseteq \alcl$ with $\thcl\ra \thcl''$ and $\thcl'\ra \thcl''$.
\item \label{arrow:it4}
If $\thcl\ra \thcl'$, then for all $e \in \thcl$ and $e' \in \thcl'$ there 
exists a witness $f$
to $\thcl\ra \thcl'$ satisfying $f(e)=e'$.
\item \label{arrow:it5}
If $d$ is a difference term for $\m a$, then $\thcl_1\ra \thcl_2$ for all 
$\theta$-classes $\thcl_1,\thcl_2$ contained in a common $\alpha$-class.
\end{enumerate}
\end{lm}

\begin{proof}
\eqref{arrow:it1.5} 
It suffices to prove this
for basic positive
$(d,\alcl)$-translations.  Assume $f(x)=d(x,e,e')$ with $e,e' \in E$ and 
$f(\thcl)\subseteq \thcl'$.  
For all $a,b \in \thcl$,
$(f(a),f(b))\stackrel\Delta\equiv (a,b)$ follows from Lemma~\ref{lm:samerange}\eqref{samerange:new}.  This also proves $\thcl^2/\Delta \subseteq (\thcl')^2/
\Delta$, so $\ran(\thcl)\subseteq \ran(\thcl')$.
A symmetrical proof works if $f(x) = d(e,e',x)$.

\eqref{arrow:it3} 
Choose $e \in \thcl$ and $e' \in \thcl'$
and let $\thcl''$ be the $\theta$-class containing
$d(e,e,e')$.  $\thcl''\subseteq \alcl$ by idempotence of $d$.  Define
$f_1(x) = d(x,e,e')$ and $f_2(x) = d(e,e,x)$.  Then 
$f_1,f_2$ witness $\thcl\ra\thcl''$ and $\thcl'\ra\thcl''$ respectively.


\eqref{arrow:it4}
Suppose $f$ witnesses $\thcl\ra \thcl'$, $e \in \thcl$, and $e' \in \thcl'$.  
Let $g(x) = d(x,f(e),e')$ and $h = g\circ f$.  Then $h$ witnesses
$\thcl\ra\thcl'$ and $h(e) = e'$.

\eqref{arrow:it5}
Taking $e_i \in \thcl_i$, the polynomial
$f(x)=d(x,e_1,e_2)$ will witness $\thcl_1\ra \thcl_2$, since
$\calV\models d(x,x,y) \approx y$.
\end{proof}

\begin{cor} \label{cor:directed}
Suppose $\m a,\calV,\theta,\alpha,\Delta,\Dmon$ are as in Definition~\ref{df:lambda-e}.  Let $\alcl \in A/\alpha$.
\begin{enumerate}
\item \label{directed:it1}
The set
$S_\alcl:=\set{\ran(\thcl)}{$\thcl \in A/\theta$ and $\thcl\subseteq \alcl$}$
is a directed set of subgroups of $\Grpd{D(\m a,\theta)}(\Dmon,0_\alcl)$
whose union is $\Grpd{D(\m a,\theta)}(\Dmon,0_\alcl)$.
\item \label{directed:it2}
Hence if $\alcl$ is finite, there exists 
$\thcl \in \alcl/\theta$
with $\ran(\thcl)=\Grpd{D(\m a,\theta)}(\Dmon,0_\alcl)$.
\item \label{directed:it3}
If $\calV$ has a difference term, then $S_{\alcl} = \{\Grpd{D(\m a,\theta)}(\Dmon,0_\alcl)\}$.
\end{enumerate}
\end{cor}

\begin{proof}
For all $\thcl_1,\thcl_2 \in A/\theta$ with $\thcl_1,\thcl_2\subseteq \alcl$, there exists $\thcl
\in A/\theta$ with 
$\thcl\subseteq \alcl$ and
$\ran(\thcl_1)\cup \ran(\thcl_2)\subseteq \ran(\thcl)$ by
Lemma~\ref{lm:arrow}; this means $S_\alcl$ is directed.
To show that the union of $S_\alcl$ is all of $\Grpd{D(\m a,\theta)}(\Dmon,0_\alcl)$,
let $(a,b)/\Delta \in \Grpd{D(\m a,\theta)}(\Dmon,0_\alcl)$.  Let $\thcl=b/\theta$;
then $a \in \thcl$ so $(a,b)/\Delta \in
\ran(\thcl) \subseteq \bigcup S_\alcl$.
This proves \eqref{directed:it1} and \eqref{directed:it2}.  
Item \eqref{directed:it3}
follows from Lemma~\ref{lm:inject}\eqref{inject:it2}.
\end{proof}

As mentioned at the start of this section,
if $\m a,\calV,\theta,\alpha,\Delta,\Dmon$ are as in 
Definition~\ref{df:lambda-e} and $\calV$ is congruence modular,
or more generally has a difference term,
then it follows from Lemma~\ref{lm:inject}\eqref{inject:it2} that
all $\theta$-classes in a common $\alpha$-class
have the same cardinality.
As we will show in Appendix~\hyperlink{A2}{2},
this uniformity is no longer
to be expected in varieties with a weak difference term.  
However, we do know of one nontrivial restriction on the sizes of the
classes of an abelian \emph{minimal} congruence
$\theta$, when the algebra is idempotent and $(0:\theta)=1$.

\begin{lm} \label{lm:idem}
Suppose $\m a$ is a finite idempotent
algebra in a variety with a Taylor term, $\theta$ is an
abelian minimal congruence, and $(0:\theta)=1$.  Let $q$ be the maximal size of a $\theta$-class.
Then every $\theta$-class has size $1$ or $q$.
\end{lm}

\begin{proof}
$\HSP(\m a)$ has a weak
difference term by finiteness of $\m a$ and Theorem~\ref{thm:locfin}.
Pick a $\theta$-class $\Bmax$ of size $q$.  Suppose there exists a $\theta$-class 
$\thcl$ with $1<|\thcl|<q$.  Let $\Delta=\Delta_{\theta,1}$.
Fix $e \in \Bmax$ and $0 \in \thcl$.  We have 
$\ran(C)\subset A(\theta)/\Delta = \ran(\Bmax)$ by 
Lemma~\ref{lm:inject} and Corollary~\ref{cor:directed}.
Fix $b \in \thcl\setminus \{0\}$ and let $b' \in \Bmax\setminus \{e\}$ be such that
$\embed{0}(b)=\embed e(b')$.
Also pick $u \in \Bmax$ with $\embed e(u) \not\in \ran(C)$.

By Corollary~\ref{cor:maltsev}, 
there exists a unary polynomial $f \in \Pol_1(\m b)$
with $f(b')=u$ and $f(e)=e$.  Choose a $(1+n)$-ary term $t(x,\tup y)$ and $\tup c \in A^n$
so that $f(x) = t(x,\tup c)$.  Let $u' = t(b,0,\ldots,0)$.  We have the following:
\[
(b',e) \stackrel{\Delta}{\equiv} (b,0), \quad
(c_1,c_1) \stackrel{\Delta}{\equiv} (0,0), \quad \ldots,\quad
(c_n,c_n) \stackrel{\Delta}{\equiv} (0,0)
\]
where the first congruence holds because $\embed e(b')=\embed 0(b)$ while the
remaining congruences are because $(0:\theta)=1$.
Now applying $t$ coordinatewise and using idempotency gives
$(u,e) \stackrel{\Delta}{\equiv} (u',0)$
which contradicts our choice of $u$.
\end{proof}

\section{The division ring of an abelian minimal congruence} \label{sec:divring}

Let $\m a$ be an algebra with a weak difference term, let $\theta$ be an
abelian congruence, and assume that $\theta$ is \emph{minimal},
i.e., $0\prec \theta$.
In this section we will 
explain in detail how to construct a division ring
$\bbF_\theta$ and equip each $\theta$-class with the structure of a
left vector space over $\bbF_\theta$.
One construction of $\bbF_\theta$ was sketched by Freese 
in \cite[Section 4]{freese1983}
assuming $\m a$ belongs to a congruence modular variety.
Our construction, which uses the difference algebra of $\theta$,
is simpler and better adapted to varieties which have a weak
difference term but no difference term.

\subsection{Easy case}
We first consider an easy case.
Throughout this subsection we make the following blanket assumptions:

\begin{assum}\label{assum:easycase}
$\calV$ is a variety with
a weak difference term, $\m d\in \calV$, $\varphi \in \Con\m d$, $0\prec\varphi$,
$\varphi$ is abelian, and $\m d$ has a subuniverse $T\leq \m d$ which is a
transversal for $\varphi$.
\end{assum}

\begin{df}
Given Assumption~\ref{assum:easycase}, 
\begin{align*}
F =F_{\m d,\varphi,T} &:= \set{\lambda \in \End(\m d)}{$\lambda(\varphi)
\subseteq \varphi$ and $\lambda(e)=e$ for all $e \in T$}.\\
0=0_{\m d,\varphi,T}&:=  \mbox{the retraction $\m d\ra \m t$ given by
$x \stackrel\varphi\equiv 0(x) \in T$.}\\
1&:= \mathrm{id}_D
\end{align*}

If $\lambda,\mu \in F$, then $\lambda+\mu:D\ra D$ and $-\lambda:D\ra D$ are defined by
\begin{align*}
(\lambda+\mu)(x) &:= d(\lambda(x),0(x),\mu(x)).\rule{2.35in}{0in}\\
(-\lambda)(x) &:= d(0(x),\lambda(x),0(x))
\end{align*}
where $d$ is a weak difference term for $\calV$. (Note that this definition does not
depend on the choice of $d$, by Lemma~\ref{lm:samewdt}.)
\end{df}

\begin{lm} \label{lm:ring-easycase}
Given Assumption~\ref{assum:easycase}:
\begin{enumerate}
\item \label{ringeasy:it1}
If $\lambda,\mu \in F$, then $\lambda+\mu \in F$, $-\lambda \in F$, and $\lambda\circ\mu
\in F$.
\item \label{ringeasy:it2}
$(F,+,-,\circ,0,1)$ is a division ring.
\end{enumerate}
\end{lm}

\begin{proof}
\eqref{ringeasy:it1} Let $\lambda,\mu \in F$.
Clearly $\lambda\circ\mu\in F$.   Put $\sigma := \lambda+\mu$.
Fix a $\varphi$-class $\thcl$, let $e$ be the unique element of $\thcl \cap T$,
and observe that $\lambda(\thcl)\subseteq \thcl$ and $\mu(\thcl)\subseteq\thcl$
while $0(\thcl)=\{e\}$.
Thus $\sigma(\thcl) = d(\lambda(\thcl),e,\sigma(\thcl))\subseteq \thcl$
and $\sigma(e)=d(\lambda(e),e,\sigma(e)) = d(e,e,e)=e$, which proves
$\sigma(\varphi)\subseteq \varphi$ and $\sigma|_T=\mathrm{id}_T$.

To prove $\sigma \in \End(\m d)$, let $f$ be a $k$-ary basic operation
and let $\tup a=(a_1,\ldots,a_k)\in D^k$.  Let $a'=f(\tup a)$, for each $
i \in [k]$ let $e_i=0(a_i) \in T$ and let $e' = 0(a') \in T$.  Thus
$(a',e') \in \varphi$ and $(a_i,e_i) \in \varphi$ for all $i\in [k]$.
Observe that $f(\tup e) \stackrel\varphi\equiv f(\tup a) 
=a'\stackrel\varphi\equiv e'$ and
$f(\tup e) \in T$ (since $T\leq \m d$), which implies $f(\tup e)=e'$.
Let $\lambda(\tup a):= (\lambda(a_1),\ldots,\lambda(a_k))$ and
$\mu(\tup a):= (\mu(a_1),\ldots,\mu(a_k)$.
Since $\varphi$ is abelian, we can apply 
Proposition~\ref{prp:TC2} to $f$ and $\lambda(a_i)
\stackrel{\varphi}\equiv e_i \stackrel{\varphi}\equiv \mu(a_i)$ for
$i \in [k]$ to get
\[
d(f(\lambda(\tup a)),f(\tup e),f(\mu(\tup a)))
= f(d(\lambda(a_1),e_1,\mu(a_1),\ldots,d(\lambda(a_k),e_k,\mu(a_k)))).
\]
Using $f(\lambda(\tup a))=\lambda(f(\tup a))$ and $f(\mu(\tup a))=
\mu(f(\tup a))$, the above equation simplifies to
\[
\sigma(f(\tup a))= f(\sigma(a_1),\ldots,\sigma(a_k)),
\]
proving $\sigma \in \End(\m d)$.  Thus $\sigma \in F$, and a similar proof
shows $-\lambda \in F$.  

\eqref{ringeasy:it2}
Let $\lambda,\mu \in F$.
Observe that for each $\varphi$-class $\thcl$, if $e$ is the unique element
of $\thcl\cap T$, then $\lambda\restrict \thcl$ and $\mu\restrict \thcl$
are endomorphisms of the abelian
group $\Grp d(\varphi,e)$, $0\restrict \thcl$ is the constant zero endomorphism
of this group, and $(\lambda+\mu)\restrict \thcl$ is the pointwise sum of
$\lambda\restrict \thcl$ and $\mu\restrict \thcl$ in this group.  Thus
the map $\Phi:F\ra \prod_{e \in T}\End(\Grp d(\varphi,e))$ given by
$\Phi(\lambda)_e = \lambda\restrict\thcl$ where $\thcl=e/\varphi$, is a unital
embedding of $(F,+,-,\circ,0,1)$ into the direct product of the endomorphism
rings of the groups $\Grp d(\varphi,e)$ for $e \in T$.  Hence
$(F,+,-,\circ,0,1)$ is a unital ring.  

It remains to show that each
element of $F\setminus \{0\}$ is invertible.
Let $\lambda \in F\setminus \{0\}$.  One can easily check that
$T\leq \ran(\lambda)\leq \m d$ and $T\leq \lambda^{-1}(T)\leq \m d$.
Observe that $T$ is a maximal proper subuniverse of $\m d$ by Lemma~\ref{lm:poly}.
Because $\lambda\ne 0$, we get $\ran(\lambda)\ne T$  and 
$\lambda^{-1}(T)\ne D$.  Hence $\ran(\lambda)=D$ and
$\lambda^{-1}(T)=T$.  Thus $\lambda$ is surjective, and $\lambda\restrict \thcl$
has trivial kernel for each $\varphi$-class $\thcl$, so $\lambda$ is injective 
as well.  Letting $\lambda^{-1}$ be the functional inverse of $\lambda$, we can
easily check that $\lambda^{-1} \in F$.
\end{proof}

\begin{df}
Given Assumption~\ref{assum:easycase}, we 
denote the division ring $(F,+,-,\circ,0,1)$ just constructed by 
$\bbF_{\m d,\varphi,T}$.
\end{df}

\begin{cor}
Given Assumption~\ref{assum:easycase}, let $\bbF:= \bbF_{\m d,\varphi,T}$.
For each $e \in T$, 
the group $\Grp d(\varphi,e)$ becomes  
a left $\bbF$-vector space $\FGrp d(\varphi,e)$
via the action
$\lambda\cdot a := \lambda(a)$.
\end{cor}

\subsection{The general case}
In this subsection we revoke Assumption~\ref{assum:easycase}.
We will instead assume only that $\calV$ is a variety with a weak
difference term, $\m a$ is an algebra in $\calV$, and $\theta$ is an abelian
minimal congruence of $\m a$.  
Let $\m d=D(\m a,\theta)$ be the difference algebra for $\theta$, let
$\varphi = \baralpha/\Delta$ be the derived congruence, where 
$\alpha=(0:\theta)$ and $\Delta=\Delta_{\theta,\alpha}$, and let
$\Do = \set{0_\alcl}{$\alcl\in A/\alpha$}$ be the canonical transversal.
By Corollary~\ref{cor:D(A)}, $\m d$ is subdirectly irreducible, $\varphi$
is its monolith, $\varphi$ is abelian,
and $\Do$ is a transversal for $\varphi$ which is a subuniverse of $\m d$.
Thus we can appeal to the previous subsection to obtain the division ring
$\bbF_{\m d,\varphi,\Do}$.
In this section we will explain how $\bbF_{\m d,\varphi,\Do}$ acts naturally
on each group $\Grp a(\theta,e)$ for any $e \in A$.  The next lemma
explains the crucial fact needed.

\begin{lm} \label{lm:ran-sub}
Suppose $\calV$ is a variety with a weak difference term, $\m a\in \calV$,
$\theta$ is an abelian minimal congruence of $\m a$, and 
$\alpha=(0:\theta)$.  Let 
$\m d:=D(\m a,\theta)$
be the difference algebra for $\theta$, let $\varphi$ be the derived congruence,
let $\Do$ be the canonical transversal, and let $\bbF:= \bbF_{\m d,\varphi,\Do}$
be the division ring defined in the previous subsection.  Let $\alcl$ be an
$\alpha$-class.  For every $\theta$-class $\thcl$ contained in $\alcl$,
$\ran(\thcl)$ is an $\bbF$-subspace of the vector space 
$\FGrp d(\varphi,0_\alcl)$.
\end{lm}

\begin{proof}
Let $\Delta=\Delta_{\theta,\alpha}$.
Fix an $\alpha$-class $\alcl$ and a $\theta$-class $\thcl\subseteq \alcl$.
By Definition~\ref{df:Brange}, we only need to show that $\ran(\thcl)$ is
closed under the action of $\bbF$.
Let $a,e \in \thcl$, so $q:=(a,e)/\Delta$ is a typical element of $\ran(\thcl)$.
Let $\lambda \in \bbF$.  
If $a=e$, then $q = 0_\alcl \in \Do$ and so $\lambda(q)=q \in
\ran(\thcl)$ as required.  For the remainder of the proof, assume $a\ne e$.

Choose $(b,u) \in \theta$ with $\lambda((a,e)/\Delta)=(b,u)/\Delta$ and 
let $\thcl_0 = u/\theta$.
By Corollary~\ref{cor:directed}, we can choose a $\theta$-class $\thcl'$
with $\ran(\thcl)\cup \ran(\thcl_0)\subseteq \ran(\thcl')$.  Fix 
$e' \in \thcl'$ and recall that $\ran(\thcl')=\ran(\embed {e'})$.
Hence there exist $a',b' \in \thcl'$ with $(a',e')/\Delta = (a,e)/\Delta = q$
and $(b',e')/\Delta = (b,u)/\Delta = \lambda(q)$.

Since $(a,e)\stackrel\Delta\equiv (a',e')$ and $a\ne e$, we get $a'\ne e'$
by Theorem~\ref{thm:D(A)}\eqref{D(A):it1}.  
Thus $(a,e)\in \cg(a',e')$ by minimality of $\theta$, and so 
Corollary~\ref{cor:maltsev}
gives a polynomial $f \in \Pol_1(\m a)$ with $f(a')=a$ and $f(e')=e$.
Observe that $f(\thcl')\subseteq \thcl$.

Let $b^\ast = f(b') \in \thcl$.  We will show $\lambda(q)=(b^\ast,e)/\Delta$,
which will finish the proof.  Choose a term $t(x,\tup y)$ and parameters
$c_1,\ldots,c_n \in A$ so that $f(x)=t(x,\tup c)$.  For each $i\in [n]$
let $\alcl_i=c_i/\alpha$.  
Then calculating in $\m d$,
\begin{align*}
t(q,0_{\alcl_1},\ldots,0_{\alcl_n}) &= 
t((a',e')/\Delta,(c_1,c_1)/\Delta,\ldots,(c_n,c_n)/\Delta)\\
&= (t(a',\tup c),t(e',\tup c))/\Delta\\
&= (a,e)/\Delta=q.
\end{align*}

Recall that $\lambda \in \bbF$; in particular, $\lambda \in \End(\m d)$
and $\lambda(0_{\alcl'})=0_{\alcl'}$ for every $\alpha$-class $\alcl'$.
Thus
\begin{align*}
\lambda(q) &= \lambda(t(q,0_{\alcl_1},\ldots,0_{\alcl_n}))\\
&= t(\lambda(q),0_{\alcl_1},\ldots,0_{\alcl_n}) \\
&= t((b',e')/\Delta,(c_1,c_1)/\Delta,\ldots,(c_n,c_n)/\Delta)\\
&= (t(b',\tup c),t(e',\tup c))/\Delta\\
&= (b^\ast,e)/\Delta
\end{align*}
as required.
\end{proof}

As a consequence, in the context of Lemma~\ref{lm:ran-sub}, we can ``translate"
the $\bbF$-vector space action of the vector space $\FGrp d(\varphi,0_\alcl)$
to the abelian groups $\Grp a(\theta,e)$ via $\embed e$ and
$\embed e^{-1}$, for each $e \in \alcl$.

\begin{cor} \label{cor:ran-sub}
Suppose $\calV$ is a variety with a weak difference term, $\m a\in \calV$,
$\theta$ is an abelian minimal congruence of $\m a$, and 
$\alpha=(0:\theta)$.  Let 
$\m d:=D(\m a,\theta)$
be the difference algebra for $\theta$, let $\varphi$ be the derived congruence,
let $\Do$ be the canonical transversal, and let $\bbF:= \bbF_{\m d,\varphi,\Do}$.  For every $\theta$-class $\thcl$ and every $e \in \thcl$, the rule
\[
\lambda\cdot a := (\embed e^{-1}\circ\lambda\circ\embed e)(a),\quad
\lambda \in \bbF,~a\in \thcl
\]
defines a left $\bbF$-vector space action on $\Grp a(\theta,e)$.
With respect to this action, $\embed e$ is a left $\bbF$-vector space
embedding of $\FGrp a(\theta,e)$ into $\FGrp d(\varphi,0_\thcl)$.
\end{cor}

\begin{df}
Suppose $\calV$ is a variety with a weak difference term, $\m a \in \calV$,
and $\theta$ is an abelian minimal congruence of $\m a$.  
\begin{enumerate}
\item
The \emph{division
ring associated to $\theta$} is the division ring $\bbF_\theta:= 
\bbF_{\m d,\varphi,\Do}$ specified in the previous corollary.
\item
For each $e \in A$, the \emph{canonical
action of $\bbF_\theta$ on $\Grp a(\theta,e)$} is the action defined in the
previous corollary.
\end{enumerate}
\end{df}

The following result will be useful.

\begin{prp} \label{prp:samefield}
Suppose $\calV$ is a variety with a weak difference term, $\m a,\m b \in \calV$,
and $\theta \in \Con(\m a)$, $\kappa \in \Con(\m b)$
are such that $0\prec \theta$, $0\prec \kappa$, and $\theta,\kappa$
are both abelian.  If $D(\m a,\theta)\cong D(\m b,\kappa)$, then
$\bbF_\theta \cong \bbF_\kappa$.
\end{prp}

\begin{proof}
Let $\varphi_{\m a}$ be the monolith of $D(\m a,\theta)$,
let $\varphi_{\m b}$ be the monolith of $D(\m b,\kappa)$,
and
let $\psi:D(\m a,\theta)\cong D(\m b,\kappa)$ be an isomorphism.  
Let $\Do_{\m a}$ and $\Do_{\m b}$ be the canonical transversals for 
$\varphi_{\m a}$ and $\varphi_{\m b}$ respectively.
Clearly $\psi(\varphi_{\m a})=\varphi_{\m b}$ since the monolith of
any subdirectly irreducible algebra is a characteristic congruence
(i.e., preserved by all automorphisms).
Hence $\psi(\Do_{\m a})$ is a transversal for $\varphi_{\m b}$.
By Lemma~\ref{lm:aut}, there exists an automorphism $\sigma\in
\Aut(D(\m b,\kappa))$ satisfying $\sigma(\psi(\Do_{\m a}))=
\Do_{\m b}$.  Thus $\sigma\circ\psi:(D(\m a,\theta),\varphi_{\m a},\Do_{\m a})
\cong (D(\m b,\kappa),\varphi_{\m b},\Do_{\m b})$.  Since these triples are
exactly the data used to define $\bbF_\theta$ and $\bbF_\kappa$, it follows that
the map $\lambda \mapsto (\sigma\circ\psi)\circ \lambda \circ 
(\sigma\circ\psi)^{-1}$ is an isomorphism from $\bbF_\theta$ to $\bbF_\kappa$.
\end{proof}

Note that if $\gamma,\theta\in \Con(\m a)$ with $\gamma\prec
\theta$ and $\theta/\gamma$ abelian, then the construction 
above provides a division ring $\bbF_{\theta/\gamma}$ associated
to the abelian minimal congruence $\theta/\gamma$ of $\m a/\gamma$.
Thus every abelian cover pair of congruences (of an algebra
in a variety with a weak difference term) can be assigned a division ring.
It turns out that the division rings assigned to perspective abelian
cover pairs in this way are isomorphic.

\begin{thm} \label{thm:persp'}
Suppose $\calV$ is a variety with a weak difference term, $\m a \in \calV$,
and $\gamma,\theta,\delta,\varepsilon \in \Con(\m a)$ with 
$\gamma<\theta$
and $\delta< \varepsilon$, $\theta/\gamma$ and $\varepsilon/\delta$
are abelian, and $(\gamma,\theta)
\nearrow (\delta,\varepsilon)$.
Let $\alpha = (\gamma:\theta)$.
\begin{enumerate}
\item \label{persp':it1}
The rule 
\[
(a/\gamma,b/\gamma)/\Delta_{\theta/\gamma,\alpha/\gamma} \mapsto
(a/\delta,b/\delta)/\Delta_{\varepsilon/\delta,\alpha/\delta}
\quad
\mbox{for}~(a,b)\in \theta
\]
is a well-defined isomorphism $D(\m a/\gamma,\theta/\gamma)\cong
D(\m a/\delta,\varepsilon/\delta)$.
\item \label{persp':it2}
If $\gamma\prec\theta$ and $\delta\prec\varepsilon$, then
$\bbF_{\theta/\gamma}\cong \bbF_{\varepsilon/\delta}$.
\end{enumerate}
\end{thm}

\begin{proof}
Observe that 
$(0:\theta/\gamma) = \alpha/\gamma$  
by Lemma~\ref{lm:quotient}\eqref{quotient:it1}, so we may assume
with no loss of generality that $\gamma=0$. 
Let $\barma$, $\barvarepsilon$ and $\baralpha$ denote
$\m a/\delta$, $\varepsilon/\delta$ and 
$\alpha/\delta$ respectively.
Note that
$(0:\barvarepsilon)=\baralpha$ 
by Lemma~\ref{lm:quotient}\eqref{quotient:it1} 
and Proposition~\ref{prp:ann}, so $D(\barma,\barvarepsilon)=\barma(\barvarepsilon)/\Delta_{\barvarepsilon,\baralpha}$.  

\eqref{persp':it1}
Let $\psi:\m a(\theta)\ra \barma(\barvarepsilon)$ be the homomorphism
given by
by $\psi((a,b)) = (a/\delta,b/\delta)$, let
$\nu:\barma(\barvarepsilon)\ra 
D(\barma,\barvarepsilon)$ be the natural projection map, and let
$\Psi=\nu\circ\psi$.
Since $\varepsilon = \delta\circ\theta\circ\delta$ by 
Proposition~\ref{prp:ann}, $\psi$ is surjective, so $\Psi$ is also surjective.
We must show that $\ker(\Psi)=\Delta_{\theta,\alpha}$.  Let $\Delta=\Delta_{\theta,\alpha}$. Because $\varepsilon=\delta\circ\theta\circ\delta$, we have 
\begin{equation} \label{persp':eq1}
M(\barvarepsilon,\baralpha) = \left\{\begin{pmatrix}
a/\delta & a'/\delta\\b/\delta & b'/\delta\end{pmatrix}:
\begin{pmatrix} a&a'\\b&b'\end{pmatrix}\in M(\theta,\alpha) \right\}.
\end{equation}
Hence $\Delta\subseteq \ker(\Psi)$.
Before proving the opposite inclusion, we establish the following.

\medskip\noindent\textsc{Claim.}
If $(a,b),(r,s) \in \theta$ and $(a,r),(b,s) \in \delta$, then
$(a,b) \stackrel{\Delta}\equiv (r,s)$.

\begin{proof}[Proof of Claim]
Let
$d$ be a weak difference term for $\calV$.
Note first that $d(a,r,r) \stackrel\theta\equiv d(b,s,r)$ and
$d(a,r,r) \stackrel{\delta}\equiv d(a,a,a) = a= d(b,b,a) \stackrel\delta\equiv
d(b,s,r)$.  Since $\theta\cap\delta=0$, this proves
\begin{equation} \label{persp':eq2}
d(a,r,r) = d(b,s,r).
\end{equation}
Next, note that 
\begin{equation} \label{persp':eq3}
d(\underline{r},r,r) = r = d(\underline{r},s,s),
\end{equation}
so since $(r,s) \in \theta$ and 
$(r,b) \in \delta\vee\theta = \varepsilon\leq \alpha$ and
$C(\alpha,\theta;0)$, we can replace the underlined occurrences of ``$r$"
in equation~\ref{persp':eq3} by ``$b$" to get
\begin{equation} \label{persp':eq4}
d(b,r,r)  = d(b,s,s).
\end{equation}
Now we build some $(\theta,\alpha)$-matrices.
Let
\[
M_1 = \begin{pmatrix}a&a\\b&b\end{pmatrix},
~~
M_2 = \begin{pmatrix}a&r\\a&r\end{pmatrix},
~~
M_3 = \begin{pmatrix}b&s\\b&s\end{pmatrix},
~~
M_4 = \begin{pmatrix}s&s\\s&s\end{pmatrix},
~~
M_5 = \begin{pmatrix}r&r\\s&s\end{pmatrix}.
\]
Clearly $M_1,\ldots,M_5 \in M(\theta,\alpha)$.
Applying the weak difference term gives
\begin{align*}
N:=d(M_1,M_2,M_2) &= \begin{pmatrix} a & d(a,r,r)\\b & d(b,r,r)\end{pmatrix} \in
M(\theta,\alpha)\\
N':=d(M_3,M_4,M_5) &= \begin{pmatrix} d(b,s,r)&r\\d(b,s,s)&s\end{pmatrix} \in
M(\theta,\alpha).
\end{align*}
The second column of $N$ and the first column of $N'$ are identical
by equations~\eqref{persp':eq2} and~\eqref{persp':eq4}.
Since $\Delta$ is the ``horizontal transitive closure" of
$M(\theta,\alpha)$, 
this proves $(a,b)\stackequiv{\Delta}(r,s)$, proving 
the Claim.
\end{proof}

Now we prove $\ker(\Psi)\subseteq \Delta$.
Assume $((a,b),(a',b')) \in \ker(\Psi)$; that is, 
\[
\mbox{$(a,b),(a',b') \in \theta$
and $(a/\delta,b/\delta) \stackrel{\Delta_{\barvarepsilon,\baralpha}}\equiv
(a'/\delta,b'/\delta)$.}
\]
By definition of $\Delta_{\barvarepsilon,\baralpha}$ and 
equation~\ref{persp':eq1}, there exist $n\geq 0$ and
\begin{equation} \label{persp':eq5}
\begin{pmatrix} r_i&a_i\\s_i&b_i\end{pmatrix} \in M(\theta,\alpha)\quad
\mbox{for $i\leq n$}
\end{equation}
such that 
$(a/\delta,b/\delta)=(r_0/\delta,s_0/\delta)$, 
$(a_i/\delta,b_i/\delta) = (r_{i+1}/\delta,s_{i+1}/\delta)$ for all $i<n$, and
$(a'/\delta,b'/\delta)=(a_n/\delta,b_n/\delta)$.
The Claim then gives
$(a,b)\stackrel{\Delta}\equiv (r_0,s_0)$, 
$(a_i,b_i)\stackrel{\Delta}\equiv (r_{i+1},s_{i+1})$ for 
$i<n$, and 
$(a_n,b_n)\stackrel{\Delta}\equiv (a',b')$. 
As equation~\ref{persp':eq5} gives $(r_i,s_i) \stackrel\Delta\equiv (a_i,
b_i)$ for $i<n$, we get $(a,b)\stackrel\Delta\equiv(a',b')$ by transitivity
as required.

\eqref{persp':it2} follows from \eqref{persp':it1} and Proposition~\ref{prp:samefield}.
\end{proof}

\section{Similarity} \label{sec:sim}

In his 1983 paper \cite{freese1983}, Freese introduced a relation 
of \emph{similarity} between subdirectly irreducible algebras
in a congruence modular variety.  If the monoliths of the algebras
are nonabelian, then the algebras are similar if and only if they are 
isomorphic.  If the monoliths are abelian, then the relation is more
interesting.  Freese's original definition in the latter case, as well as
the equivalent formulation found in \cite{freese-mckenzie}, 
relied heavily on having a difference term, and does not easily adapt to
varieties with a weak difference term.

Happily, Freese gave two other equivalent formulations of the similarity
relation in the abelian monolith case, one of which is based 
on the construction in Section~\ref{sec:update}.  Both formulations
generalize nicely to varieties with a weak difference term.
We develop this generalization in this section.

The following definition, in the context of congruence modular varieties, is found in
\cite[Definition 10.9]{freese-mckenzie}.  

\begin{df} \label{df:D(A)}
Suppose $\calV$ is a variety with a weak difference term, and
$\m a \in \calV$ is subdirectly irreducible with monolith $\mu$.  The algebra $D(\m a)$ is
defined as follows:
\begin{enumerate}
\item
If $\mu$ is nonabelian, then $D(\m a)=\m a$.
\item
If $\mu$ is abelian, then $D(\m a)$ is $D(\m a,\mu)$, the difference algebra for $\mu$.
\end{enumerate}
\end{df}

Observe that, in the context of the previous definition, $D(\m a)$ is a subdirectly irreducible
algebra in $\HS(\m a^2)$; furthermore, the monoliths of $\m a$ and $D(\m a)$ 
are either both abelian or both nonabelian.

\begin{df} \label{df:similar}
Suppose $\calV$ is a variety with a weak difference term, and
$\m a,\m b \in \calV$ are subdirectly irreducible.  We say that $\m a$ and $\m b$ are
\emph{similar}, and write $\m a\sim \m b$, if $D(\m a)\cong D(\m b)$.
\end{df}

By \cite[Theorem 5.1]{freese1983} and \cite[Theorem~10.11]{freese-mckenzie},
our definition of similarity, when
restricted to the congruence modular setting, is equivalent to the original definition given
in \cite{freese1983} and \cite{freese-mckenzie}.  It is in this sense that our definition
extends the original one.

If we let $\mu$ and $\montwo$ denote the respective monoliths of 
$\m a$ and $\m b$, then
Definition~\ref{df:similar} can be rephrased as follows: $\m a\sim \m b$ if and only if one of the following holds:
\begin{enumerate}
\item
$\mu$ and $\montwo$ are both nonabelian and $\m a\cong \m b$; or
\item
$\mu$ and $\montwo$ are both abelian and $D(\m a,\mu)\cong D(\m b,\montwo)$.
\end{enumerate}

Note also that, 
if the monoliths $\mu$ and $\montwo$ are
abelian and $\m a\sim \m b$, then 
$\bbF_\mu\cong \bbF_\montwo$ 
by Proposition~\ref{prp:samefield}.

Before stating the main theorem of this section, we introduce 
a variant of a  crucial construct due to Zhuk \cite[Section 6.6]{zhuk2020}.

\begin{df} \label{df:bridge}
Suppose $\m a,\m b$
are subdirectly irreducible algebras in a common signature with monoliths
$\mu,\montwo$ respectively.  A \emph{\proper\ bridge} from $\m a$ to $\m b$
is a subuniverse $T\leq\m a\times \m a\times \m b\times \m b$ satisfying
\begin{enumerate}[label=(B\arabic*)]
\item \label{bridge:it1}
$\proj_{1,2}(T)=\mu$ and $\proj_{3,4}(T)=\montwo$.
\item \label{bridge:it2}
For all $(a_1,a_2,b_1,b_2) \in T$ we have $a_1=a_2$ if and only if $b_1=b_2$.
\item \label{bridge:it3}
For all $(a_1,a_2,b_1,b_2)\in T$ we have $(a_i,a_i,b_i,b_i) \in T$ for $i=1,2$.
\end{enumerate}
If $T$ is a \proper\ bridge from $\m a$ to $\m b$, then we define its \emph{trace}
to be the set $\tr(T) = \set{(a,b)\in A\times B}{$(a,a,b,b) \in T$}$, and
its \emph{kernel} to be the set 
\[
\ker(T) = 
\set{((a_1,b_1),(a_2,b_2))}{$(a_1,a_2,b_1,b_2) \in T$}.
\]
\end{df}

Observe that for any \proper\ bridge $T$ we have
$\tr(T) = \proj_{1,3}(T) = \proj_{2,4}(T)$ by property \ref{bridge:it3}, 
and hence $\tr(T)\leq_{sd}\m a\times \m b$ by property
\ref{bridge:it1}.

In the special case where one of the two monoliths 
is nonabelian,
the existence of a \proper\ bridge is possible only in the most trivial way.

\begin{prp} \label{prp:onenonabel}
Suppose $\calV$ is a variety with a weak difference term, $\m a,\m b
\in \calV$ are subdirectly irreducible algebras with monoliths
$\mu,\montwo$ respectively, and $T$ is a \proper\ bridge
from $\m a$ to $\m b$.  
Assume that one of $\mu$ or $\montwo$ is nonabelian. Then so is the other, 
$\tr(T)$ is the graph of an isomorphism $h:\m a\cong\m b$, and
$T = \set{(a,b,h(a),h(b))}{$(a,b)\in \mu$}$.
\end{prp}

\begin{proof}
Assume with no loss of generality that 
$\montwo$ is nonabelian.  We will first show that
$\tr(T)$ is the graph of a surjective homomorphism $h:\m a\ra \m b$.  
Let
\[
S = \set{(b,c) \in B^2}{$\exists a \in A$ with $(a,b),(a,c) \in \tr(T)$}.
\]
$S$ is a tolerance of $\m b$.  Let
\[
Y = \left\{\begin{pmatrix}b&c\\b'&c'\end{pmatrix} \in B^{2\times 2}~:~
\mbox{$\exists\, (a,a') \in \mu$ with $(a,a',b,b'),(a,a',c,c') \in T$}\right\}.
\]
One can easily show, using \ref{bridge:it1} and the definitions of $S$ and $Y$, that
$X(\montwo,S)\subseteq Y$ 
(see Definition~\ref{df:M}).  Since $Y$ is a subuniverse of $\m b^{2\times 2}$, we 
get $M(\montwo,S)\subseteq Y$.  Then from \ref{bridge:it2} and 
Definition~\ref{df:cent}\eqref{dfcent:it2} we can deduce $C(S,\montwo;0_B)$.
Since $S$ is a tolerance, we can improve this to $C(\cg(S),\montwo;0_B)$ by
\cite[Theorem 2.19(2)]{kearnes-kiss}.  Now assume that $S\ne 0_B$; then $\cg(S)\geq \montwo$
since $\montwo$ is the monolith of $\m b$; but then $\montwo$ would be
abelian, contrary to our assumption.  This proves $S=0_B$ and hence $\tr(T)$ is the graph
of a surjective homomorphism $h:\m a\ra \m b$.

Assume that $h$ is not injective.  Then $\mu\subseteq \ker(h)$ as $\mu$ is the
monolith of $\m a$.
Choose $(a,b) \in \mu$ with $a\ne b$.  By \ref{bridge:it1}, there exists $(c,d) \in
\montwo$ with $(a,b,c,d) \in T$.  Then $c\ne d$ by \ref{bridge:it2}.  By \ref{bridge:it3},
we have $(a,c),(b,d) \in \tr(T)$, or equivalently, $c=h(a)$ and $d=h(b)$.  
But then $(a,b) \in \mu
\subseteq \ker(h)$ implies $c=d$, contradiction.  Thus $h$ is an isomorphism.

Finally, for every $(a,b) \in \mu$ there is at least one pair $(c,d) \in \montwo$ with
$(a,b,c,d) \in T$ by \ref{bridge:it1}. On the other hand, the only pair 
$(c,d)$ that can work is $(c,d)=(h(a),h(b))$ by \ref{bridge:it3} and the
definition of $h$.  Hence $T=\set{(a,b,h(a),h(b))}{$(a,b) \in \mu$}$.
\end{proof}

\begin{lm} \label{lm:bridge}
Suppose $\calV$ is a variety with a weak difference term, $\m a,\m b
\in \calV$ are subdirectly directly algebras with monoliths
$\mu,\montwo$ respectively, and $T$ is a \proper\ bridge
from $\m a$ to $\m b$.   Let $\m c=\tr(T)\leq_{sd}\m a\times \m b$
and $\tau = \ker(T)$.  Then $\tau \in \Con(\m c)$.  Hence
\begin{enumerate}[label=(B\arabic*)]
\setcounter{enumi}{3}
\item \label{bridge:it4}
For all $(a_1,a_2,b_1,b_2) \in T$ we have $(a_2,a_1,b_2,b_1)\in T$.
\item \label{bridge:it5}
If $(a_1,a_2,b_1,b_2) \in T$ and $(a_2,a_3,b_2,b_3) \in T$, then
$(a_1,a_3,b_1,b_3) \in T$.
\end{enumerate}
\end{lm}

\begin{proof}
Note that $\tau$ is a reflexive subuniverse of $\m c^2$ by \ref{bridge:it3}, 
and that
\ref{bridge:it4} and \ref{bridge:it5} simply express the symmetry and
transitivity of $\tau$.  Hence all we need to prove is that $\tau$ is symmetric
and transitive.

If $\mu$ or $\kappa$ is nonabelian, then $C$ is the graph of an isomorphism
$h:\m a\cong \m b$ and $\tau = \set{((a,h(a)),(b,h(b))}{$(a,b)\in \mu$}$
by Proposition~\ref{prp:onenonabel}.  Clearly $\tau$ is symmetric and transitive
 in this case.

If $\mu$ and $\kappa$ are both abelian,
let $d(x,y,z)$ be a weak difference term for $\calV$.  Let 
\[
\theta = 
\set{((a_1,b_1),(a_2,b_2)) \in C^2}{$(a_1,a_2) \in \mu$ and $(b_1,b_2) \in
\montwo$}.
\]
Clearly $\theta \in \Con(\m c)$.
Because $\mu$ and
$\montwo$ are abelian, the restriction of $d$ to any $\theta$-class is a Maltsev
operation.  
We have that $\tau$ is a reflexive subuniverse of 
$\m c^2$ and $\tau\subseteq \theta$ by \ref{bridge:it1}.
Thus $\tau$ is symmetric and transitive by Lemma~\ref{lm:maltsev}.
\end{proof}

\begin{lm} \label{lm:bridge2}
Suppose $\calV$ is a variety with a weak difference term, $\m a,\m b
\in \calV$ are subdirectly directly with monoliths $\mu,\montwo$ respectively, 
and $T$ is a \proper\ bridge
from $\m a$ to $\m b$.   Let $\m c=\tr(T)\leq_{sd}\m a\times \m b$
and $\tau = \ker(T)$.  
Also let
$\delta_1,\delta_2 \in \Con(\m c)$ be the kernels of the two projection
homomorphisms $\proj_1:\m c\ra \m a$ and $\proj_2:\m c\ra \m b$, and let
$\coverdelta_i$ be the unique upper cover of $\delta_i$ in $\Con(\m c)$
for $i=1,2$.
\begin{enumerate}
\item \label{bridge2:it1}
$(\delta_1,\coverdelta_1)\searrow (0,\tau) \nearrow (\delta_2,\coverdelta_2)$.
\item \label{bridge2:it2}
For all $(a,r),(b,s) \in \tr(T)$ we have $(a,b) \in (0:\mu) \iff (r,s) \in
(0:\montwo)$.
\item \label{bridge2:it3}
$\m a/(0:\mu) \cong \m b/(0:\montwo)$.
\end{enumerate}
\end{lm}

\begin{proof}
\eqref{bridge2:it1}
Clearly $\tau \leq \coverdelta_1 \wedge \coverdelta_2$ by \ref{bridge:it1}.
Choose $(a_1,a_2) \in \mu$ with $a_1\ne a_2$.  
Again by \ref{bridge:it1},
there exist $b_1,b_2 \in B$ with 
$(a_1,a_2,b_1,b_2) \in T$ and hence $(a_1,b_1)\stackequiv{\tau} (a_2,b_2)$,
which proves $\tau\nleq \delta_1$.
A similar argument shows $\tau\nleq\delta_2$.  This proves
$\tau \vee \delta_i=\coverdelta_i$ for $i=1,2$.
Finally, suppose $((a_1,b_1),(a_2,b_2)) \in \tau\cap \delta_1$.  Then
$a_1=a_2$, so $b_1=b_2$ by \ref{bridge:it2}, so $(a_1,b_1)=(a_2,b_2)$.
This shows $\tau\wedge\delta_1=0$,
and a similar proof shows $\tau\wedge\delta_2=0$.
This 
proves $(\delta_1,\coverdelta_1)\searrow (0,\tau) \nearrow 
(\delta_2,\coverdelta_2)$ as required.

\eqref{bridge2:it2}
By \eqref{bridge2:it1} and 
Proposition~\ref{prp:ann}\eqref{ann:it2}
we get $(\delta_1:\coverdelta_1)=(0:\tau) = (\delta_2:\coverdelta_2)$.
On the other hand, letting $\proj_1:\m c\ra \m a$ be the first projection
homomorphism, we have 
$(\delta_1:\coverdelta_1) = (\proj_1^{-1}(0):\proj_1^{-1}(\mu))=
\proj_1^{-1}((0:\mu))$ by 
Lemma~\ref{lm:quotient}\eqref{quotient:it2}, and similarly $(\delta_2:\coverdelta_2)=\proj_2^{-1}((0:\montwo))$.
Together these facts prove $\proj_1^{-1}((0:\mu))=\proj_2^{-1}((0:\montwo))$, which
is equivalent to the claim in \eqref{bridge2:it2}.

\eqref{bridge2:it3}
Let $\alpha=(0:\mu)$ and $\beta = (0:\montwo)$.
The set $\barC:= \set{(a/\alpha,r/\beta)}{$(a,r) \in C$}$ clearly satisfies
$\barC\leq_{sd}\m a/\alpha\times \m b/\beta$, and is the graph of a bijection 
by \eqref{bridge2:it2}.
\end{proof}

Here is the main theorem of this section.  
The equivalence of items~\eqref{persp:it1}--\eqref{persp:it3}
was proved for congruence modular varieties
by Freese \cite[Theorems 2.1 and 5.1]{freese1983}; see also \cite[Theorems 10.8 and 10.11]{freese-mckenzie}.

\begin{thm} \label{thm:persp}
Suppose $\calV$ is a variety with a weak difference term, and
$\m a,\m b\in \calV$ are subdirectly irreducible.
The following are equivalent:
\begin{enumerate}
\item \label{persp:it1}
$\m a\sim \m b$.
\item \label{persp:it2}
There exist an algebra $\m c \in \calV$,
surjective homomorphisms $f_1:\m c\ra \m a$ and $f_2:\m c\ra \m b$,
 and congruences 
$\psi,\tau \in \Con(\m C)$ with $\psi<\tau$, such that, letting $\delta_i = \ker(f_i)$ and 
letting $\coverdelta_i$
denote the unique upper cover of $\delta_i$ in $\Con(\m c)$ for $i=1,2$, we have
$(\delta_1,\coverdelta_1) \searrow (\psi,\tau) \nearrow (\delta_2,\coverdelta_2)$.

\item \label{persp:it3}
There exist an algebra $\m c \leq_{sd}\m a\times \m b$ and a congruence
$\tau \in \Con(\m c)$ which, with 
the projection
homomorphisms $f_1=\proj_1:\m c\ra \m a$ and $f_2=\proj_2:\m c\ra \m b$ and $\psi:=0_C$,
 satisfy the conclusion of item \eqref{persp:it2}.
\item \label{persp:it4}
There exists a \proper\ bridge from $\m a$ to $\m b$.
\end{enumerate}
\end{thm}

\begin{proof}
Throughout this proof, let $\mu,\montwo$ denote the monoliths of $\m a,\m b$ respectively.
The implication \eqref{persp:it3}\Implies\eqref{persp:it2} is clear,
and \eqref{persp:it4}\Implies\eqref{persp:it3} holds by 
Lemma~\ref{lm:bridge2}\eqref{bridge2:it1}.
 We will prove
\eqref{persp:it2}\Implies\eqref{persp:it1}\Implies\eqref{persp:it4}
first in the case when one of $\mu$ or $\montwo$ is nonabelian, and then
in the case when both $\mu$ and $\montwo$ are abelian.

\medskip\noindent\textsc{Case 1}: $\mu$ or $\montwo$ is nonabelian.

\medskip
\eqref{persp:it1}\Implies\eqref{persp:it4}.  
Since $\m a\sim\m b$, it follows that both $\mu$ and $\montwo$ are nonabelian
and there exists
an isomorphism $h:\m a\cong \m b$.  Then the set $T=\set{(a,b,h(a),h(b))}{$(a,b)
\in \mu$}$ is a \proper\ bridge from $\m a$ to $\m b$.

\medskip
\eqref{persp:it2}\Implies\eqref{persp:it1}.  
Assume there exist $\m c \in \calV$,
surjective homomorphisms $f_1:\m c\ra \m a$ and $f_2:\m c\ra \m b$,
 and congruences 
$\psi,\tau \in \Con(\m C)$ with $\psi<\tau$, such that, letting $\delta_i = \ker(f_i)$ and 
letting $\coverdelta_i$
denote the unique upper cover of $\delta_i$ for $i=1,2$, we have
$(\delta_1,\coverdelta_1) \searrow (\psi,\tau) \nearrow (\delta_2,\coverdelta_2)$.
Because we are in Case 1, at least one of $\coverdelta_1/\delta_1$ or $\coverdelta_2/\delta_2$
is nonabelian.  By Proposition~\ref{prp:memoir3}\eqref{mem:3.27}, 
both are nonabelian, as is $\tau/\psi$.  

For each $i=1,2$ we have $\delta_i\wedge \tau=\psi$ and hence $C(\delta_i,\tau;\psi)$;
 see e.g.\ \cite[Theorem 2.19(8)]{kearnes-kiss}; thus $C(\delta_1\vee \delta_2,\tau;\psi)$
by semidistributivity of the centralizer relation in the first variable; see e.g.\ 
\cite[Theorem 2.19(5)]{kearnes-kiss}.  Suppose $\delta_1\ne \delta_2$; then without loss
of generality we can assume $\delta_1\vee\delta_2>\delta_2$.  Hence $\delta_1\vee\delta_2\geq
\coverdelta_2$ (as $\delta_2$ is completely meet-irreducible) and thus $\delta_1\vee\delta_2
\geq \tau$.  Then $C(\delta_1\vee\delta_2,\tau;\psi)$ implies $C(\tau,\tau;\psi)$ by 
monotonicity, which contradicts the fact that $\tau/\psi$ is nonabelian.  This proves
$\delta_1=\delta_2$ and hence $\m a \cong \m c/\delta_1=\m c/\delta_2\cong \m b$.  Hence
$\m a\sim \m b$.

\medskip
\noindent\textsc{Case 2}: $\mu$ and $\montwo$ are abelian.

Thus $D(\m a)=D(\m a,\mu)$ and $D(\m b)=D(\m b,\montwo)$.
For the remainder of this proof,
let $\alpha = (0:\mu)$, $\Delta_{\m a}=\Delta_{\mu,\alpha}$,
and $\Dmon_{\m a}=\baralpha_{\m a}/\Delta_{\m a}$.  Also
let $\Do_{\m a}\leq D(\m a)$ as given in
the proof of Corollary~\ref{cor:D(A)} for $(\m a,\mu)$.
Similarly define $\beta:=(0:\montwo)$, $\Delta_{\m b}$, $\Dmon_{\m b}$,
and $\Do_{\m b}$ for $(\m b,\montwo)$.

\medskip
\eqref{persp:it2}\Implies\eqref{persp:it1}
We have $D(\m a) \cong D(\m c/\delta_1,\delta_1^+/\delta_1)$ and
$D(\m b)\cong D(\m c/\delta_2,\delta_2^+/\delta_2)$ since $\m a\cong
\m c/\delta_1$ and $\m b\cong \m c/\delta_2$.  
We can replace $\tau$ with any $\tau' \in \Con(\m c)$ satisfying
$\psi\prec\tau'\leq \tau$ and the hypotheses of \eqref{persp:it2} still hold;
thus we may assume $\psi\prec\tau$.
We have that $\tau/\psi$ is abelian, since $(\psi,\tau)\nearrow (\delta_1,\delta_1^+)$ and $\delta_1^+/\delta_1$ is abelian.  Then by Theorem~\ref{thm:persp'}
we get $D(\m c/\delta_1,\delta_1^+/\delta_1) \cong D(\m c/\psi,\tau/\psi) \cong
D(\m c/\delta_2,\delta_2^+/\delta_2)$.  Hence $D(\m a)\cong D(\m b)$,
so $\m a\sim \m b$.

\medskip
\eqref{persp:it1}\Implies\eqref{persp:it4}.  
Assume $D(\m a)\cong D(\m b)$.
As in the proof of Proposition~\ref{prp:samefield}, there exists an
isomorphism $\lambda:D(\m a)\cong D(\m b)$ 
satisfying $\lambda(\varphi_{\m a})=\varphi_{\m b}$ and
$\lambda(\Do_{\m a})=\Do_{\m b}$.  
Define
\begin{align*}
T = \{(a_1,a_2,b_1,b_2) \in A\times A\times B\times B &:
~\mbox{$a_1\stackequiv{\mu}a_2$ and 
$b_1\stackequiv{\montwo}b_2$ and}\\
&\rule{.2in}{0in}
\mbox{$\lambda((a_1,a_2)/\Delta_{\m a}) = (b_1,b_2)/\Delta_{\m b}$}\}.
\end{align*}

Because $\lambda$ is an isomorphism, we get $T\leq \m a\times \m a\times
\m b\times \m b$ with $\proj_{1,2}(T)=\mu$ and 
$\proj_{3,4}(T)=\montwo$, proving~\ref{bridge:it1}.
Let $(a_1,a_2,b_1,b_2) \in T$ and assume $a_1=a_2$.  
Then $(a_1,a_2)/\Delta_{\m a} \in \Do_{\m a}$,  As $\lambda(\Do_{\m a})=
\Do_{\m b}$, we get $(b_1,b_2)/\Delta_{\m b}=\lambda((a_1,a_2)/\Delta_{\m a})
\in \Do_{\m b}$, so $b_1=b_2$ (see Theorem~\ref{thm:D(A)}\eqref{D(A):it1}).  A similar
argument shows that if $b_1=b_2$, then $a_1=a_2$.
This proves \ref{bridge:it2}.

Again suppose $(a_1,a_2,b_1,b_2) \in T$ and let 
$r=\lambda((a_1,a_1)/\Delta_{\m a})$ and $s=(b_1,b_1)/\Delta_{\m b}$.  
Observe that 
$(a_1,a_1)/\Delta_{\m a} \stackequiv{\Dmon_{\m a}}
(a_1,a_2)/\Delta_{\m a}$ and
$(b_1,b_1)/\Delta_{\m b} \stackequiv{\Dmon_{\m b}}
(b_1,b_2)/\Delta_{\m b}$.  
Then
\begin{align*}
r=\lambda((a_1,a_1)/\Delta_{\m a}) &\stackequiv{\Dmon_{\m b}}
\lambda((a_1,a_2)/\Delta_{\m a}) &\mbox{as $\lambda(\varphi_{\m a})=\varphi_{\m b}$}\\
&= (b_1,b_2)/\Delta_{\m b} &\mbox{as $(a_1,a_2,b_1,b_2)
\in T$}\\
&\stackequiv{\Dmon_{\m b}} (b_1,b_1)/\Delta_{\m b}=s,
\end{align*}
i.e., $(r,s) \in \Dmon_{\m b}$.
Also recall that $(a_1,a_1)/\Delta_{\m a} \in \Do_{\m a}$ and
$s=(b_1,b_1)/\Delta_{\m b} \in \Do_{\m b}$, and 
since $\lambda(\Do_{\m a})=\Do_{\m b}$
we also get $r=\lambda((a_1,a_1)/\Delta_{\m a})\in \Do_{\m b}$.  So
$r,s \in \Do_{\m b}$ and $(r,s)\in \Dmon_{\m b}$.  But $\Do_{\m b}$ is
a transversal for $\Dmon_{\m b}$.  These facts imply $r=s$,  which implies
$(a_1,a_1,b_1,b_1) \in T$.  A similar argument gives $(a_2,a_2,b_2,b_2)\in T$.
This proves \ref{bridge:it3}.  Thus $T$ is a \proper\ bridge from $\m a$ to $\m b$.
\end{proof}

Item~\eqref{simprop:it1} in the next result
extends the first claim of \cite[Theorem~10.10]{freese-mckenzie}
from congruence modular varieties to varieties with a weak difference term.
(The remainder of \cite[Theorem~10.10]{freese-mckenzie} was extended by Corollary~\ref{cor:D(A)}.)
Item~\eqref{simprop:it2} extends a fact established, in the congruence modular setting, within
the proof of \cite[Theorem~10.11]{freese-mckenzie}.

\begin{cor} \label{cor:simprop}
Suppose $\calV$ is a variety with a weak difference term, and
$\m a\in \calV$ is subdirectly irreducible with abelian monolith $\mu$.
\begin{enumerate}
\item \label{simprop:it1}
$\m a\sim D(\m a)$.
\item \label{simprop:it1.5}
Setting $\alpha = (0:\mu)$ and $\Delta = \Delta_{\mu,\alpha}$, the set 
\[
\TD a := \set{(a,b,(a,e)/\Delta,(b,e)/\Delta)}{$a\stackrel{\mu}{\equiv}
b \stackrel{\mu}{\equiv}e$}
\]
is a \proper\ bridge from $\m a$ to $D(\m a)$.
\item \label{simprop:it2}
If $(0:\mu)=\mu$ and 
there exists a subuniverse $A_0\leq \m a$ which is a transversal for
$\mu$, then $\m a\cong D(\m a)$.
\end{enumerate}
Hence each similarity class in $\calV$ consisting of subdirectly irreducible algebras
with abelian monoliths contains a unique (up to isomorphism)
 subdirectly irreducible 
algebra $\m s$ with abelian monolith $\Dmon$ satisfying (i)
$(0:\Dmon)=\Dmon$, and (ii) $\m s$ has a subuniverse which is a transversal
for $\Dmon$.
\end{cor}  

\begin{proof}
\eqref{simprop:it1}
Let $\eta_1,\eta_2$ be the kernels of the two projection maps
$\proj_1,\proj_2:\m a(\mu)\ra \m a$.
By Theorem~\ref{thm:D(A)}, if we set $\m b=D(\m a)$,
$\m c=\m a(\mu)$, $\delta_1 = \eta_1$,
$\delta_2 = \Delta$, $\coverdelta_1 = \barmu$, $\coverdelta_2 = \baralpha$,
$\psi=0$, and $\tau=\eta_2$, then the conditions of 
Theorem~\ref{thm:persp}\eqref{persp:it2} are met, showing $\m a\sim D(\m a)$.

\eqref{simprop:it1.5}
This can be deduced from the above proof of item~\eqref{simprop:it1}, following the proof
of Theorem~\ref{thm:persp} \eqref{persp:it2}\Implies\eqref{persp:it4}
in the abelian monolith case.
Instead, we give a direct proof.  
Let $\Dmon = \baralpha/\Delta$.
Clearly $\proj_{1,2}(\TD a)
= \mu$, $\proj_{3,4}(\TD a)\subseteq \Dmon$,
 \ref{bridge:it3} holds, and the \parenImplies\ direction of \ref{bridge:it2} holds.
To prove the \parenImpliedby\ direction of \ref{bridge:it2}, assume that
$a,b,e$ belong to a common $\mu$ class and $(a,e)\stackrel{\Delta}{\equiv} (b,e)$.
Using the notation of Definition~\ref{df:lambda-e}, we can express this as
$\embed e(a)=\embed e(b)$.  But $\embed e$ is injective by Lemma~\ref{lm:inject},
so $a=b$, completing the proof of \ref{bridge:it2}.

It remains to prove $\Dmon \subseteq \proj_{3,4}(\TD a)$.  
Let $((a_1,a_2)/\Delta,(b_1,b_2)/\Delta) \in \Dmon$
be given; thus $(a_1,a_2),(b_1,b_2) \in \mu$ and $a_1,a_2,b_1,b_2$ belong to a common
$\alpha$-class.  By Lemma~\ref{lm:arrow}, there exists $e \in a_1/\alpha$
with $\ran(\embed{a_2}) \cup \ran(\embed{b_2}) \subseteq \ran(\embed e)$.
As $(a_1,a_2)/\Delta \in \ran(\embed {a_2})$ and $(b_1,b_2)/\Delta \in 
\ran(\embed{b_2})$,
this implies the existence of $a,b \in e/\mu$ with $\embed e(a)=(a_1,a_2)/\Delta$ and
$\embed e(b)=(b_1,b_2)/\Delta$, which means
\[
((a_1,a_2)/\Delta,(b_1,b_2)/\Delta) = ((a,e)/\Delta,(b,e)/\Delta) \in \TD a.
\]

\eqref{simprop:it2}
Since $(0:\mu)=\mu$, we have $D(\m a)=\m a(\mu)/\Delta$ where $\Delta=\Delta_{\mu,\mu}$.
Since $\mu$ is abelian, we get $\Delta = \left\{
((a,b),(a',b')):{\scriptsize{\begin{pmatrix}a&a'\\b&b'\end{pmatrix}}} \in M(\mu,\mu)\right\}$
by Lemma~\ref{lm:faces}\eqref{faces:it1.5}.
Let $\pi:\m a\ra \m a_0$ be the homomorphism with 
$\pi\restrict{A_0}=\mathrm{id}_{A_0}$ and $\ker(\pi)=\mu$.  Then define
$\lambda:\m a\ra D(\m a)$ by $\lambda(x) = (x,\pi(x))/\Delta$.  
Suppose $(a,b) \in \ker(\lambda)$.  
Then $((a,\pi(a),(b,\pi(b)))\in \Delta$, so 
$M:={\scriptsize{\begin{pmatrix} a&b\\\pi(a)&\pi(b)
\end{pmatrix} \in M(\mu,\mu)}}$.  
Also $(a,b)\in \mu$, so $\pi(a)=\pi(b)$.  Thus the bottom entries of $M$ are equal.
Since $\mu$ is abelian, it follows that the top entries of $M$ are also equal,
i.e., $a=b$.  Thus $\lambda$ is injective.

To prove that $\lambda$ is surjective, let $(a,b) \in \mu$ be given and define
$e=\pi(a)=\pi(b)$ and $c = d(a,b,e)$ where $d(x,y,z)$ is a weak difference term for $\calV$.
Then $(a,b)/\Delta = (c,e)/\Delta=\lambda(c)$ as required.
\end{proof}

\section{Conclusion and questions} \label{sec:conclu}

In this paper we have extended Freese's analysis of abelian minimal
congruences, and his similarity relation between subdirectly irreducible
algebras, from the congruence modular setting to the setting of varieties
with a weak difference term.  We also established a connection 
between similarity and certain constructs which 
we have called ``\proper\ bridges."  In the companion paper \cite{bridges}, we 
will compare \proper\ bridges to Zhuk's original bridges (which are superficially
weaker) and explain some related elements of Zhuk's work in terms of centrality
and similarity.  In a third paper \cite{critical}, we 
will unite the analyses of 
``rectangular critical relations" by Zhuk
\cite[Lemma 8.21]{zhuk2020} (in the idempotent locally finite Taylor case)
and by
Kearnes and Szendrei \cite{kearnes-szendrei-parallel} (in the congruence
modular case), showing in particular the role of similarity, associated finite fields,
and linear equations. 

In the congruence modular setting, Freese and McKenzie \cite{freese-mckenzie} established
some nontrivial results concerning the similarity relation, 
such as bounding the number of similarity classes of
subdirectly irreducible algebras in finitely generated varieties.
One can ask whether any of these properties persist in varieties with a weak difference 
term, or with a difference term.  For example:

\begin{prb}
If $\m a$ is a finite algebra and $\HSP(\m a)$ has
a weak difference term, does it follow that there are only finitely many similarity classes
of subdirectly irreducible algebras in $\HSP(\m a)$ with abelian monolith?
\end{prb}

Finite subdirectly irreducible algebras with abelian monoliths in 
\emph{idempotent}
 locally finite
Taylor varieties are of particular interest.  We have seen one result
(Lemma~\ref{lm:idem}) showing how the assumption of idempotency 
can restrict the possible sizes of monolith classes.
In what other ways does idempotency affect the sizes of monolith classes?  

\begin{prb}
Suppose $\m a$ is a finite idempotent subdirectly irreducible 
algebra in a Taylor variety with abelian monolith $\mu$,
$\alpha = (0:\mu)$, and $\alcl$ is an $\alpha$-class.  If $q$ is the maximal size of a
$\mu$-class in $\alcl$, does it follow that every $\mu$-class in $\alcl$ has size $q$ or 1?
\end{prb}

When $\m a$ is a finite algebra in a variety with a weak difference term
and $\theta$ is an abelian minimal congruence,
tame congruence theory tells us that $\typ_{\m a}(0,\theta)=\tup 2$ and
therefore the minimal algebras $\m a|_N$ associated to $\langle 0,
\theta\rangle$-traces are weakly isomorphic to a one-dimensional vector 
space over a finite field $\bbF$ determined by $\theta$.

\begin{prb}
Is the finite field $\bbF$ mentioned above isomorphic to $\bbF_\theta$?  If not,
how are $\bbF$ and $\bbF_\theta$ related?
\end{prb}

\phantomsection
\hypertarget{A1}{}

\section*{Appendix 1: One division ring}

Let $\m a$ be an algebra in a variety with a weak difference term,
let $\theta$ be an abelian minimal congruence, and let
$\bbF_\theta$ be the division ring
constructed in Section~\ref{sec:divring}.
In this appendix we will show that 
$\bbF_\theta$ is isomorphic to another division ring $\bbD$ associated to $\theta$,
which Freese \cite[Section 4, page 150]{freese1983} 
described in the congruence modular setting and whose construction easily adapts
to algebras having a weak difference term.
We also show that when $\m a$ is finite, $\bbF_\theta$ is isomorphic
to the finite
field $\bbF$ arising in the proof sketch of Theorem~\ref{thm:finite-vecspace}.

First we describe our adaptation of Freese's construction.
Assume only that $\m a$ is an algebra having a weak difference term $d$,
and $\theta$ is an abelian minimal congruence.
Let the $\theta$-classes be indexed
by a set $I$ via the notation $A/\theta = \set{\thcl_i}{$i \in I$}$.
Let $T$ be a transversal for $\theta$ and for each $i \in I$ let
$e_i$ be the unique element of $T\cap \thcl_i$.  For $i,j \in I$ define
\[
R_{ij} = \set{f\restrict{\thcl_j}}{$f \in \Pol_1(\m a)$
and $f(e_j)=e_i$}.
\]
(Freese denoted this set by $\Hom(\theta,e_j,e_i)$.)
Also define $O_{ij} \in R_{ij}$ to be the constant map $\thcl_j\ra \{e_i\}$.
Let $\calR$ be the set of $I\times I$-matrices $(r_{ij})_{i,j \in I}$
where $r_{ij} \in R_{ij}$ for all $i,j \in I$, and for each $j \in I$,
$r_{ij} = O_{ij}$ for all but finitely many $i \in I$.
Let $V$ denote the direct sum of the abelian groups 
$\Grp a(\theta,e_i)$, $i \in I$,
realized as the subgroup of $\prod_{i \in I}\Grp a(\theta,e_i)$ 
consisting
of those $(a_i)_{i \in I} \in \prod_{i \in I}\thcl_i$ for which 
$a_j=e_j$ for all but finitely many $j \in I$.  
Observe that each $f \in R_{ij}$ is a group homomorphism from
$\Grp a(\theta,e_j)$ to $\Grp a(\theta,e_i)$, by Proposition~\ref{prp:affine};
that $R_{ij}$ is closed under the pointwise addition operation in
$\Grp a(\theta,e_i)$; and that $f \in R_{ij}$, $g \in R_{jk}$ imply
$f\circ g \in R_{ik}$.

Thus $\calR$ acts in the usual ``matrix" way on $V$: 
given $M=(r_{ij})_{i,j \in I} \in \calR$,
define $\sfL_M:V\ra V$ so that $\sfL_M((a_i)_{i \in I}) = (b_i)_{i \in I}$
where
\[
b_i = \sum_{j \in I,\,a_j\ne e_j} r_{ij}(a_j)\quad \mbox{with the sum computed in
$\Grp a(\theta,e_i)$}.
\]
One can show that each $\sfL_M \in \End(V)$, the map $M\mapsto \sfL_M$ is injective,
and its range $\sfL_{\calR}:=\set{\sfL_M}{$M \in \calR$}$ is a unital subring of $\bEnd(V)$.
Via the identification of $\calR$ with $\sfL_\calR$, 
$\calR$ inherits the operations of a unital ring, and we obtain the faithful
left $\calR$-module 
${}_\calR V$.
The minimality of $\theta$ implies that ${}_\calR V$ is a simple module.
Freese \cite[p.\ 150]{freese1983} explained this in the congruence modular
case; here is a proof that does not require congruence modularity.
Let $\zerobar = (e_i)_{i \in I} \in V$.
The claim of simplicity reduces to proving the following: 
if $\tup a = (a_i)_{i \in I} \in V$ with $\tup a\ne \zerobar$, 
and $\tup b = (b_i)_{i \in I} \in V$ is such that there exists $k \in I$ with
$b_i=e_i$ for all $i \in I\setminus\{k\}$, then $\tup b \in \calR\tup a$.
To prove it,
pick $\ell \in I$ with $a_\ell\ne e_\ell$.  By minimality of $\theta$ we have
$(b_k,e_k) \in \cg(a_\ell,e_\ell)$.  Thus by Corollary~\ref{cor:maltsev}, there
exists $f \in \Pol_1(\m a)$ with $f(a_\ell)=b_k$ and $f(e_\ell)=e_k$.
Let $M = (r_{ij})_{i,j \in I} \in \calR$ be the matrix given by 
$r_{k\ell}=f\restrict{\thcl_\ell}$ and $r_{ij}=O_{ij}$ for $(i,j)\ne (k,\ell)$.
Then $M\tup a = \sfL_M(\tup a)=\tup b$ as required.
It follows by Schur's Lemma
for modules that the endomorphism ring 
$\bEnd({}_{\calR}V)$ of the module ${}_\calR V$ is a division ring $\bbD$.
This is the division ring Freese associated to $\theta$ and the transversal
$T$ (in the congruence modular case).

\begin{appendixoneprp} \label{prp:new=F}
Suppose $\m a$ is an algebra in a variety with a weak difference term and
$\theta$ is an abelian minimal congruence of $\m a$.  Let $\bbD$ be the
division ring associated to $\theta$ and some transversal $T$
in the previous two paragraphs.  Then $\bbD \cong \bbF_\theta$.
\end{appendixoneprp}

\stepcounter{appendixonecounter}

\begin{proof}[Proof sketch]
Let $\alpha=(0:\theta)$ and $\Delta=\Delta_{\theta,\alpha}$.  Recall that
the universe of
$\bbF_\theta$ is the set of all endomorphisms $\lambda$ of the difference
algebra $D(\m a,\theta) = \m a(\theta)/\Delta$ which 
preserve $\varphi:=\baralpha/\Delta$ and fix the elements of 
$\Do = \set{0_\alcl}{$ \alcl \in A/\alpha$}$ pointwise.
Let $V$ be the abelian group defined in the paragraphs preceding
Proposition~\ref{prp:new=F}.  
Define the ring $\m p = \prod_{i \in I}\bEnd(\Grp a(\theta,e_i))$ and let
$\Psi:\m p\ra \bEnd(V)$ be the obvious (coordinatewise) ring embedding.
Given $\lambda \in \bbF_\theta$, define $\Phi(\lambda)=(\lambda_i)_{i \in I}$
where for each $i \in I$, $\lambda_i:\thcl_i\ra \thcl_i$ is the map given
as follows: given $a \in \thcl_i$, $\lambda_i(a)$ is the unique $b \in \thcl_i$
such that $\lambda((a,e_i)/\Delta) = (b,e_i)/\Delta$
(see Corollary~\ref{cor:ran-sub}).  
It is tedious but straightforward to check that $\Phi$ 
is an injective ring homomorphism from $\bbF_\theta$ into $\m p$.

It remains only to show that the range of 
$\Psi\circ\Phi$ is precisely $\End({}_{\calR}V)$, 
for then $\Psi\circ\Phi$ will be an isomorphism from $\bbF_\theta$ to $\bbD$.
We can establish this by proving the following:
\begin{enumerate}
\item
$\End({}_\calR V)\subseteq \Psi(P)$.
\item
For all $(\lambda_i)_{i \in I} \in P$, the following are equivalent:
\begin{enumerate}
\item \label{newF:it1}
$\Psi((\lambda_i)_{i \in I}) \in \End({}_\calR V)$.
\item \label{newF:it2}
For all $i,j \in I$ and all $r \in R_{ij}$,
$\lambda_i\circ r = r\circ\lambda_j$.
\item \label{newF:it3}
$(\lambda_i)_{i\in I} \in \Phi(\bbF_\theta)$.
\end{enumerate}
\end{enumerate}
The first item can be deduced by 
applying the law $h\circ \sfL_M = \sfL_M\circ h$ for $h \in \End({}_\calR V)$
and $M \in \calR$, in particular for $M=(m_{ij})_{i,j\in I}$ 
where for some $k$ we have $m_{kk}=\mathrm{id}_{\thcl_k}$
and $m_{ij}=O_{ij}$ for all $(i,j)\ne (k,k)$.
A similar argument establishes the equivalence of 
\eqref{newF:it1} and \eqref{newF:it2}.

It remains to prove \eqref{newF:it2} $\Leftrightarrow$ \eqref{newF:it3}.
First we record a helpful observation.

\medskip\noindent\textsc{Claim 1.}
For all $i,j \in I$ and $r \in R_{ij}$, there exists a term
$t(x,y_1,\ldots,y_n)$ and $\alpha$-classes $\alcl_1,\ldots,\alcl_n$ such
that for all $a \in \thcl_j$,
\[
(r(a),e_i)/\Delta = t^{D(\m a,\theta)}((a,e_j)/\Delta,0_{\alcl_1},\ldots,
0_{\alcl_n}).
\]

To prove Claim 1, choose 
$f \in \Pol_1(\m a)$ with $f(e_j)=
e_i$ and $r=f\restrict{\thcl_j}$, and pick a term $t(x,y_1,\ldots,y_n)$
and $\tup u \in A^n$ so that $f(x)=t(x,\tup u)$.  
Then for each $\ell\in [n]$ let $\alcl_\ell = u_\ell/\alpha$.
Then for any $a \in \thcl_j$,
\begin{align*}
(r(a),e_i)/\Delta &= (t(a,\tup u),t(e_j,\tup u))/\Delta\\
&= t^{D(\m a,\theta)}((a,e_j)/\Delta,(u_1,u_1)/\Delta,\ldots,
(u_n,u_n)/\Delta)\\
&= t^{D(\m a,\theta)}((a,e_j)/\Delta,0_{\alcl_1},\ldots,0_{\alcl_n}),
\end{align*}
which proves Claim 1.

\medskip
Now
assume that $(\lambda_i)_{i \in I} = \Phi(\lambda)$ with $\lambda \in 
\bbF_\theta$.
Fix $i,j \in I$ and $r \in R_{ij}$ and let $a \in \thcl_j$.
We must show $\lambda_i(r(a)) = r(\lambda_j(a))$.  
Choose $t(x,\tup y)$ and $\alcl_1,\ldots,\alcl_n\in A/\alpha$ for $i,j,r$
as in Claim 1.
By definition,
$\lambda_i(r(a))$ is the unique $b \in \thcl_i$ such that
$\lambda((r(a),e_i)/\Delta) = (b,e_i)/\Delta$.  Well,
\begin{align*}
\lambda((r(a),e_i)/\Delta) &= \lambda(
t^{D(\m a,\theta)}((a,e_j)/\Delta,0_{\alcl_1},\ldots,0_{\alcl_n}))
\\
&= t^{D(\m a,\theta)}(\lambda((a,e_j)/\Delta),0_{\alcl_1},\ldots,
0_{\alcl_n}) &\mbox{as $\lambda \in \bbF_\theta$}\\
&= t^{D(\m a,\theta)}((\lambda_j(a),e_j)/\Delta,0_{\alcl_1},\ldots,0_{\alcl_n})\\
&= (r(\lambda_j(a)),e_i)/\Delta,
\end{align*}
which implies $\lambda_i(r(a))=r(\lambda_j(a))$, proving \eqref{newF:it3}
$\Rightarrow$ \eqref{newF:it2}.

For the opposite implication, assume that $(\lambda_i)_{i \in I}$ satisfies
the condition of \eqref{newF:it2}.
We first claim that for all $i,j \in I$, all $a \in \thcl_i$, and all
$b \in \thcl_j$,
\begin{equation} \label{eq:prp:newF}
(a,e_i)\stackrel\Delta\equiv (b,e_j) ~~\mbox{implies}~~
(\lambda_i(a),e_i) \stackrel\Delta\equiv (\lambda_j(b),e_j).
\end{equation}
Indeed, assume $(a,e_i)\stackrel\Delta\equiv (b,e_j)$.
Choose a $\theta$-class $\thcl_k$ so that $\thcl_i,\thcl_j\ra \thcl_k$
(see Lemma~\ref{lm:arrow}).  Let $f_i,f_j \in \Pol_1(\m a)$ be witnesses to
$\thcl_i\ra\thcl_k$ and $\thcl_j\ra\thcl_k$ respectively.  
By Lemma~\ref{lm:arrow}, we can assume $f_i,f_j$ are chosen so that 
$f_\ell(e_\ell)=e_k$ for $\ell=i,j$.  
Then
$(f_i(a),e_k) \stackrel\Delta\equiv (a,e_i)\stackrel\Delta\equiv
(b,e_j) \stackrel\Delta\equiv (f_j(b),e_k)$ by Lemma~\ref{lm:arrow}\eqref{arrow:it1.5} and the assumption, which implies
$f_i(a)=f_j(b)$ by Lemma~\ref{lm:inject}\eqref{inject:it1}.  
Let $r=f_i\restrict{\theta_i}$ and $s=f_j\restrict{\theta_j}$ and note that
$r \in R_{ki}$ and $s \in R_{kj}$.
Condition \eqref{newF:it2} then gives
$\lambda_k\circ r = r\circ \lambda_i$ and $\lambda_k\circ s=s\circ \lambda_j$.
Then
\begin{align*}
(\lambda_i(a),e_i) \stackrel\Delta\equiv (r(\lambda_i(a)),e_k) 
&= (\lambda_k(r(a)),e_k)\\  
&=(\lambda_k(s(b)),e_k) =
(s(\lambda_j(b)),e_k) \stackrel\Delta\equiv
(\lambda_j(b),e_j)
\end{align*}
which proves \eqref{eq:prp:newF}.

Now let $\lambda:D(\m a,\theta)\ra D(\m a,\theta)$ be the unique map satisfying
\[
\lambda((x,e_i)/\Delta) = (\lambda_i(x),e_i)/\Delta
\]
for all $i \in I$ and all
$x \in \thcl_i$.  ($\lambda$ is well-defined by \eqref{eq:prp:newF}.)
We will show that $\lambda \in \bbF_\theta$, which will complete the
proof of Proposition~\ref{prp:new=F}.  Clearly 
$\lambda(\varphi)\subseteq \varphi$ and $\lambda(0_\alcl)=0_\alcl$ for all 
$\alcl\in A/\alpha$.  It remains to show that 
$\lambda \in \End(D(\m a,\theta))$.
The next Claim is an immediate consequence of Proposition~\ref{prp:affine}.

\medskip\noindent\textsc{Claim 2.}
For all $n\geq 1$, all $f \in \Pol_n(\m a)$, and all 
$\theta$-classes $\thcl_{i_1},\ldots, \thcl_{i_n}$, if $\thcl_\ell$ is
the $\theta$-class containing $f(e_{i_1},\ldots,e_{i_n})$, then 
there exist $r_t \in R_{\ell i_t}$ for $t\in [n]$ such that the following holds:
for all
$(a_1,\ldots,a_n) \in \thcl_{i_1}\times \cdots \times \thcl_{i_n}$,
\[
f(\tup a) = \sum_{t=1}^n r_t(a_t) + f(\tup e) \quad
\mbox{computed in $\Grp a(\theta,e_\ell)$.}
\]

Now let $F$ be an $n$-ary basic operation.  Choose $n$ elements of 
$D(\m a,\theta)$, which may be written $(a_1,e_{i_1})/\Delta,\ldots,
(a_n,e_{i_n})/\Delta$
for some $\theta$-classes $\theta_{i_t}$ and 
$a_t \in \thcl_{i_t}$.
We must show
\begin{equation}\label{eq:whattoshow}
\lambda(F((a_1,e_{i_1})/\Delta,\ldots,(a_n,e_{i_n})/\Delta))
= F(\lambda((a_1,e_{i_1})/\Delta),\ldots,\lambda((a_n,e_{i_n})/\Delta)).
\end{equation}
Let $\thcl_\ell$ be the $\theta$-class containing $F(\tup a)$ and choose $r_t
\in R_{\ell i_t}$ for $t \in [n]$ as in Claim 2 for $F$ and $C_{i_1},\ldots,
C_{i_n}$.
We have $(F(\tup a),F(\tup e)) \stackrel\Delta\equiv (d(F(\tup a),F(\tup e),
e_\ell),e_\ell)$ by Lemma~\ref{lm:samerange}\eqref{samerange:it1}.  
Thus the left-hand
side of \eqref{eq:whattoshow} can be written
\begin{align*}
\lambda(F((a_1,e_{i_1})/\Delta,\ldots,(a_n,e_{i_n})/\Delta)) &=
\lambda((F(\tup a),F(\tup e))/\Delta)\\
& = \lambda((d(F(\tup a), F(\tup e),e_\ell))/\Delta)\\
& = (\lambda_\ell(d(F(\tup a),F(\tup e),e_\ell)),e_\ell)/\Delta.
\end{align*}
The right-hand side of \eqref{eq:whattoshow} can be written
\begin{align*}
F(\lambda((a_1,e_{i_1})/\Delta),\ldots,\lambda((a_n,e_{i_n})/\Delta))
&= F(
(\lambda_{i_1}(a_1),e_{i_1})/\Delta,\ldots,
(\lambda_{i_n}(a_n),e_{i_n})/\Delta)\\
&= (F(\lambda_{i_1}(a_1),\ldots,\lambda_{i_n}(a_n)),F(\tup e))/\Delta\\
&= (d(F(\lambda_{i-1}(a_1),\ldots,\lambda_{i_n}(a_n)),F(\tup e),e_\ell),e_\ell)
/\Delta.
\end{align*}
Thus to prove \eqref{eq:whattoshow}, it will suffice to show
\begin{equation} \label{eq:whattoshow2}
\lambda_\ell(d(F(\tup a),F(\tup e),e_\ell)) = 
d(F(\lambda_{i_1}(a_1),\ldots,
\lambda_{i_n}(a_n)),F(\tup e),e_\ell).
\end{equation}

By Claim 2, there exist $r_t \in R_{\ell i_t}$ for $t \in [n]$ such that
for all $(x_1,\ldots,x_n)\in \thcl_{i_1}\times \cdots \times \thcl_{i_n}$,
$F(\tup x) = \sum_{t=1}^n r_t(x_t) + F(\tup e)$.
Recall that each $\lambda_j$ and $r_t$
is a group 
homomorphism between the appropriate groups from
$\Grp a(\theta,e_{i_1}),\ldots,\Grp a(\theta,e_{i_n}),\Grp a(\theta,e_\ell)$.
Thus computing in $\Grp a(\theta,e_\ell)$,
\begin{align*}
\lambda_\ell(d(F(\tup a),F(\tup e),e_\ell)) &= \lambda_\ell(F(\tup d)) - 
\lambda_\ell(F(\tup e))\\
&= \lambda_\ell\left(\sum_{t=1}^n r_t(a_t) + F(\tup e)\right) - \lambda_\ell(F(\tup e))\\
&= \sum_{t=1}^n \lambda_\ell(r_t(a_t)) &(\dagger)
\end{align*}
while
\begin{align*}
d(F(\lambda_{i_1}(a_1),\ldots,\lambda_{i_n}(a_n)),F(\tup e),e_\ell) &=
\left(\sum_{t=1}^n r_t(\lambda_{i_t}(a_t)) + F(\tup e)\right) - F(\tup e)\\
&= \sum_{t=1}^n r_t(\lambda_{i_t}(a_t). & (\ddagger)
\end{align*}
Since $\lambda_\ell\circ r_t=r_t\circ \lambda_{i_t}$ for each $t \in [n]$ by
assumption \eqref{newF:it2}, we have established \eqref{eq:whattoshow2} and
so have finished the proof of Proposition~\ref{prp:new=F}.
\end{proof}

Finally, we show that $\bbF_\theta$ coincides with the finite field $\bbF$ 
constructed in the proof of Theorem~\ref{thm:finite-vecspace} when $\m a$ is 
finite.

\begin{appendixoneprp} \label{prp:new=old}
Suppose $\calV$ is a variety with a weak difference term $d$, $\m a$ is a finite
algebra in $\calV$,
and $\theta$ is an abelian minimal congruence of $\m a$.  
Let $\bbF$ be the finite field associated to $\theta$ in the proof 
of Theorem~\ref{thm:finite-vecspace}.  That is, let $\thcl_1,\ldots,
\thcl_k$ be a list of the $\theta$-classes, 
let $\barthcl:=\thcl_1\times \cdots \times\thcl_k$,
choose $\tup e=(e_1,\ldots,e_k)
\in \barthcl$, and define $+$ on $\barthcl$ by $x+y:=d(x,\tup e,y)$
so that $\m c:=(\barthcl,+,\tup e)$ is an abelian group.
Let $\mathfrak{C}(\m a,\nu_\theta)$ be defined as in the
proof of Theorem~\ref{thm:finite-vecspace}, let $\m r$ be the subring
of
$\End(\m c)$ such that $\mathfrak{C}(\m a,\nu_\theta)$ is polynomially
equivalent to  $\RGrp\m c$, and let 
$\bbF=\End(\RGrp\m c)$.  Then $\bbF\cong \bbF_\theta$.
\end{appendixoneprp}

\begin{proof}[Proof sketch]
First recall the construction of Freese's division ring $\bbD$ associated
to $\theta$ and the transversal $T:= \{e_1,\ldots,e_k\}$
described at the beginning of this section.  By finiteness,
the abelian group $V$ in Freese's construction coincides with the group $\m c$.
Also recall the unital subring $\sfL_{\calR}$ of $\bEnd(\m c)$ 
from the construction
of $\bbD$.  If we can prove $\sfL_{\calR}=\m r$, then it will follow
that $\bbD=\bbF$ and therefore $\bbF\cong \bbF_\theta$ 
by Proposition~\ref{prp:new=F}.

Following \cite{gumm-maltsev}, the universe of $\m r$ may be given by
\[
R = \set{f(x)\in \Pol_1(\mathfrak{C}(\m a,\nu_\theta))}{$f(\tup e)=\tup e$}.
\] 
By \cite[Remark 2.8]{mayr-szendrei}, the unary polynomials of 
$\mathfrak{C}(\m a,\nu_\theta)$ are the functions $f:C\ra C$ which can
be described as follows: for some $k$-ary polynomials
$f_1,\ldots,f_k \in \Pol_k(\m a)$ with $f_i(\tup e) \in \thcl_i$ for $i\in [k]$,
\[
f(\tup c) = (f_1(\tup c),\ldots,f_k(\tup c))\quad\mbox{for all $\tup c \in C$}.
\]
Assuming $f \in R$, i.e., $f(\tup e)=\tup e$, we get
$f_i(\tup e)=e_i$ for each $i\in [k]$.  Thus
for each $i \in [k]$ there exist $r_1,\ldots,r_k \in R_{ij}$ (see Freese's 
construction) so that for all $\tup c=(c_1,\ldots,c_k) \in C$,
\[
f_i(\tup c) = \sum_{j=1}^k r_{ij}(c_j) \quad\mbox{computed in
$\Grp a(\theta,e_i)$},
\]
by Proposition~\ref{prp:affine}.  That is, $f = \sfL_M$ where
$M=(r_{ij})_{i,j\in [k]} \in \calR$ (assuming $f \in R$).  The logic
is reversible and we have $R=\sfL_{\calR}$ as desired.
\end{proof}

\phantomsection
\hypertarget{A2}{}

\section*{Appendix 2: Examples}

In this appendix we illustrate some of our results from Sections~\ref{sec:structure} and~\ref{sec:divring} by presenting
a recipe for constructing certain finite subdirectly irreducible algebras with
abelian monoliths in varieties having weak difference terms.  The algebras
$\m a$ can be constructed 
so that if $\mu$ is the monolith, $\alpha=(0:\mu)$, and
$\Dmon=\baralpha/\Delta_{\mu,\alpha}$,
then
\begin{enumerate}
\item
$\bbF_\mu$ can be any desired finite field $\bbF$;
\item
$|A/\alpha|$ can be any positive integer.
\item
The vector spaces $\FGrpd{D(\m a)}(\varphi,0_\alcl)$ as $\alcl$ ranges over
the $\alpha$-classes can have any desired (finite) dimensions.
\item
The ranges of the 
$\mu$-classes within a given $\alpha$-class $\alcl$ can be any 
collection of subspaces of 
$\FGrpd{D(\m a)}(\Dmon,0_\alcl)$, as long as the collection includes
$0_\alcl/\Dmon$ itself.
\end{enumerate}

Our recipe is a variant of 
a known construction \cite{bulatov-SMB}.
We first give a useful method for constructing varieties with a prescribed
weak difference term.
For this, we need the following result characterizing weak 
difference terms at the level of varieties.
Following \cite{kearnes-szendrei}, given a variety $\calV$, let $\m f=\m f_\calV(3)$
be the $\calV$-free algebra freely generated by $\{x,y,z\}$, let
$\alpha = \cg(x,z)$, $\beta=\cg(x,y)$ and $\gamma = \cg(y,z)$ calculated in $\m f$,
and define $\beta_0=\beta$, $\gamma_0=\gamma$, $\beta_{n+1} = \beta\vee(\alpha\wedge
\gamma_n)$, and $\gamma_{n+1}=\gamma\vee(\alpha\wedge \beta_n)$ for $n\geq 0$.

\begin{appendixtwothm}
\label{thm:maltsev-cond}
There exists an infinite sequence $\Sigma_0,\Sigma_1,\Sigma_2,\ldots$ of finite sets of
identities involving variables $x,y$ and ternary
operation symbols $\sfd,\sff_1,\sff_2,\ldots$, such that for every variety $\calV$
and every ternary term $d(x,y,z)$ in the signature of $\calV$,
the following are equivalent:
\begin{enumerate}
\item \label{maltsev-cond:it1}
$d$ is a weak difference term for $\calV$.
\item \label{maltsev-cond:it2}
For some $n\geq 0$ 
there exist idempotent ternary terms $f_1,f_2,\ldots$ in the signature of $\calV$ so 
that, when the symbols $\sfd,\sff_1,\sff_2,\ldots$ are interpreted by 
$d,f_1,f_2,\ldots$ respectively, the identities in 
$\Sigma_n$ are true in $\calV$.
\item \label{maltsev-cond:it3}
Letting $\m f=\m f_\calV(3)$ and interpreting $d$ as an element of $\m f$ in the
usual way, i.e.\ as $d^{\m f}(x,y,z)$, we have $(x,d) \in \gamma_m$ and $(d,z) \in \beta_m$ for some $m\geq 0$.
\end{enumerate}
\end{appendixtwothm}

\stepcounter{appendixtwocounter}

\begin{proof}[Proof sketch]
\eqref{maltsev-cond:it1}\Implies\eqref{maltsev-cond:it3} can be extracted from the proof
of Theorem 4.8 (1)\Implies(2)\Implies(3) in \cite{kearnes-szendrei}.
Standard arguments with roots in \cite{pixley,wille}
can produce sets $\Sigma_0,\Sigma_1,\Sigma_2,\ldots$ of identities witnessing 
\eqref{maltsev-cond:it2}\Iff\eqref{maltsev-cond:it3}.
To prove \eqref{maltsev-cond:it3}\Implies\eqref{maltsev-cond:it1}, 
assume that $(x,d)\in \gamma_m$ and $(d,z)\in \beta_m$.  Then by \cite[Theorem 2.1 
(3)\Implies(4)]{wires}, there exist idempotent ternary terms $f_1,\ldots,f_n,g_1,\ldots,g_n$
(for some $n$) such that $\calV\models f_i(x,y,x)\approx g_i(x,y,x)$ for each
$i \in [n]$, and for all $\m a\in \calV$ and all $a,b \in A$, if
\[
\bigwedge_{i=1}^n \left(\rule{0in}{.2in} f_i(a,a,b)=g_i(a,a,b)\leftrightarrow f_i(a,b,b)=g_i(a,b,b)\right),
\]
then $d(a,a,b)=b=d(b,a,a)$ and $d(b,b,a)=a=d(a,b,b)$.  Now one can easily adapt the proof
of \cite[Theorem 1.2 (3)\Implies(1)]{diffterm} to get that $d$ is a weak difference term 
for $\calV$.
\end{proof}

\begin{appendixtwodf} \label{df:semiover}
A \emph{semilattice-over-Maltsev} 
operation is any operation $d:A^3\ra A$
constructed in the following way.
Let $(S,\wedge)$ be a (meet) semilattice.  For each $s \in S$ let $V_s$ be a \nonempty\ set,
and assume that $V_s\cap V_t=\varnothing$ for all $s,t \in S$ with $s\ne t$.
For each pair $s,t \in S$ with $s\geq t$, assign a function $f_{(s,t)}:V_s\ra V_t$,
so that $f_{(s,s)}=\mathrm{id}_{V_s}$.
Finally, for each $s \in S$ let $m_s:(V_s)^3\ra V_s$ be a Maltsev operation.

From this data, define $A = \bigcup_{s \in S}V_s$ 
and $\nu:A\ra S$ so that $a \in V_{\nu(a)}$ for all $a \in A$.
Then define $d:A^3\ra A$ by
\[
d(a,b,c) = m_t(f_{(\nu(a),t)}(a),f_{(\nu(b),t)}(b),f_{(\nu(c),t)}(c))\quad\mbox{where
$t=\nu(a)\wedge \nu(b) \wedge \nu(c)$}.
\]
\end{appendixtwodf}
\stepcounter{appendixtwocounter}

\begin{appendixtwolm} \label{lm:variety}
Suppose $\m a$ is an algebra
and $d$ is a ternary term operation of $\m a$
which is semilattice-over-Maltsev.
Then $d$ is a weak difference term for $\HSP(\m a)$.
\end{appendixtwolm}
\stepcounter{appendixtwocounter}

\begin{proof}
Clearly $d$ is idempotent, and it can be easily checked that $\m a$ satisfies the
following identities:
\begin{align}
d(d(x,y,y),x,x) &\approx d(x,y,y) \approx d(d(x,y,y),y,y) \label{variety:eq}\\
d(x,x,d(x,x,y)) &\approx d(x,x,y) \approx d(y,y,d(x,x,y)). \notag
\end{align}
Define ternary terms $f_i,g_i$ for $i=0,1,2,3$ by
\begin{align*}
f_0(x,y,z) &= x   & g_0(x,y,z) &= d(x,z,z)\\
f_1(x,y,z) &= d(x,y,y) & g_1(x,y,z) &= d(d(x,y,y),z,z)\\
f_2(x,y,z) &= z   & g_2(x,y,z) &= d(x,x,z)\\
f_3(x,y,z) &= d(y,y,z) & g_3(x,y,z) &= d(x,x,d(y,y,z)).
\end{align*}
From idempotency of $d$ and the identities~\eqref{variety:eq}, one can deduce that
$\m a$ satisfies the identities
\begin{equation} \label{variety:eq3}
f_i(x,y,x) \approx g_i(x,y,x)\quad\mbox{for all $i$}
\end{equation}
and
\begin{align} 
x &\approx f_0(x,y,y) & f_2(x,x,y) &\approx y \label{variety:eq2}\\
f_1(x,x,y) &\approx f_0(x,x,y) & f_2(x,y,y) &\approx f_3(x,y,y)\notag\\
f_1(x,y,y) &\approx g_1(x,y,y) & g_3(x,x,y) &\approx f_3(x,x,y) \notag\\
g_0(x,x,y) &\approx g_1(x,x,y) & g_3(x,y,y) &\approx g_2(x,y,y)\notag \\
g_0(x,y,y) &\approx d(x,y,y) & d(x,x,y) &\approx g_2(x,x,y). \notag
\end{align}
The identities \eqref{variety:eq3} and \eqref{variety:eq2}
therefore hold in $\HSP(\m a)$.
Let $\m f=\m f_{\HSP(\m a)}(3)$ be the free algebra in $\HSP(\m a)$ freely generated
by $\{x,y,z\}$.
As in the discussion preceding Theorem~\ref{thm:maltsev-cond},
let $\alpha=\cg(x,z)$, $\beta=\cg(x,y)$ and $\gamma = \cg(y,z)$
be calculated in $\m f$.  Interpreting $d$ and $f_0,\ldots,f_3,g_0,\ldots,g_3$ as 
elements of $\m f$ in the usual way, we find that the identities \eqref{variety:eq3}
and \eqref{variety:eq2} insure that $f_0,\ldots,f_3,g_0,\ldots,g_3$ witness
\begin{align*}
(x,d) &\in \gamma \circ(\alpha\cap (\beta \circ (\alpha\cap \gamma)\circ\beta)) \circ
\gamma 
\subseteq \gamma \vee (\alpha \wedge (\beta \vee (\alpha\wedge \gamma))) = \gamma_2\\
(d,z) &\in \beta \circ(\alpha\cap (\gamma \circ (\alpha\cap \beta)\circ\gamma)) \circ
\beta
\subseteq \beta \vee (\alpha \wedge (\gamma \vee (\alpha\wedge \beta))) = \beta_2.
\end{align*}
Thus $d$ is a weak difference term for $\HSP(\m a)$ by Theorem~\ref{thm:maltsev-cond}.
\end{proof}

Here is our recipe for constructing interesting finite subdirectly irreducible algebras
with abelian monoliths.

\begin{appendixtwoexmp} \label{exmp1}
Fix a finite field $\bbF$.  Let $\Vn 00,\Vn 01,\ldots,
\Vn 0m$ be finite-dimensional vector
spaces over $\bbF$ with $m \geq 0$ and
$\dim(\Vn 00) >0$, chosen so that
$\Vn 00,\ldots,\Vn 0m$ are pairwise disjoint.
For each $\ell\leq m$ let $\Wn 0\ell,\Wn 1\ell,\ldots,\Wn {n_\ell}\ell$
be a list of (not necessarily distinct) subspaces of $\Vn 0\ell$ with 
$\Wn 0\ell = \Vn 0\ell$.

Choose vector spaces $\Vn i\ell$ over $\bbF$
and isomorphisms $\sigman i\ell:\Vn i\ell \cong \Wn i\ell$ for 
$\ell \leq m$ and $1\leq i\leq n_\ell$, so that
the spaces $\Vn i\ell$ ($0\leq \ell\leq m$, $0\leq i\leq n_\ell$) are pairwise disjoint.
Also set $\sigman 0\ell = \mathrm{id}_{\Vn 0\ell}$ for each $\ell\leq m$.  
Let $\Cn \ell$ be the (disjoint) union of $\Vn 0\ell,\Vn 1\ell,\ldots,\Vn {n_\ell}\ell$, and
let $A$ be the (disjoint) union of $\Cn 0,\ldots,\Cn m$.
Then let $\mu$ and $\alpha$ be the equivalence relations on $A$ whose equivalence classes
are the spaces $\Vn i\ell$ ($\ell\leq m$, $i\leq n_\ell$) and the sets
$\Cn 0,\ldots,\Cn m$ respectively.  
Also let $B=\Vn 00\cup\cdots\cup\Vn 0m$ and let 
$\psi$ be the equivalence relation on $B$ whose equivalence classes are
$\Vn 00,\ldots,\Vn 0m$.
We will build  an algebra $\m a$ with universe $A$ such that:
\begin{itemize}
\item
$\m a$ belongs to a variety with a weak difference term.
\item
$\m a$ is subdirectly irreducible with abelian monolith $\mu$.  In fact,
$\mu$ will be the unique minimal reflexive subuniverse of $\m a^2$ properly
containing $0_A$.
\item
$\alpha = (0:\mu)$.
\item
Letting $\Delta = \Delta_{\mu,\alpha}$ and $\Dmon = \baralpha/\Delta$,
there is a bijection from $D(\m a,\mu)$ to $B$ 
which sends
$\Dmon$ to $\psi$ and sends each 
$\Dmon$-class $(\mu\cap(\Cn \ell)^2)/\Delta$ to $\Vn 0\ell$.
\item
Modulo this bijection, for each $\alpha$-class $\Cn \ell$, the ranges 
(see Definition~\ref{df:Brange})
of
the $\mu$-classes $\Vn i\ell$ contained
in $\Cn \ell$ correspond precisely to the subspaces $\Wn i\ell$ in $\Vn 0\ell$.
\item
The division ring $\bbF_\mu$ associated to $\mu$ is $\bbF$ (up to isomorphism).
\end{itemize}

Begin by letting $S = \set{(\ell,i)}{$\ell\leq m$ and $i\leq n_\ell$}$.
Thus $S$ is the set of indices for the vector spaces $\Vn i\ell$ which constitute $A$.
Define a binary operation $\wedge$ on $S$ by
\[
(\ell,i) \wedge (k,j) = \left\{\begin{array}{cl}
(\ell,i) & \mbox{if $(\ell,i)=(k,j)$}\\
(\ell,0) & \mbox{if $k=\ell$ but $i\ne j$}\\
(0,0) & \mbox{if $k\ne \ell$.}
\end{array}\right.
\]
Then $(S,\wedge)$ is the meet semilattice with least element $(0,0)$
pictured in Figure~\ref{fig:semilattice}.

\begin{figure} 
\begin{tikzpicture}

\draw [fill] (0,0) circle (2.0pt);
\draw [fill] (-1.5,1) circle (2.0pt);
\draw [fill] (0,1) circle (2.0pt);
\node at (-.75,1) {$\dots$};
\draw [fill] (2,1) circle (2.0pt);
\draw [fill] (6,1) circle (2.0pt);
\node at (4,1) {$\dots$};
\node at (2,2) {$\dots$};
\node at (6,2) {$\dots$};
\draw [fill] (1.25,2) circle (2.0pt);
\draw [fill] (2.75,2) circle (2.0pt);
\draw [fill] (5.25,2) circle (2.0pt);
\draw [fill] (6.75,2) circle (2.0pt);

\draw [-] (-1.5,1) -- (0,0) -- (0,1);
\draw [-] (2,1) -- (0,0) -- (6,1);
\draw [-] (1.25,2) -- (2,1) -- (2.75,2);
\draw [-] (5.25,2) -- (6,1) -- (6.75,2);

\node [anchor=north] at (0,0) {{\scriptsize $(0,0)$}};
\node [anchor=north] at (6,1) {{\scriptsize $(m,0)$}};
\node [anchor=north] at (2.2,1) {{\scriptsize $(1,0)$}};
\node [anchor=south] at (-1.5,1) {{\scriptsize $(0,1)$}};
\node [anchor=south] at (0,1) {{\scriptsize $(0,n_0)$}};

\node [anchor=south] at (1.25,2) {{\scriptsize $(1,1)$}};
\node [anchor=south] at (2.75,2) {{\scriptsize $(1,n_1)$}};
\node [anchor=south] at (5.25,2) {{\scriptsize $(m,1)$}};
\node [anchor=south] at (6.75,2) {{\scriptsize $(m,n_m)$}};

\end{tikzpicture}
\caption{The semilattice $(S,\wedge)$} \label{fig:semilattice}
\end{figure}

Recall that for each $s=(\ell,i) \in S$ we have a corresponding vector space $V_s$
over $\bbF$.  
To each pair of elements $s,t\in S$ with $s\geq t$
in the semilattice ordering,
we associate an $\bbF$-linear map $f_{(s,t)}: V_s
\ra V_t$ as
follows.
\begin{enumerate}
\item
$f_{(s,s)} = \mathrm{id}_{V_s}$ for all $s \in S$.
\item
$f_{((\ell,i),(\ell,0))} = \sigman i\ell$ for $(\ell,i)\in S$ with
$i\ne 0$.
\item
$f_{((\ell,i),(0,0))}=$ the constant zero map, 
for $(\ell,i) \in S$ with $\ell\ne 0$.
\end{enumerate}

Now let $d:A^3\ra A$ be the semilattice-over-Maltsev operation 
(see Definition~\ref{df:semiover})
defined by the data
$(S,\wedge)$, $(V_s:s \in S)$, $(f_{(s,t)}:s,t \in S,\,s\geq t)$, and the family
of Maltsev operations $m_s(x,y,z)=x-y+z$ evaluated in each vector space $V_s$.
That is, 
for all $a \in V_{s_1}$, $b \in V_{s_2}$ and $c \in V_{s_3}$, letting $t=s_1\wedge s_2\wedge
s_3$ we have
\[
\mbox{$d(a,b,c) = 
f_{(s_1,t)}(a) -
f_{(s_2,t)}(b) +
f_{(s_3,t)}(c)$\quad calculated in $V_t$}.
\]
For each $\ell\leq m$ define
$\sigma_\ell:\Cn\ell\ra \Vn 0\ell$ 
so that $\sigma_\ell\restrict{\Vn i\ell} = 
\sigman i\ell$ for all $i\leq n_\ell$.
For each $s \in S$ let $\zeron s$ denote the zero element of $V_s$.
We define some unary operations on $A$.
\begin{enumerate}
\item
For each $(\ell,i)\in S$ and each
nonconstant $\bbF$-linear map $g:\Vn 0\ell\ra \Vn i\ell$,
define $\Fn i\ell g:A\ra A$ by
\[
\Fn i\ell g(a) = \left\{\begin{array}{cl}
(g\circ\sigma_\ell)(a) & \mbox{if $x \in \Cn \ell$}\\
\zeroiell 0k & \mbox{if $a \in \Cn k$ with $k\ne \ell$.}
\end{array}\right.
\]

\item
For each $\ell\in \{1,\ldots,m\}$ with $\dim(\Vn 0\ell)>0$,  choose nonconstant
$\bbF$-linear maps $g_\ell:\Vn 00\ra \Vn 0\ell$ and $h_\ell:\Vn 0\ell\ra \Vn 00$ and
define $G_\ell,G'_\ell:A\ra A$ by
\[
G_\ell(a) = 
\left\{\begin{array}{cl}
(g_\ell \circ\sigma_0)(a) & \mbox{if $a \in \Cn 0$}\\
\zeroiell 0k & \mbox{if $a \in \Cn k$ with $k\ne 0$}
\end{array}\right.
\]
and
\[
G_\ell'(a) = 
\left\{\begin{array}{cl}
(h_\ell \circ\sigma_\ell)(a) & \mbox{if $a \in \Cn \ell$}\\
\zeroiell 0k & \mbox{if $a \in \Cn k$ with $k\ne \ell$}
\end{array}\right.
\]
\end{enumerate}

We also need the following binary operations.
\begin{enumerate}
\setcounter{enumi}{2}
\item
Fix $1 \in \Vn 00$ with $1\ne \zeroiell 00$.
For every $s \in S$, define $H_s:A^2\ra A$ by
\[
H_s(a,b) = \left\{\begin{array}{cl}
1 & \mbox{if $a,b \in V_s$}\\
\zeroiell 00 & \mbox{otherwise.}\end{array}\right.
\]

\item
For each $\ell \in \{1,\ldots,m\}$, define $K_\ell:A^2\ra A$ by
\[
K_\ell(a,b) = \left\{\begin{array}{cl}
b & \mbox{if $a \in \Cn \ell$}\\
\zeroiell 0k & \mbox{if $a\not\in \Cn \ell$ and $b \in \Cn k$}.
\end{array}\right.
\]
\end{enumerate}

Let $\scrF$ be the set of all the operations $\Fn i\ell g$, let
$\scrG$ 
be the set of all the operations $G_\ell$ and $G_\ell'$,
let $\scrH = \set{H_s}{$s \in S$}$, and let 
$\scrK = \set{K_\ell}{$1\leq \ell\leq m$}$.
Finally, let
\[
\m a = (A,\{d\}\cup\scrF\cup \scrG\cup\scrH\cup\scrK).
\]
\begin{appendixtwoclm} \label{clm:examp}
\rule{.1in}{0in}
\begin{enumerate}
\item \label{ex-clm:it1}
$d$ is a weak difference term $\HSP(\m a)$.
\item \label{ex-clm:it2}
$\m a$ is subdirectly irreducible with abelian monolith $\mu$.
Moreover, $\mu$ is the unique minimal reflexive subuniverse of $\m a^2$ properly
containing $0_A$.
\item \label{ex-clm:it3}
$(0:\mu)=\alpha$.
\item \label{ex-clm:it4}
$\Delta_{\mu,\alpha}$ is the set of all pairs $((a,b),(a',b'))$ such that 
for some $\ell\leq m$ and some $i,j \leq n_\ell$,
$\{a,b\}\subseteq  \Vn i\ell$ while $\{a',b'\}\subseteq \Vn j\ell$ and
\[
\mbox{$\sigman i\ell(a)-\sigman i\ell(b) = \sigman j\ell(a')-\sigman j\ell(b')$ 
calculated in $\Vn 0\ell$.}
\]
\end{enumerate}
\end{appendixtwoclm}
\stepcounter{appendixtwocounter}

\begin{proof}
\eqref{ex-clm:it1} follows from Lemma~\ref{lm:variety}.  
It is easy to check that $\mu$ and $\alpha$ are congruences of $\m a$.  
Let $\theta \in \Con(\m a)$ satisfy $0\ne \theta \subseteq \mu$.
(We want to show $\theta=\mu$.)
We can choose $(\ell,i)\in S$ so that $\theta\cap (\Vn i\ell)^2$ contains a 
pair
$(c,c')$ with $c\ne c'$.  
Let $a=c-c'$ calculated in $\Vn i\ell$.  Then 
$a\ne \zeroiell i\ell$ and $(a,\zeroiell i\ell)\in \theta$ because of $d(x,y,z)$.
Now let $j\leq n_\ell$ and $b \in \Vn j\ell$ with $b\ne 0_{(\ell,j)}$.
Let $a' = \sigman i\ell(a)$.  Then $a' \ne \zeroiell 0\ell$.  
Hence there exists an $\bbF$-linear map
$g:\Vn 0\ell \ra V_{(n,j)}$ with $g(a')=b$.
Then $\Fn j\ell g(a)=b$ while $\Fn j\ell g(\zeroiell i\ell)=\zeroiell j\ell$,
proving $(b,\zeroiell j\ell)\in \theta$.  
As $j$ and $b$ were arbitrary, 
this proves $\mu\cap (\Cn \ell)^2\subseteq \theta$.

If $\ell>0$, we can use the operation $G_\ell'$ to show that $\theta \cap (\Vn 00)^2$
contains a nontrivial pair.  Then the argument in the
previous paragraph gives $\mu\cap (\Cn 0)^2 \subseteq \theta$.
Now let $k \in \{1,\ldots,m\}$ be such that 
some $\mu$-class $\Vn jk$ contained in $\Cn k$ has more than one element.
Then 
we can use the operation $G_k$ to show that $\theta \cap (\Vn 0k)^2$ contains
a nontrivial pair.  Again by the argument in the previous paragraph, we get $\mu\cap(\Cn k)^2
\subseteq \theta$.  This finishes the proof that $\theta=\mu$; hence 
$0_A\prec \mu$.

Next let $(c,c')$ be any nontrivial pair in $A^2$ with $(c,c')\not\in \mu$.
Say $c \in V_s$ and $c' \in V_t$ with $s\ne t$.
Then the 
polynomial $f(x) = H_s(x,c)$ sends
$(c,c')$ to $(1,\zeroiell 00)$, so
$\cg(c,c')\cap \mu\ne 0$. As we have shown
$0_A\prec \mu$, this proves $\mu$ is the monolith of $\m a$.
Combining the above argument with Corollary~\ref{cor:maltsev}, we also get
that for all $(a,b) \in \mu$ and all $(c,c') \in A^2$ with $c\ne c'$, there
exists $f \in \Pol_1(\m a)$ with $f(c)=a$ and $f(c')=b$.  This proves that
$\mu$ is the unique minimal reflexive subuniverse of $\m a^2$ properly
containing $0_A$.

Next we will show that $C(\alpha,\mu;0)$ holds.
Define $R$ to be the set of all matrices $\begin{pmatrix}a&a'\\b&b' \end{pmatrix}
\in A^{2\times 2}$ such that for some $\ell\leq m$ and some $i,j\leq n_\ell$,
\begin{itemize}
\item
$0\in \{i,j\}$ or $i=j$.
\item
$a,b \in \Vn i\ell$ and $a',b' \in \Vn j\ell$.
\item
$\sigma_\ell(a)-\sigma_\ell(b) = \sigma_\ell(a')-\sigma_\ell(b')$ calculated in $\Vn 0\ell$.
\end{itemize}
One can show that $R$ is a subuniverse of $\m a^{2\times 2}$ (exercise).  Clearly 
$X(\mu,\alpha)\subseteq R$ (see Definition~\ref{df:M}).
Conversely, suppose 
\[
\begin{pmatrix}a&a'\\b&b'\end{pmatrix} \in R\quad\mbox{with
$a,b \in \Vn i\ell$ and $a',b' \in \Vn j\ell$}.
\]
We can assume with no loss of generality that $j=0$ or $i=j$.  Then
\[
\begin{pmatrix} a&a'\\b&b'\end{pmatrix} = d\left(\begin{pmatrix}
a&a'\\a&a'\end{pmatrix},\begin{pmatrix}a&a\\a&a\end{pmatrix},
\begin{pmatrix} a&a\\b&b\end{pmatrix}\right),
\]
proving $R\subseteq \sg^{\m a^{2\times 2}}(X(\mu,\alpha))$.  These remarks prove $R=M(\mu,
\alpha)$.  Since every matrix in $R$ has the property that the first column is in
$0_A$
if and only if the second column is in $0_A$, we have proved that $C(\alpha,\mu;0)$ holds
and therefore $\mu$ is abelian and $(0:\mu)\geq \alpha$.  To prove $(0:\mu)=\alpha$,
suppose $\theta \in \Con(\m a)$ and $\theta>\alpha$.  Pick $(a,b) \in \theta\setminus
\alpha$, say with $a \in \Cn \ell$ and $b \in \Cn k$ with $k\ne \ell$.  Using the operation
$K_\ell$ and the pair $(\zeroiell 00,1) \in \mu\setminus 0_A$, we have
\[
K_\ell(b,\zeroiell 00) = \zeroiell 00 = K_\ell(b,1)
\]
but
\[
K_\ell(a,\zeroiell 00)
= \zeroiell 00 \ne 1 = K_\ell(a,1),
\]
proving that $C(\theta,\mu;0)$ fails for every $\theta>\alpha$.
Hence $(0:\mu)= \alpha$.

Recall from Lemma~\ref{lm:faces}\eqref{faces:it1} that $\Deltah(\mu,\alpha)$ is the
horizontal transitive closure of $M(\mu,\alpha)$.
Thus $\Delta_{\mu,\alpha}$ is the transitive closure of
\[
\left\{((a,b),(a',b')): \begin{pmatrix}a&a'\\b&b'\end{pmatrix} \in R\right\}.
\]
The description of $\Delta_{\mu,\alpha}$ in \eqref{ex-clm:it4} then follows.
\end{proof}

Recall that $B = \Vn 00 \cup \cdots \cup \Vn 0m$.
Let $\Delta = \Delta_{\mu,\alpha}$, so $D(\m a,\mu)=\m a(\mu)/\Delta$.
Observe that every element of $D(\m a,\mu)$ can be uniquely written in the form
$(a,\zeroiell 0\ell)/\Delta$ for some $\ell\leq m$ and $a \in \Vn 0\ell$.
Let $h:D(\m a,\mu)\ra B$ be the bijection sending $(a,\zeroiell 0\ell)/\Delta \mapsto a$.
Then by construction, for every $(\ell,i) \in S$ we have 
\begin{align*}
\ran(\Vn i\ell) &= \ran(\embed{\zeroiell i\ell}) &\mbox{Definition~\ref{df:Brange}}\\
&= \set{(a,\zeroiell i\ell)/\Delta}{$a \in \Vn i\ell$}\\
&= \set{(\sigman i\ell(a),\zeroiell 0\ell)/\Delta}{$a \in \Vn i\ell$} &
\mbox{Claim~\ref{clm:examp}\eqref{ex-clm:it4}}\\
&= \set{(a',\zeroiell 0\ell)/\Delta}{$a' \in \Wn i\ell$} &\mbox{as $\sigman i\ell(\Vn i\ell)=
\Wn i\ell$.}
\end{align*}
Hence $h(\ran(\Vn i\ell)) = \Wn i\ell$ as desired.  

Finally, let 
$T$ be the transversal for $\mu$ given by $T=\set{0_s}{$s \in S$}$,
and let $\bbD$ be the division ring associated to $\mu, T$ by Freese's
construction as described in Appendix~\hyperlink{A1}{1}.  We will
verify that $\bbF\cong \bbD$, which will imply $\bbF \cong \bbF_\mu$
by Proposition~\ref{prp:new=F}.  

As in the proof of Proposition~\ref{prp:new=F}, let $V$ be the group
$\prod_{s \in S}\Grp{a}(\mu,0_s)$,
let $\m p$ be the ring $\prod_{s \in S} \bEnd(\Grp a(\mu,0_s))$, let $\Psi:\m p\ra
\bEnd(V)$ be the obvious (coordinatewise) ring embedding, and recall that
$\bbD$ is the subring of $\bEnd(V)$ whose universe is
\[
\set{\Psi((\lambda_s)_{s \in S})}{$\lambda_s\circ f\restrict{V_t}=f\circ \lambda_t$ 
for all $f \in \Pol_1(\m a)$ with $f(0_t)=0_s$}.
\]

For each $\zeta \in \bbF$ and $s \in S$
let $\zeta_s:V_s\ra V_s$ be the 
``scalar multiplication by $\zeta$ in $V_s$" map and let $\Theta(\zeta)=
(\zeta_s)_{s \in S}$.  
Clearly $(\zeta_s)_{s \in S} \in P$.
Every $n$-ary basic operation of $\m a$, when 
restricted to an $n$-tuple of $\mu$-classes, is $\bbF$-affine, and this
property is inherited by polynomial operations.  It follows that
$(\Psi\circ\Theta)(\zeta) \in \bbD$.  It is not hard to check that
$\Psi\circ\Theta$ is a ring embedding of $\bbF$ into $\bbD$.

To prove $\ran(\Psi\circ\Theta)=\bbD$, fix $\lambda \in \bbD$ 
and choose $(\lambda_s)_{s \in S} \in P$ with $\Psi((\lambda_s)_{s \in S})=\lambda$.
Thus
\begin{equation}
\mbox{$\lambda_s\circ f\restrict{V_t} = f \circ \lambda_t$ 
for all $f \in \Pol_1(\m a)$ and $s, t \in S$ with $f(0_t)=0_s$.}
\tag{$\ast$} \label{eq:ex}
\end{equation}
Fix $s \in S$ with $\dim(V_s)>0$. 
Since every $\bbF$-linear map $V_s\ra V_s$
is a restriction to $V_s$ of some unary polynomial, 
\eqref{eq:ex} implies that $\lambda_{s}$ commutes with every $r \in
\End_\bbF(V_s)$.  Hence there exists
$\zeta^{s} \in \bbF$ such that $(\zeta^{s})_{s}=\lambda_{s}$.
We will show that $\zeta^s=\zeta^t$ for all $s,t \in S$
with $\dim(V_s),\dim(V_t)>0$, which will
imply $\lambda \in \ran(\Psi\circ\Theta)$.

Consider first the case where $s=(\ell,0)$ with $\ell\ne 0$ and $\dim(\Vn 0\ell)>0$, and $t=(0,0)$.
Recall the nonconstant $\bbF$-linear map 
$g_\ell:\Vn00 \ra \Vn0\ell$ used to define the basic operation $G_\ell$.
Then $g_\ell = G_\ell\restrict{\Vn00}$, so \eqref{eq:ex} gives
$\lambda_{s}\circ g_\ell = g_\ell \circ \lambda_{t}$.
Choose $a \in V_t$ with $b:=g_\ell(a)\ne 0_s$.  Then
\begin{align*}
\zeta^t \cdot b = \zeta^t\cdot g_\ell(a) &= g_\ell(\zeta^t\cdot a) &\mbox{$\bbF$-linearity of $g_\ell$}\\
&= g_\ell(\lambda_{t}(a))  \\
&= \lambda_{s}(g_\ell(a)) &\mbox{by \eqref{eq:ex}}\\
&= \zeta^s\cdot b.
\end{align*}
Since $b\ne 0_s$, we get $\zeta^s=\zeta^t$.
A similar argument using operations in $\scrF$ gives
$\zeta^{(\ell,i)}=\zeta^{(\ell,0)}$ whenever $\ell\leq m$ and 
$\dim(\Vn i\ell)>0$.
This completes the proof that $\bbF_\mu\cong \bbF$.
\end{appendixtwoexmp}
\stepcounter{appendixtwocounter}

\bibliography{similar-bib}
\end{document}